\documentclass{amsart}

\usepackage[utf8]{inputenc}
\usepackage{amssymb, amsmath, amsfonts, amsthm, graphics,enumerate,comment}
\usepackage{hyperref}
\usepackage{mathtools}
\usepackage{xspace}
\usepackage[all,cmtip]{xy}
\usepackage{tikz}
\usetikzlibrary{arrows}
\usetikzlibrary{arrows.meta}
\usetikzlibrary{positioning}
\usetikzlibrary{decorations.pathmorphing}
\DeclareFontFamily{U}{wncy}{}
    \DeclareFontShape{U}{wncy}{m}{n}{<->wncyr10}{}
    \DeclareSymbolFont{mcy}{U}{wncy}{m}{n}
    \DeclareMathSymbol{\Sha}{\mathord}{mcy}{"58} 
    \DeclareMathSymbol{\Che}{\mathord}{mcy}{"51} 
    \DeclareMathSymbol{\El}{\mathord}{mcy}{"4C} 
        \DeclareMathSymbol{\Er}{\mathord}{mcy}{"52} 
\usetikzlibrary{calc}

\newcommand{\mvdots}{\raisebox{0pt}[10 pt]{$\vdots$}}

\newenvironment{sbm}
    {\left[ \begin{smallmatrix}
    }
    { 
     \end{smallmatrix} \right]
    }

\newcommand{\newword}[1]{\textbf{\emph{#1}}}

\newcommand{\into}{\hookrightarrow}
\newcommand{\onto}{\twoheadrightarrow}

\renewcommand{\t}{{\mathord{\to}}}
\renewcommand{\int}{{\mathord{\hookrightarrow}}}
\newcommand{\ont}{{\mathord{\twoheadrightarrow}}}
\newcommand{\forces}{\rightsquigarrow}
\newcommand{\forcs}{\mathord{\rightsquigarrow}}

\newcommand{\covered}{\lessdot}
\newcommand{\covers}{\gtrdot}
\newcommand{\lleq}{\le\!\!\!\le}
\newcommand{\union}{\cup}

\newcommand{\op}{\mathrm{op}}

\DeclareMathOperator{\Hom}{Hom}

\DeclareMathOperator{\Fact}{Fact}
\DeclareMathOperator{\Mult}{Mult}

\DeclareMathOperator{\Pos}{Pos}
\DeclareMathOperator{\Cov}{Cov}
\DeclareMathOperator{\Del}{Del}

\newcommand{\meet}{\wedge}

\newcommand{\Meet}{\bigwedge}
\newcommand{\join}{\vee}
\newcommand{\bigjoin}{\bigvee}
\renewcommand{\Join}{\bigvee}

\newcommand{\C}{C}

\newcommand{\NN}{\mathbb{N}}

\newcommand{\RR}{\mathbb{R}}

\newcommand{\ZZ}{\mathbb{Z}}

\newcommand{\cA}{\mathcal{A}}

\newcommand{\cF}{\mathcal{F}}

\newcommand{\cT}{\mathcal{T}}

\newcommand{\Mod}{\mathrm{mod}}

\newcommand{\Pairs}{\operatorname{Pairs}}

\newcommand{\JIrr}{\operatorname{Ji}}
\newcommand{\MIrr}{\operatorname{Mi}}
\newcommand{\Con}{\operatorname{Con}}

\newcommand{\Concc}{\operatorname{Con}^{\operatorname{cc}}}
\newcommand{\con}{\operatorname{con}}
\newcommand{\concc}{\operatorname{con}^{\operatorname{cc}}}

\newcommand{\Downsets}{\operatorname{Downsets}}
\newcommand{\tors}{\operatorname{tors}}

\newcommand{\set}[1]{\lbrace#1\rbrace}

\newcommand{\downonto}{\operatorname{down}_{\raisebox{7pt}{\scalebox{0.8}{\rotatebox{-90}{$\onto$}}}\hspace{-1pt}}}
\newcommand{\downinto}{\operatorname{down}_{\raisebox{7pt}[7pt][2pt]{\scalebox{0.75}{\rotatebox{-90}{$\into$}}}\hspace{-0pt}}}
\newcommand{\uponto}{\operatorname{up}_{\raisebox{7pt}{\scalebox{0.8}{\rotatebox{-90}{$\onto$}}}\hspace{-1pt}}}
\newcommand{\upinto}{\operatorname{up}_{\raisebox{7pt}[7pt][2pt]{\scalebox{0.75}{\rotatebox{-90}{$\into$}}}\hspace{-0pt}}}

\renewcommand{\2}{\mathbf{2}}

\newcommand{\Bricks}{\operatorname{Bricks}}

\newcommand{\kappad}{\kappa^d}
\newcommand{\kappau}{\kappa} 

\newcommand{\mi}{meet-ir\-re\-duc\-ible\xspace}
\newcommand{\ji}{join-ir\-re\-duc\-ible\xspace}

\theoremstyle{plain}
\newtheorem{thm}{Theorem}[section]

\newtheorem{theorem}[thm]{Theorem}

\newtheorem{prop}[thm]{Proposition}
\newtheorem{proposition}[thm]{Proposition}

\newtheorem{lemma}[thm]{Lemma}

\newtheorem{cor}[thm]{Corollary}
\newtheorem{corollary}[thm]{Corollary}

\newtheorem*{theorem*}{Theorem}

\theoremstyle{definition}

\newtheorem{eg}[thm]{Example}

\theoremstyle{definition}
\newtheorem{remark}[thm]{Remark}

\usepackage{color}

\newcommand{\margincolor}{red}      
\definecolor{darkgreen}{rgb}{0,0.7,0}

\newcounter{margincounter}
\newcommand{\marginnum}{
\ifnum\value{margincounter}<10
\textcolor{\margincolor}{\begin{picture}(0,0)\put(2.2,2.4){\circle{9}}\end{picture}\footnotesize\arabic{margincounter}}
\else\ifnum\value{margincounter}<100
\textcolor{\margincolor}{\begin{picture}(0,0)\put(4.256,2.5){\circle{11}}\end{picture}\footnotesize\arabic{margincounter}}
\else
\textcolor{\margincolor}{\begin{picture}(0,0)\put(6.8,2.5){\circle{14}}\end{picture}\footnotesize\arabic{margincounter}}
\fi\fi
}

\makeatletter
\newcommand{\switchmargin}{
\if@reversemargin
\normalmarginpar
\else
\reversemarginpar
\fi
}
\makeatother

\newcommand{\br}[1]{[#1]}

\newcommand{\ignore}[1]{}

\title[FTFSDL]{The Fundamental Theorem of Finite Semidistributive Lattices}
\author{Nathan Reading}
\address{Department of Mathematics,
Box 8205, NC State University,
Raleigh, NC, United States 27695-8205}
\email{reading@math.ncsu.edu}

\author{David E Speyer}
\address{Department of Mathematics,
University of Michigan,
2074 East Hall,
530 Church Street,
Ann Arbor, MI, United States 48109-1043}
\email{speyer@umich.edu}

\author{Hugh Thomas}
\address{Département de mathématiques, UQAM,
C.P. 8888, Succursale Centre-ville,
PK-5151,
Montréal, Qc, Canada H3C 3P8}
\email{thomas.hugh\_r@uqam.ca}

\thanks{Nathan Reading was partially supported by the National Science Foundation under Grant Number DMS-1500949 and by the Simons Foundation under award number 581608.
David Speyer was partially supported by the National Science Foundation under Grant Number DMS-1600223.
Hugh Thomas was partially supported by an NSERC Discovery Grant and the Canada Research Chairs program.
}

\begin{document}

\begin{abstract}
  We prove a Fundamental Theorem of Finite Semidistributive Lattices (FTFSDL), modelled on Birkhoff's Fundamental Theorem of Finite Distributive Lattices. Our FTFSDL is of the form ``A poset $L$ is a finite semidistributive lattice if and only if there exists a set $\Sha$ with some additional structure, such that $L$ is isomorphic to the admissible subsets of $\Sha$ ordered by inclusion; in this case, $\Sha$ and its additional structure  are uniquely determined by $L$.''
The additional structure on $\Sha$ is a combinatorial abstraction of the notion of torsion pairs from representation theory and has geometric meaning in the case of posets of regions of hyperplane arrangements.
We show how the FTFSDL clarifies many constructions in lattice theory, such as canonical join representations and passing to quotients, and how the semidistributive property interacts with other major classes of lattices.
Many of our results also apply to infinite lattices.
\end{abstract}

\maketitle

\setcounter{tocdepth}{1}
\tableofcontents

\section{Introduction}\label{intro}  
A \newword{lattice} is a partially ordered set such that every pair $x,y$ of elements has a \newword{meet} (greatest lower bound) $x\meet y$ and \newword{join} (least upper bound) $x\join y$.
A lattice is \newword{distributive} if the meet operation distributes over the join operation and the join distributes over the meet.
(These two distributivity conditions are equivalent.)

The simplest class of examples of distributive lattices are the lattices of downsets in a fixed poset $P$.
A \newword{downset} (or \newword{order ideal}) in $P$ is a subset $D$ of $P$ such that if $x\in D$ and $y\le x$ then $y\in D$. 
Let $P$ be a poset and let $\Downsets(P)$ be the set of downsets in $P$, partially ordered by containment.
It is readily verified that the union of two downsets is a downset and that the intersection of two downsets is a downset.
As a consequence, $\Downsets(P)$ is a distributive lattice.
Birkhoff \cite{BirkhoffFTFDL} showed that every finite distributive lattice is of this form $\Downsets(P)$.
This result is often called Birkhoff's Representation Theorem, but since that name also refers to other theorems in universal algebra, we follow~\cite{EC1} in calling it the Fundamental Theorem of Finite Distributive Lattices (FTFDL).

An element $j$ of a lattice $L$ is \newword{join-irreducible} if, for all finite subsets $X$ of $L$, if $j = \Join X$ then $j\in X$.  
Equivalently, $j\in L$ is {join-irreducible}
if $j$ is not minimal in $L$ and cannot be written as $x \join y$ for $x,y<j$. 
In a finite lattice $L$, this is equivalent to saying that $j$ covers exactly one element of $L$; we will denote that unique element by $j_*$.
We write $\JIrr L$ for the set of join-irreducible elements of $L$, with the partial order induced from $L$.
Similarly, an element $m$ is \newword{meet-irreducible} if, for all finite subsets $X$ of $L$, if $m = \Meet X$ then $m\in X$.
Equivalently, $m$ is not maximal and we cannot write $m=x \meet y$ for $m<x$, $y$ or, equivalently in the finite case, if $m$ is covered by exactly one element $m^*$ of $L$.
We write $\MIrr L$ for the set of meet-irreducible elements of $L$, again with the induced partial order. 

\begin{theorem}[FTFDL]\label{FTFDL}
A finite poset $L$ is a distributive lattice if and only if it is isomorphic to $\Downsets(P)$ for some finite poset $P$.
In this case, $P$ is isomorphic to $\JIrr L$.  
The map $x\mapsto\{ j \in \JIrr L: j \leq x \}$ is an isomorphism from $L$ to $\Downsets(\JIrr L)$, with inverse $X\mapsto\Join X$.
\end{theorem}

The main result of this paper is a theorem, analogous to Theorem~\ref{FTFDL}, that characterizes a larger class of lattices: finite semidistributive lattices (defined below).
Two important examples of semidistributive lattices are the weak order on a finite Coxeter group  and the lattice of torsion classes of a finite-dimensional algebra (see Section~\ref{rep ssec}).
These two examples have recently been connected
by a series of papers~\cite{Mizuno,IRRT,DIRRT} relating the torsion classes of Dynkin type preprojective algebras to weak orders of the corresponding Coxeter groups.
Moreover, the Cambrian lattice, which describes the structure of the corresponding cluster algebra, is a semidistributive lattice described as a quotient of weak order, and can also be described as a lattice of torsion classes.

The language and terminology of this paper deliberately echoes these motivating examples:
Finite semidistributive lattices are realized in terms of binary relations on a set $\Sha$ (the Cyrillic letter ``sha''), suggestive of the \newword{shards} \cite{hyperplane,congruence,shardint} that govern much of the lattice theory of the weak order on a finite Coxeter group.
The relations $\t$, $\int$, and $\ont$ echo the structure of an abelian category.
We discuss these motivating examples in more detail in Section~\ref{conn sec}.

A lattice $L$ is \newword{join semidistributive} if whenever $x,y,z\in L$ satisfy $x\join y=x\join z$, they also satisfy $x\join(y\meet z)=x\join y$.
This is equivalent to the following condition:  
If $X$ is a nonempty finite subset of $L$ such that $x \join y = z$ for all $x \in X$, then $\left( \Meet_{x \in X} x \right) \join y = z$. 
The lattice is \newword{meet semidistributive} if the dual condition $(x\meet y=x\meet z)\implies(x\meet(y\join z)=x\meet y)$ holds.
Equivalently, if $X$ is a nonempty finite subset of $L$ such that $x \meet y = z$ for all $x \in X$, then $\left( \Join_{x \in X} x \right) \meet y = z$. 
The lattice is \newword{semidistributive} if it is both join semidistributive and meet semidistributive.
The notion of semidistributivity goes back to J\'{o}nsson \cite{Jonsson}, who in particular showed that free lattices are semidistributive.

In a finite semidistributive lattice, for each join-irreducible element $j$, the set $\{ y : j \meet y = j_* \}$ has a maximum element, which we call $\kappau(j)$.
For each meet-irreducible element $m$, the set $\{ x : m \join x = m^* \}$
has a minumum element $\kappad(m)$.
The existence of these elements characterizes semidistributivity of finite lattices.
The maps $\kappau$ and $\kappad$ are inverse bijections between $\JIrr L$ and $\MIrr L$ (see Theorem~\ref{semi char} and preceding references).

Given a (binary) 
relation $\t$ on a set $\Sha$ and a subset $X\subseteq\Sha$, we define 
\[X^\perp = \set{y\in\Sha : x\not\to y\  \forall x \in X } \quad  \mbox{and}  \quad ^\perp X = \set{ y\in\Sha : y \not\to x\  \forall x \in X }.\]
A \newword{maximal orthogonal pair} is a pair $(X,Y)$ of subsets of $\Sha$ with $Y = X^\perp$ and $X={}^\perp Y$. 
If $(X,Y)$ and $(X',Y')$ are maximal orthogonal pairs, then $X\subseteq X'$ if and only if $Y\supseteq Y'$.
The \newword{lattice of maximal orthogonal pairs for $\t$} is the set $\Pairs(\t)$ of maximal orthogonal pairs, partially ordered by containment in the first component, or equivalently reverse containment in the second component.
As we will discuss further in Section~\ref{MarkSec}, it is a classical fact that all lattices can be described as lattices of maximal orthogonal pairs from some relation, and in many ways. We will now describe conditions on $(\Sha, \t)$ which imply that $\Pairs(\t)$ is semidistributive, and such that each finite semidistributive lattice has a unique such representation.

Let $\Sha$ be a set and let $\t$ be a reflexive relation on $\Sha$. 
We use $\t$ to define two other relations, $\ont$ and $\int$ on $\Sha$.
Define $x\onto y$ iff for all $y\to z$, we also have $x\to z$.  
Dually, define $x\into y$ if and only if for all $z\to x$, we also have $z\to y$.
Each of $\ont$ and $\int$ is obviously a preorder (reflexive and transitive, but perhaps not antisymmetric). 
We call $(\ont, \int)$ the \newword{factorization} of $\t$ and write $(\ont, \int) = \Fact(\t)$.

Similarly, we define an operation $\Mult$ called \newword{multiplication} that takes an ordered pair of preorders on $\Sha$ to a reflexive relation on $\Sha$.
Namely, define $\Mult(\ont,\int)$ to be the relation $\t$ where $x\to z$ if and only if there exists $y\in\Sha$ such that $x\onto y\into z$.

We suggest pronouncing $\t$, $\ont$ and $\int$ as ``to", ``onto" and ``into". The relations $X \to Y$, $X \onto Y$ and $X \into Y$ should roughly be thought of as analogous to ``there is a nonzero map from $X$ to $Y$", ``there is a surjection from $X$ onto $Y$" and ``there is an injection from $X$ into $Y$" in some category to be defined later. 
For precise statements along these lines, see Section~\ref{rep ssec}; we warn the reader that the precise interpretation of $\ont$ and $\int$ is more subtle that the rough statement here. 
The definitions of $\Fact$ and $\Mult$ are easy to motivate in the context of this analogy.

We define a \newword{factorization system} to be a tuple $(\Sha,\t,\ont,\int)$ such that $\t$, $\int$, and $\ont$ are relations on a set $\Sha$ having $\Fact(\t) = (\ont, \int)$ and ${\Mult(\ont, \int) = \t}$. 
Since, in a factorization system, $\t$ and $(\ont, \int)$ determine each other, any condition on a factorization system can be thought of either as a condition on $\t$ or a condition on $(\ont, \int)$.
A factorization system $(\Sha,\t,\ont,\int)$ is called \newword{finite} if the set $\Sha$ is finite.

We say that a factorization system obeys the \newword{order condition} if we do \textbf{not} have $x \onto y \onto x$ or $x \into y \into x$ with $x \neq y$; in other words, the order condition states that the preorders $\int$ and $\ont$ are partial orders.
We will say that a factorization system obeys the \newword{brick condition} if we do not have $x \onto y \into x$ with $x \neq y$.  
The name ``brick condition" comes from the notion of a ``brick" in representation theory, which is a module $X$ such that any nonzero map $X \to X$ is an isomorphism. 
If we had $X \onto Y \into X$ for some $Y \not\cong 0, X$, then the composite map would be neither $0$ nor an isomorphism; our brick condition rules out the combinatorial analogue of this. See Section~\ref{rep ssec} for more precise statements.
We will say that a factorization system is \newword{two-acyclic} if it obeys the order condition and the brick condition. 

The first main result of the paper is the Fundamental Theorem of Finite Semidistributive Lattices.  
If $L$ is a finite semidistributive lattce, we define three relations on $\JIrr L$, namely $i \to_L j$ if and only if $i \not\leq \kappau(j)$ in $L$, $i \onto_L j$ if and only if $i\ge j$ in $L$, and $i\into_L j$ if and only if $\kappau(i)\ge\kappau(j)$ in $L$.

\begin{theorem}[FTFSDL]\label{FTFSDL1}  
A finite poset $L$ is a semidistributive lattice if and only if it is isomorphic to $\Pairs(\t)$ for a finite two-acyclic factorization system $(\Sha,\t,\ont,\int)$.  
In this case, $(\Sha,\t, \ont, \int)$ and $(\JIrr L,\t_L,\ont_L,\int_L)$ are isomorphic.  
The map
\[x\mapsto( \{ j \in \JIrr L: j \leq x \},\  \kappad \left( \{ m \in \MIrr L : m \geq x \} \right) )\]
is an isomorphism from $L$ to $\Pairs(\t_L)$, with inverse $(X,Y)\mapsto\Join X=\Meet(\kappau(Y))$.
\end{theorem}

\renewcommand{\1}{a}
\renewcommand{\2}{b}
\newcommand{\3}{g}
\newcommand{\4}{c}
\newcommand{\5}{f}
\newcommand{\6}{d}
\newcommand{\7}{e}
\newcommand{\labels}{\1,\2,\4,\6,\7,\5,\3} 
\begin{eg}\label{non unif} \label{2afs ex}  
We exhibit a two-acyclic factorization system $(\Sha,\t,\ont,\int)$, with $\Sha=\set{\labels}$.  
The non-reflexive relations of $\t$ and the Hasse diagrams of $\ont$ and $\int$ are shown in Figure~\ref{2afs fig}, and $L=\Pairs(\t)$ is shown in Figure~\ref{Pairs fig}.
The identification of $\Sha$ with $\JIrr L$ is marked on the Hasse diagram of $L$.
For each $j\in\JIrr L$, the element $\kappa(j)$ is also marked.
\end{eg}

In the course of proving Theorem~\ref{FTFSDL1}, we will prove some more general theorems that apply to some infinite lattices.
We therefore discuss some definitional distinctions which are trivial for finite lattices.
A \newword{complete lattice} is a poset $L$ such that every subset $X$ of $L$ has a \newword{meet} (greatest lower bound) $\Meet X$ and \newword{join} (least upper bound) $\Join X$.
The definition of a lattice given above is weaker:  It is equivalent to requiring that $\Meet X$ and $\Join X$ exist for all \emph{finite} subsets $X$.
In particular, all finite lattices are complete.

An element $j$ of a complete lattice $L$ is \newword{completely join-irreducible} if $j = \Join X$ implies $j\in X$ for subsets $X\subseteq L$.
Equivalently, $j$ is completely join-irreducible if and only if there exists an element $j_*$ such that $x<j$ if and only if $x\le j_*$.  
(In an infinite lattice, this is a stronger condition than simply requiring that $j$ covers exactly one element.)
The element $j_*$ is $\Join\set{x\in L:x<j}$.
(Recall that $j$ is join-irreducible if $j = \Join X$ implies $j\in X$ for \emph{finite} subsets $X\subseteq L$.
Thus every completely join-irreducible element is \ji, but not vice versa.)
A \newword{completely meet-irreducible} element $m$ is defined dually, and we write $m^*$ for the element $\Meet\set{x\in L:x>m}$ with the property that $x>m$ if and only if $x\ge m^*$.
We write $\JIrr^cL$ and $\MIrr^cL$ for the sets of completely join-irreducible elements and completely meet-irreducible elements.

\begin{figure}
\mbox{
\mbox{
  \begin{tikzpicture}[xscale=1,yscale=0.65]
  \node (3) at (0,3) {$\3$};
  \node (5) at (-2,2) {$\5$};
  \node (7) at (1.5,1) {$\7$};
  \node (6) at (2.25,0) {$\6$};
  \node (4) at (1.5,-1) {$\4$};
  \node (2) at (-2,-2) {$\2$};
  \node (1) at (0,-3) {$\1$};
  \draw[->] (2) -- (1);
  \draw[->] (3) -- (1);
  \draw[->] (3) -- (2);
  \draw[->] (3) -- (5);
  \draw[->] (3) -- (7);
  \draw[->] (4) -- (1);
  \draw[->] (4) -- (2);
  \draw[->] (5) -- (1);
  \draw[->] (5) -- (2);
  \draw[->] (5) -- (4);
  \draw[->] (5) -- (7);
  \draw[->] (6) -- (4);
  \draw[->] (7) -- (2);
  \draw[->] (7) -- (4);
  \draw[->] (7) -- (6);
  \end{tikzpicture}}
  \qquad
  \begin{tikzpicture}[xscale=1,yscale=1]
  \node (1) at (0,0) {$\1$};
  \node (2) at (0,1) {$\2$};
  \node (3) at (-0.5,2) {$\3$};
  \node (4) at (0.5,2) {$\4$};
  \node (5) at (0.5,3) {$\5$};
  \node (6) at (1,0) {$\6$};
  \node (7) at (1,1) {$\7$};
  \draw[->>] (2) -- (1);
  \draw[->>] (3) -- (2);
  \draw[->>] (4) -- (2);
  \draw[->>] (5) -- (4);
  \draw[->>] (7) -- (6);
  \end{tikzpicture}}
  \qquad
\mbox{
  \begin{tikzpicture}[xscale=1,yscale=1]
  \node (3) at (0,0) {$\3$};
  \node (5) at (0,-1) {$\5$};
  \node (1) at (-0.5,-2) {$\1$};
  \node (7) at (0.5,-2) {$\7$};
  \node (2) at (0.5,-3) {$\2$};
  \node (6) at (-1,0) {$\6$};
  \node (4) at (-1,-1) {$\4$};
  \draw[right hook->] (3) -- (5);
  \draw[right hook->] (5) -- (7);
  \draw[right hook->] (5) -- (1);
  \draw[right hook->] (7) -- (2);
  \draw[right hook->] (6) -- (4);
  \end{tikzpicture}}
\caption{A two-acyclic factorization system} 
\label{2afs fig}
\end{figure}
\begin{figure}
\mbox{
  \begin{tikzpicture}[xscale=0.85,yscale=0.85]
  \node (empty) at (0,0) [circle,fill,inner sep=1.5pt]{};
  \node (1) at (-1,1) {$\1$};
  \node (12) at (-2,2) {$\2$};
  \node (123) at (-4,4) {$\kappa(\4)=\3$\hspace*{24pt}};
  \node (124) at (-1,3) {$\4$};
  \node (6) at (2,2) {$\6$};
  \node (67) at (3,3) {\hspace*{24pt}$\7=\kappa(\1)$};
  \node (16) at (1,3) {$\kappa(\2)$};
  \node (1246) at (0,4) {$\kappa(\7)$};
  \node (12467) at (1,5) {$\kappa(\5)$};
  \node (1245) at (-2,4) {$\5$};
  \node (12345) at (-3,5) {$\kappa(\6)$};
  \node (124567) at (0,6) {$\kappa(\3)$};
  \node (1234567) at (-1,7) [circle,fill,inner sep=1.5pt]{};
  \draw[-] (empty) -- (1);
  \draw[-] (1) -- (12);
  \draw[-] (12) -- (123);
  \draw[-] (empty) -- (6);
  \draw[-] (6) -- (67);
  \draw[-] (1) -- (16);
  \draw[-] (6) -- (16);
  \draw[-] (16) -- (1246);
  \draw[-] (12) -- (124);
  \draw[-] (124) -- (1246);
  \draw[-] (1246) -- (12467);
  \draw[-] (124) -- (1245);
  \draw[-] (1245) -- (12345);
  \draw[-] (123) -- (12345);
  \draw[-] (12345) -- (1234567);
  \draw[-] (67) -- (12467);
  \draw[-] (12467) -- (124567);
  \draw[-] (124567) -- (1234567);
  \draw[-] (1245) -- (124567);
  \end{tikzpicture}}
\caption{The semidistributive lattice associated to Figure~\ref{2afs fig}} 
\label{Pairs fig}
\end{figure}

A complete lattice $L$ is \newword{spatial} if each element of $L$ can be written as a (possibly infinite) join of completely \ji elements. 
Equivalently, for all $x\in L$, we have $x=\Join_{j\leq x,\ j \in \JIrr^cL} j$.  
If the dual condition holds, we say that $L$ is \newword{dually spatial}. 
If a lattice is both spatial and dually spatial, we say it is \newword{bi-spatial}.

We will call a lattice $L$ a \newword{$\kappa$-lattice}  if it is complete, if it is bi-spatial, and if there are inverse bijections $\kappau:\JIrr^c L\to\MIrr^c L$ and $\kappad:\MIrr^cL\to\JIrr^cL$ such that $\kappau(j)$ is the maximum element of $\{ y : j \meet y = j_* \}$ and $\kappad(m)$ is the minimum element of $\{ x : m \join x = m^* \}$. 
It is a known result that a finite lattice is semidistributive if and only if it is a $\kappa$-lattice; we will reprove this as Corollary~\ref{fin kappa}.
We say that a $\kappa$-lattice $L$ is \newword{well separated} if whenever  $z_1 \not\leq z_2$, there is some $j \in\JIrr^cL$ with $z_1 \geq j$ and $\kappa(j)\geq z_2$.
(Note that the converse is clear: If there is some $j$ with $z_1 \geq j$ and $\kappa(j)\geq z_2$ then $z_1 \not\leq z_2$, as otherwise we would have $\kappa(j)\geq z_2 \geq z_1 \geq j$, which is absurd.)
See Lemma~\ref{weak-sep} for an equivalent description of the well separated property.

In any $\kappa$-lattice, we define relations $\t_L$, $\ont_L$, and $\int_L$ on $\JIrr^cL$ just as we did in the finite case: 
$i \to_L j$ if and only if $i \not\leq \kappau(j)$, $i \onto_L j$ if and only if $i\ge j$ in $L$, and $i\into_L j$ if and only if $\kappau(i)\ge\kappau(j)$ in $L$. 

We will prove the following theorem, which is a generalization of the Fundamental Theorem of Finite Semidistributive Lattices to well separated $\kappa$-lattices.
 
\begin{theorem}\label{2afs iff ws kappa}
A (not necessarily finite) poset $L$ is a well separated $\kappa$-lattice if and only if it is isomorphic to $\Pairs(\t)$ for a two-acyclic factorization system $(\Sha,\t,\ont,\int)$.  
If so, $(\Sha,\t, \ont, \int)$ is isomorphic to $(\JIrr^cL,\t_L,\ont_L,\int_L)$, and the map
\[x\mapsto( \{ j \in \JIrr^cL: j \leq x \},\  \kappad \left( \{ m \in \MIrr^cL : m \geq x \} \right) )\]
is an isomorphism from $L$ to $\Pairs(\t_L)$, with inverse $(X,Y)\mapsto\Join X=\Meet(\kappau(Y))$.
\end{theorem}

Section~\ref{main sec} contains the proof of Theorem~\ref{2afs iff ws kappa}, and Section~\ref{semi sec} shows 
how the Fundamental Theorem of Finite Semidistributive Lattices (Theorem~\ref{FTFSDL1}) follows as a special case.
In Section~\ref{inf sec}, we place Theorem~\ref{2afs iff ws kappa} in a broader context by discussing various conditions on a complete lattice that are equivalent, in the finite case, to semidistributivity.
In particular, we provide counterexamples to conceivable versions of FTFSDL that concern completely semidistributive infinite lattices.

Suppose $L$ is a finite semidistributive lattice, realized as the maximal orthogonal pairs for a two-acyclic factorization system $(\Sha,\t,\ont,\int)$.
It is apparent that any interval $[x,y]$ in $L$ is also a finite semidistributive lattice.
In Section~\ref{intervals sec}, given an interval $[x,y]$ in $\Pairs(\t)$, we explicitly construct a two-acyclic factorization system (defined on a subset of $\Sha$) whose lattice of maximal orthogonal pairs is isomorphic to $[x,y]$.  Our construction works for any lattice of the form $\Pairs(\t)$,  for a two-acyclic factorization system $(\Sha,\t,\ont,\int)$ (which need not be finite).

In Section~\ref{cjc sec}, we describe cover relations in the lattice of maximal orthogonal pairs, giving a particularly satisfying answer in the finite case.
A major tool for the study of finite semidistributive lattices is the notion of canonical join representation of elements.
Using the results on covers, we describe canonical join representations in the finite case in terms of the relation $\t$. 

It is known that every quotient of a finite semidistributive lattice is semidistributive.  
(See for example~\cite[Lemma~2.14(b)]{DIRRT} for a proof and for discussion of the infinite case.)
In Section~\ref{quot sec}, we describe all quotients of a finite semidistributive lattice in terms of the FTFSDL. 
We define a relation $\forcs$ on $\Sha$ and prove:

\begin{theorem}\label{thm finite quotients}
Suppose $(\Sha, \t, \ont, \int)$ is a finite two-acyclic factorization system and $\Che\subseteq\Sha$ is a $\forcs$-upset.
Then the restriction of $(\Sha, \t, \ont, \int)$ to $\Che$
is a two-acyclic factorization system and the map $(X,Y) \mapsto (X \cap \Che, Y \cap \Che)$ is a surjective lattice homomorphism between the corresponding lattices of maximal orthogonal pairs.
Every lattice quotient of $\Pairs(\t)$ arises in this way for a unique $\forcs$-upset~$\Che$.
\end{theorem}

Theorem~\ref{thm finite quotients} allows us to give a precise description of the lattice $\Con(\Pairs(\t))$ of congruences of any finite semidistributive lattice.
As discussed in Section~\ref{quot sec}, the congruence lattice of a finite lattice $L$ can be understood in terms of a preorder on $\JIrr L$ called the forcing preorder.
For $x\in\Sha$, define $T(x) = \{ x' : x \onto x' \}$. 

\begin{cor} \label{con pairs finite}
Suppose $(\Sha, \t, \ont, \int)$ is a finite two-acyclic factorization system. 
The map $x\mapsto(T(x), T(x)^{\perp})$ is an isomorphism from the transitive closure of $\forcs$ on $\Sha$ to the forcing preorder on $\JIrr(\Pairs(\t))$. 
This map induces an isomorphism from the poset of $\forcs$-downsets under containment to the congruence lattice $\Con(\Pairs(\t))$.
\end{cor}

We prove Theorem~\ref{thm finite quotients} and Corollary~\ref{con pairs finite} by generalizing them to the infinite case, with appropriate conditions on the two-acyclic factorization systems and congruences (Theorem~\ref{thm lattice quotients} and Corollary~\ref{con pairs infinite}).

Section~\ref{lattice types sec} relates the FTFSDL to work on other classes of finite lattices, namely distributive lattices, congruence uniform lattices, general finite lattices, and extremal lattices.
Section~\ref{conn sec} concludes the paper by discussing the motivating examples, namely posets of regions of tight hyperplane arrangements and lattices of torsion classes of finite-dimensional algebras.

\section{The fundamental theorem}\label{main sec} 
In this section, we prove the FTFSDL by proving its generalization Theorem~\ref{2afs iff ws kappa}.
We will separate out the following two auxiliary theorems: Theorem~\ref{kappa 2afs} is a useful fact about $\kappa$-lattices that are not necessarily well separated; Theorem~\ref{2afs kappa}  is a portion of Theorem~\ref{2afs iff ws kappa}.

\begin{theorem}\label{kappa 2afs}   
If $L$ is a $\kappa$-lattice, then the map
\[x\mapsto( \{ j \in \JIrr^cL: j \leq x \},\  \kappad \left( \{ m \in \MIrr^cL : m \geq x \} \right) )\]
is an isomorphism from $L$ to $\Pairs(\t_L)$, with inverse $(X,Y)\mapsto\Join X=\Meet(\kappau(Y))$.
The system $(\JIrr^cL,\t_L,\ont_L,\int_L)$ has all of the properties of a two-acyclic factorization system except that possibly $\Mult(\ont_L,\int_L)\subsetneq\t_L$.
\end{theorem}

\begin{theorem}\label{2afs kappa}   
If $(\Sha,\t,\ont,\int)$ is a (not necessarily finite) two-acyclic factorization system, then $\Pairs(\t)$ is a well separated $\kappa$-lattice.
Furthermore, writing $L$ for $\Pairs(\t)$, the system $(\Sha,\t, \ont, \int)$ is isomorphic to $(\JIrr^cL,\t_L,\ont_L,\int_L)$.
\end{theorem}

After some preliminary material in Sections~\ref{conv sec} and~\ref{fund sec}, we will prove Theorem~\ref{kappa 2afs} in Section~\ref{toL sec}.
We prove Theorem~\ref{2afs kappa} in Section~\ref{clos sec} and then complete the proof of Theorem~\ref{2afs iff ws kappa}.
Finally, in Section~\ref{semi sec}, we prove the Fundamental Theorem of Finite Semidistributive Lattices (Theorem~\ref{FTFSDL1}) by showing that it is the finite case of Theorem~\ref{2afs iff ws kappa}.

\subsection{Conventions about partially ordered sets}\label{conv sec}
In a two-acyclic factorization system, the relations $\ont$ and $\int$ are partial orders. When we use the language of partial orders to describe them, we consider $x$ to be greater than or equal to $y$ if $x \onto y$ or $x \into y$ respectively. 
For example, if we say ``$x$ is an $\ont$-minimal element of $S$", we mean that, for any $x' \in S$ other than $x$, we do not have $x \onto x'$. 
Throughout the paper, if we say that a set $P$ is a partially ordered set under the relation $R$ and if $R$ is some sort of arrow pointing from $x$ to $y$, then we intend $x\ R\ y$ to correspond to $x \geq y$.  

For a directed graph with vertex set $V$, we define a \newword{downset} of $V$ to be a set $D \subseteq V$ such that $x \in D$ and $x \to y$ implies $y \in D$; we define $U$ to be an \newword{upset} if $y \in U$ and $x \to y$ implies $x \in U$.
We allow ourselves to use this language even when the relation $\t$ is not transitive and not antisymmetric.
The downsets of $\t$ are the same as the downsets of the transitive closure of $\t$, and are in natural bijective correspondence with the downsets of the graph formed by collapsing each strongly connected component of $\t$ to a point.

In a partially ordered set $P$, we say that an element $x$ covers another element $y$, written $x \covers y$, if $x > y$ and there does not exist any $z$ with $x> z > y$.
The Hasse diagram of $P$ is the graph with vertex set $P$ and an edge from $x$ to $y$ if $x \covers y$.

\subsection{Fundamentals of factorization systems}\label{fund sec}
We now prove some basic facts about the operations $\Fact$ and $\Mult$ that were defined in the introduction.
\begin{prop}\label{Mult Fact facts}
Suppose $\t$ is a reflexive relation, and $\ont$ and $\int$ are reflexive and transitive relations, on a set $\Sha$. \begin{enumerate}[\qquad\rm1.]
\item \label{obv}
Given two pairs of reflexive and transitive relations $(\ont_1,\int_1)\subseteq(\ont_2,\int_2)$, then $\Mult(\ont_1,\int_1)\subseteq\Mult(\ont_2,\int_2)$.
\item \label{union}
$(\ont\cup \int)\subseteq\Mult(\ont,\int)$.
\item \label{mf}
$\Mult(\Fact(\t))\subseteq \t$.
\item \label{fm}
$\Fact(\Mult(\ont,\int))\supseteq (\ont,\int)$.
\item \label{mfm}
$\Mult(\Fact(\Mult(\ont,\int)))=\Mult(\ont,\int)$.
\end{enumerate}
The containment symbol between ordered pairs stands for containment in each component.
\end{prop}
\begin{proof}
Property~\ref{obv} is obvious. 

For Property~\ref{union}, write $\t$ for $\Mult(\ont,\int)$.  
Suppose $x\onto y$.
By definition of multiplication, since $\int$ is reflexive (and thus $y\into y$), we have $x\to y$. 
Similarly, if $x\into y$, then since $x\onto x$, we have $x\to y$.

For Property~\ref{mf}, write $(\ont',\int')$ for $\Fact(\t)$ and write $\t'$ for $\Mult(\Fact(\t))$.
If $x\to'y$, then there exists $z\in\Sha$ such that $x\onto' z\into' y$.
By the definition of $\int'$ and the fact that $\t$ is reflexive, we have $z \to y$. 
Then, by the definition of $\ont'$, we also have $x \to y$.

To prove Property~\ref{fm}, we write $\t'$ for $\Mult(\ont,\int)$, and write $(\ont',\int')$ for $\Fact(\Mult(\ont,\int))$.
We must show that, if $x \onto y$, then $x \onto' y$. 
Unpacking the definition of $\ont'$, we must show that, if $y \to' z$, then $x \to' z$.
Since $y\to' z$, there exists $w$ with $y\onto w\into z$.
Then by transitivity of $\ont$, we have $x\onto w\into z$, so $x\to' z$.
We have shown that $x\onto'y$.
The proof for $\int$ and $\int'$ is similar.

For Property~\ref{mfm}, we have $\Mult(\Fact(\Mult(\ont,\int)))\subseteq\Mult(\ont,\int)$ by Property~\ref{mf}.
Combining Properties~\ref{obv} and~\ref{fm}, we have the opposite containment as well.
\end{proof}

Proposition~\ref{Mult Fact facts}.\ref{union} implies the following fact, which is useful in computing examples.
\begin{prop}\label{useful}
If $(\Sha, \t, \ont, \int)$ is a factorization system, then $(\ont\cup \int)\subseteq\t$.
\end{prop}

Another useful fact to keep in mind, in computing examples, is the following proposition.

\begin{prop} \label{TorsionImpliesDownset} 
Let $\t$ be a relation on $\Sha$ and let $\Fact(\t) = (\ont, \int)$. Let $(T,F)$ be a maximal orthogonal pair for $\t$. Then $T$ is a downset for $\ont$ and $F$ is an upset for $\int$. In other words, if $x \onto x'$ and $x \in T$ then $x' \in T$ and, if $y' \into y$ and $y \in F$ then $y' \in F$.
\end{prop} 

\begin{proof}
We prove the claim about $T$; the claim about $F$ is analogous. Let $x \onto x'$ and suppose that $x \in T$. We must show that, if $x \not\to y$ then $x' \not\to y$. 
The contrapositive, that if $x' \to y$ then $x \to y$, is the definition of $\ont$. 
\end{proof}

\begin{remark}
The letters $T$ and $F$ suggest ``Torsion" and ``torsion-Free", terminology from the theory of torsion pairs. See Section~\ref{rep ssec}. 
\end{remark}

This would be an excellent moment for the reader to take another look at Example~\ref{2afs ex}.
In one direction, one can verify directly that $(\Sha,\t,\ont,\int)$ is a two-acyclic factorization system.
In light of Proposition~\ref{TorsionImpliesDownset}, the task in computing $\Pairs(\t)$ is to determine which $\ont$-downsets $X$ have the property that $(X,X^\perp)$ is a maximal orthogonal pair.
Equivalently, as will be explained in Section~\ref{clos sec}, the task is to find all sets $X$ such that ${}^\perp(X^\perp)=X$.
In the other direction, one can verify directly that the lattice $L$ shown in Figure~\ref{Pairs fig} is semidistributive and then construct $(\JIrr L,\t_L,\ont_L,\int_L)$ as in Theorem~\ref{FTFSDL1}.

\begin{prop} \label{AcyclicStronger}
Let $(\Sha, \t, \ont, \int)$ be a two-acyclic factorization system. If $x \to y \into x$ or $x \onto y \to x$, then $x = y$.
\end{prop}

\begin{proof}
We prove the first statement;
the second is similar. Let $x \to y \into x$. Since $\t = \Mult(\ont, \int)$, there is some $w$ with $x \onto w \into y \into x$. By transitivity of $\int$, we have $x \onto w \into x$ so, by the brick property $x=w$. Now we have $x=w \into y \into x$ so, by the order property, $x=y$ as well.
\end{proof}

Our definitions have a symmetry under reversing arrows, which we spell out in the easy proposition below. 
\begin{prop}\label{duality}
Let $(\Sha,\t, \ont, \int)$ be a factorization system. Define $\t^{\op}$ by $y \to^{\op} x$ if $x \to y$ and similarly define $x \into^{\op} y$ if $y \into x$ and $x \onto^{\op} y$ if $y \onto x$. 
Then $(\Sha,\to^{\op}, \int^{\op}, \ont^{\op})$ is a factorization system; to be explicit, $\int^{\op}$ has taken on the role $\ont$ and vice versa. 
The system $(\Sha,\to^{\op}, \int^{\op}, \ont^{\op})$ is two-acyclic if and only if $(\Sha,\t, \ont, \int)$ is.
The map $(X,Y)\mapsto(Y,X)$ is an anti-isomorphism from $\Pairs(\t^{\op})$ to $\Pairs(\t)$.
\end{prop}

In a factorization system, the pair $(\ont, \int)$ determines $\t$. 
It is often practical to construct examples of factorization systems by giving $(\ont, \int)$ and defining $\t = \Mult(\ont, \int)$. 
It is therefore useful to state the hypotheses of being two-acyclic and factorization in terms solely of the partial orders $\ont$ and $\int$. 
After introducing some notation, we do this:

Given partial orders $\ont$ and $\int$ on $\Sha$ and a subset $X\subseteq\Sha$, we write $\downonto X$ for the set $\set{y\in\Sha:\exists x\in X,\,x\onto y}$ of elements of $\Sha$ that are
 below elements of $X$ in the sense of $\ont$.
Similarly, we write  $\uponto X$, $\downinto X$ and $\upinto X$.

\begin{prop}\label{posets only 2-afs}
Given two partial orders $\int$ and $\ont$ on a set $\Sha$, the tuple $(\Sha,\Mult(\ont,\int),\ont,\int)$ is a two-acyclic factorization system if and only if the following conditions hold.
\begin{enumerate}[\qquad\rm(i)]
\item \label{brick crit}
There do not exist distinct $x$ and $y$ in $\Sha$ with $x \into y$ and $y \onto x$.
\item \label{ont char}
$x\onto y$ if and only if $\downinto\downonto\set{y}\subseteq\downinto\downonto\set{x}$ for all $x,y\in\Sha$.
\item \label{int char}
$x\into y$ if and only if $\uponto\upinto\set{x}\subseteq\uponto\upinto\set{y}$ for all $x,y\in\Sha$.
\end{enumerate}
If the conditions hold, then for any $X\subseteq\Sha$, the sets $X^\perp$ and $^\perp X$ are described by 
\[
X^\perp = \Sha\setminus \bigl(\downinto(\downonto X)\bigr) \quad  and  \quad ^\perp X = \Sha\setminus \bigl(\uponto(\upinto X)\bigr).
\]
\end{prop}

\begin{proof}
We write $\t$ for $\Mult(\ont,\int)$.  Then $\downinto(\downonto X)=
 {\{y: \exists x\in X, x\to y\}}$ and $\uponto(\upinto X)=\{y: \exists x\in X, y\to x\}$. 
Conditions \eqref{ont char} and \eqref{int char} amount to the assertion that $\Fact(\Mult(\ont,\int))=(\ont,\int)$.
Thus $(\Sha,\Mult(\ont,\int),\ont,\int)$ is a factorization system if and only if \eqref{ont char} and \eqref{int char} hold.
The order condition holds in any case, because $\ont$ and $\int$ are partial orders.
Condition~\eqref{brick crit} is a restatement of the brick condition. The final
statement is clear.
\end{proof}

\subsection{From $\kappa$-lattices to factorization systems}\label{toL sec} 
In this section, we prove Theorem~\ref{kappa 2afs} and begin the proof of Theorem~\ref{2afs iff ws kappa}.

\begin{prop} \label{Fact toL} 
Suppose $L$ is a $\kappa$-lattice.
Then $\Fact(\t_L)=(\ont_L,\int_L)$ and $\Mult(\ont_L,\int_L)\subseteq\t_L$.
\end{prop}
\begin{proof}
There are two facts to verify.
First, we want to show that $i\onto_L j$ if and only if whenever $j\to_L \ell$, we also have $i\to_L \ell$.
That is, we want $i\ge j$ if and only if $(j\not\le\kappau(\ell))\implies(i\not\le\kappau(\ell))$ for all $\ell\in\JIrr^cL$.
Taking the contrapositive, we want $i \geq j$ if and only if $\{ \ell \in \JIrr^cL : i \leq  \kappau(\ell) \} \subseteq \{ \ell \in \JIrr^c L: j  \leq  \kappau(\ell) \}$.
The forward implication is obvious, and, since the image of $\kappau$ is $\MIrr^c L$, the reverse follows because in a $\kappa$-lattice, every element is the meet of the set of completely meet irreducible elements above it.

Second, we want to show that $i\into_L j$ if and only if whenever $\ell\to_L i$, we also have $\ell\to_L j$. The argument is dual. 
We see that $\Fact(\t_L)=(\ont_L,\int_L)$.
Proposition~\ref{Mult Fact facts}.\ref{mf} thus implies that $\Mult(\ont_L,\int_L)\subseteq\t_L$.
\end{proof}

We now check that $(\JIrr^cL, \t_L, \ont_L, \int_L)$ is two-acyclic. We first check the order condition, and then the brick condition.

\begin{prop} \label{finite length}
Suppose $L$ is a $\kappa$-lattice.
If $x \onto_L y \onto_L x$ or $x \into_L y \into_L x$, then $x=y$.
\end{prop}

\begin{proof}
The definition of $x \onto_L y \onto_L x$ is that $x \leq y \leq x$, so this follows because $\leq$ is a partial order.
Similarly, $x \into_L y \into_L x$ means $\kappau(x) \leq \kappau(y) \leq \kappau(x)$ so $\kappau(x) = \kappau(y)$, and the map $\kappau$ is invertible so $x=y$.
\end{proof}

\begin{prop}\label{brick ontoL intoL}
Suppose $L$ is a $\kappa$-lattice and $x,y \in \JIrr^cL$.
If $x \onto_L y \into_L x$, then $x=y$.
\end{prop}
\begin{proof}
The hypotheses of the proposition are that $x \geq y$ and $\kappau(y) \geq \kappau(x)$. 
As before, let $x_{\ast}$ be the unique element covered by $x$. 
If $x > y$ then $x_* \geq y$ so $\kappau(y) \geq \kappau(x) \geq x_* \geq y$, contradicting that $\kappau(y) \not\geq y$. 
We deduce that $x$ is not strictly greater than $y$, so $x=y$.
\end{proof}

It remains to construct an isomorphism from $L$ to $\Pairs(\t_L)$.
We define $J_x = {\{ j \in \JIrr^cL : j \leq x \}}$ and $M_x = \{ m \in \MIrr^cL : m \geq x \}$.

\begin{prop}\label{isom}
Suppose $L$ is a $\kappa$-lattice.
The map $x\mapsto(J_x,\kappad(M_x))$ is an isomorphism from $L$ to $\Pairs(\t_L)$.
\end{prop}
\begin{proof}
Since each $x\in L$ has $x=\Join J_x$, the map $x\mapsto(J_x,\kappad(M_x))$ is an isomorphism from $L$ to its image, where the image is partially ordered by containment in the first entry and equivalently reverse containment in the second entry.
Thus it remains only to show that the image is the set of maximal orthogonal pairs for $\t_L$.

For $x\in L$ the set $J_x^\perp$ is $\set{j\in\JIrr^cL:i\le\kappau(j),\,\forall i\in J_x}=\set{j\in\JIrr^cL:x\le\kappau(j)}=\kappad(\set{m\in\MIrr^cL: x\le m})=\kappad(M_x)$.
Also, $^\perp(\kappad(M_x))=\set{i\in\JIrr^cL:{i\le m,}\,{\forall m\in M_x}}$, which equals ${\set{i\in\JIrr^cL:i\le x}}=J_x$.
We see that $(J_x,\kappad(M_x))$ is a maximal orthogonal pair.

On the other hand, suppose $X$ and $Y$ are subsets of $\JIrr^cL$ such that $(X,Y)$ is a maximal orthogonal pair for $\t_L$ and write $x=\Join X$.
Then $Y$ equals $X^\perp$, which equals $\set{j\in\JIrr^cL:i\le\kappau(j),\,\forall i\in X}=\set{j\in\JIrr^cL:x\le \kappau(j)}=\kappad(M_x)$.
Computing $^\perp(\kappad(M_x))$ as in the previous paragraph, we see that $X=J_x$.
\end{proof}

This completes the proof of Theorem~\ref{kappa 2afs}. Theorem~\ref{kappa 2afs} deals with $\kappa$-lattices which may not be well separated, and constructs factorization systems which may not obey $\Mult(\ont_L, \int_L) = \t_L$. The following proposition shows that the hypothesis of well separation fixes this defect. 

\begin{prop}\label{well-sep kappa pairs}
If $L$ is a well separated $\kappa$-lattice, then $(\JIrr^cL, \t_L, \ont_L, \int_L)$ is a two-acyclic factorization system and $L$ is isomorphic to $\Pairs(\t_L)$.
\end{prop}

\begin{proof}
Since $L$ is a $\kappa$-lattice, by Theorem~\ref{kappa 2afs} it remains only to check that $\t_L\subseteq \Mult(\ont_L, \int_L)$.
Let $i$ and $j$ be completely \ji elements with $i \to_L j$ or, in other words, $i\not\leq \kappa(j)$. 
Since $L$ is well separated, there is a join irreducible element $\ell$ such that $i \geq \ell$ and $\kappa(\ell) \geq \kappa(j)$.
In other words, $i \onto_L \ell \into_L j$. 
\end{proof}

We have now proved all the parts of Theorem~\ref{2afs iff ws kappa} which state that a well separated $\kappa$-lattice has a particular form.
In Section~\ref{clos sec}, we turn to proving the reverse direction. 

\subsection{From factorization systems to $\kappa$-lattices}\label{clos sec}  
In this section, we prove Theorem~\ref{2afs kappa} and complete the proof of Theorem~\ref{2afs iff ws kappa}.  
We begin by quickly showing that $\Pairs(\t)$ is a complete lattice for any binary relation $\t$.

Given a binary relation $\t$ on a set $\Sha$, and a subset $X\subseteq\Sha$, we define $\overline X$ to be ${}^\perp(X^\perp)$. We call a set $X$ \newword{closed} if $\overline{X} = X$.
We quote the following standard results from Birkhoff~\cite{Birkhoff}. 
Note that Birkhoff works more generally with two sets and a relation between them (a situation which will occur for us in Section~\ref{MarkSec} ) and that his relation $\rho$ corresponds to our relation $\not \to$. 

\begin{prop}[\cite{Birkhoff}, Theorem V.19]  \label{closure} 
For any binary relation $\t$, the map $X\mapsto\overline X$ is a closure operator on $\Sha$. 
In other words, \begin{itemize} \item $X\subseteq\overline{X}$,
\item $\overline{\overline{X}}=\overline{X}$ for all $X\subseteq\Sha$, and
  \item $X \subseteq Y$ implies $\overline{X} \subseteq \overline{Y}$ for all $X,Y\subseteq\Sha$.\end{itemize}
\end{prop}

\begin{prop}[\cite{Birkhoff}, Theorems V.1 and V.2]  \label{lattice}  
Let $X \mapsto \overline{X}$ be any closure operator on a set $\Sha$. Then the containment order on closed subsets of $\Sha$ is a complete lattice.
The meet operation is $\cap$.
\end{prop}

We also make some easy observations.

\begin{prop}\label{step} 
If $\t$ is a binary relation on a set $\Sha$, then $(X,Y)\mapsto X$ is an isomorphism from $\Pairs(\t)$ to the complete lattice of closed sets under containment order.
The inverse is $X\mapsto(X,X^\perp)$.
\end{prop}

\begin{prop} \label{perp is closed}
If $\t$ is a binary relation on $\Sha$, and $S \subseteq \Sha$, then ${}^{\perp} S$ is closed.
\end{prop}

\begin{proof}
We must show that $\overline{{}^{\perp} S} = S$. Expanding the definition of closure, this says that ${}^{\perp} S  = {}^{\perp} {\big (} {\big (} {}^{\perp} S {\big )}^{\perp} {\big )}$. This is standard, see for example~\cite{Birkhoff}, the Corollary before Theorem V.19.
\end{proof}

We will often identify $\Pairs(\t)$ with the poset of closed sets and specify elements $(X,Y)$ of $\Pairs(\t)$ by giving only the closed set $X$.  

We now show that $\Pairs(\t)$ is a $\kappa$-lattice when $(\Sha,\t, \ont, \int)$ is a two-acyclic factorization system.
To this end, we study completely join-irreducible elements and completely meet-irreducible elements in $\Pairs(\t)$.  
For $x\in\Sha$, define 
\[T(x) = \{ x' : x \onto x' \} \quad\text{and}\quad F(x) = \{ x' : x' \into x \}.\]
\begin{prop}\label{TF prop}
Let $(\Sha, \t, \ont, \int)$ be a two-acyclic factorization system. 
For any $x \in \Sha$, we have $T(x) = {}^{\perp} (\{ x \}^{\perp})$ and $F(x) = ({}^{\perp} \{ x \})^{\perp}$. Moreover, $x$ is uniquely determined by $T(x)$ and is uniquely determined by $F(x)$. 
\end{prop}

\begin{proof} 
Checking that $T(x)$ is the closure $\overline{\{x\}} : = {}^{\perp} (\{ x \}^{\perp})$ is a matter of unfolding definitions: 
We have $x' \in \overline{\{ x \}}$ if and only if, for all $y \in \Sha$, if $x \not\to y$ then $x' \not\to y$. 
Taking the contrapositive, we have $x' \in \overline{\{ x \}}$ if and only if $x' \to y$ implies $x \to y$, which is the definition of $x \onto x'$.
If $T(x_1) = T(x_2)$, then $x_1 \onto x_2$ and $x_2 \onto x_1$, so the order condition implies that $x_1=x_2$.
We have checked both assertions about $T(x)$; the assertions about $F(x)$ are proved similarly.
\end{proof}

We can now describe the completely join-irreducible elements and completely meet-irreducible elements of $\Pairs(\t)$.  
We define $T_*(x) = T(x) \setminus \{ x \}$ and $F^*(x)=F(x)\setminus\set{x}$.

\begin{prop} \label{JIrr and MIrr Characterize}  
Let $(\Sha, \t, \ont, \int)$ be a two-acyclic factorization system. 
\begin{enumerate}[\qquad\rm1.]
\item
$\JIrr^c(\Pairs(\t))=\set{(T(x), T(x)^{\perp}):x\in\Sha}$.
The unique element covered by $(T(x), T(x)^{\perp})$ is $(T_*(x), T_*(x)^{\perp})$.
\item
$\MIrr^c(\Pairs(\t))=\set{({}^{\perp} F(x), F(x)):x\in\Sha}$.
The unique element covering $({}^{\perp} F(x), F(x))$ is $({}^{\perp} F^*(x), F^*(x))$.
\end{enumerate}
\end{prop}

\begin{proof}  
We check the claim about completely join-irreducible elements; the claim about completely meet-irreducible elements is dual. 

The set $T(x)$ is closed by Proposition~\ref{TF prop}.  
If $T'$ is a closed set with $T'\subseteq T(x)$ and $T'\not\subseteq T_*(x)$, then $x \in T'$, and thus $T(x)\subseteq T'$ by Proposition~\ref{TorsionImpliesDownset}.
That is, every closed set strictly contained in $T(x)$ is contained in $T_*(x)$, so we can complete the proof by showing that $T_*(x)$ is closed.
Since closure preserves containment and $T(x)$ is closed, the closure of $T_*(x)$ is either $T_*(x)$ or $T(x)$. To see that $\overline{T_*(x)} \neq T(x)$, note that $x \in T_*(x)^{\perp}$ by Proposition~\ref{AcyclicStronger}. So $x \not \in {}^{\perp} (T_*(x)^{\perp}) =  \overline{T_*(x)}$ and we deduce that $T_*(x)$ is closed.

Conversely, suppose that $T$ is a completely join-irreducible closed set. 
By Proposition~\ref{TorsionImpliesDownset}, $T$ is a downset of $\ont$, so $T = \bigcup_{x \in T}T(x)$.
Each $T(x)$ is closed, so $T=\Join_{x \in T}T(x)$. 
By the definition of complete join-irreducibility, $T=T(x)$ for some $x\in T$, as desired.
\end{proof}

Thus in a two-acyclic factorization system, $(T(x), T(x)^{\perp}) \leftrightarrow ({}^{\perp} F(x), F(x))$ is a bijection between $\JIrr^c(\Pairs(\t))$ and $\MIrr^c(\Pairs(\t))$.
We will show that these bijections yield maps $\kappau$ and $\kappad$ manifesting that $\Pairs(\t)$ is a $\kappa$-lattice.

\begin{prop}\label{pairs kappa}  
Let $x\in\Sha$.  
Then ${}^{\perp} F(x)$ is the maximum element in the set of closed sets $T$ obeying $T(x) \cap T = T_*(x)$.
\end{prop}

\begin{proof}
We first check that $T(x) \cap ({}^{\perp} F(x)) = T_*(x)$.
Since $x \in F(x)$, we have $x \not\in {}^{\perp} F(x)$, so we only need to show that $T_*(x)\subseteq {}^{\perp}F(x)$.
That is, given $y\in\Sha$ with $x \onto y$ and $x \neq y$ and $z\in\Sha$ with $z\into x$, we need to show that $y\not\to z$.
But if $y \to z$, then there exists $w\in\Sha$ such that $y \onto w \into z$. 
By the transitivity and antisymmetry of $\ont$, we have $x \onto w$ and $x \neq w$. 
By the transitivity of $\int$, we have $w \into x$. 
But then $x \onto w \into x$, contradicting the brick condition.
By this contradiction, we conclude that $T(x) \cap ({}^{\perp} F(x)) = T_*(x)$.

It remains to show that any closed set $T$ with $T(x) \cap T = T_*(x)$ has $T \subseteq {}^{\perp} F(x)$. 
The hypothesis that $T(x) \cap T = T_*(x)$ can be restated as $T_*(x) \subseteq T$ and $x \not\in T$.

Suppose for the sake of contradiction that there is some $u \in T$ and $u \not\in {}^{\perp} F(x)$.
Thus $u \to v \into x$ for some $v$, so there exists $y$ such that $u \onto y \into v$.
Since $T$ is an order ideal for $\ont$, we have $y\in T$ and, by transitivity of $\int$, we have $y \into x$.

Now, $x \not \in T$ and $T$ is closed, so there is some $z \in T^\perp$ such that $x \to z$.
Thus, for some $w$, we have $x\onto w\into z$, and in particular $w\to z$ (Proposition~\ref{Mult Fact facts}.2).
If $w \neq x$ then $w \in T_*(x) \subseteq T$, but then the facts $w\to z$ and $z\in T^\perp$ contradict each other.
Therefore, we must have $x=w$, so that $x \into z$. 
We also showed above that $y \into x$, so now we conclude that $y \into z$ and thus $y\to z$ (using Proposition~\ref{Mult Fact facts}.2 again).
Since $y \in T$, this contradicts $z \in T^\perp$. 
We conclude that $T\subseteq {}^{\perp} F(x)$, and the proof is complete.
\end{proof}

Writing $F^*(x)$ for $F(x)\setminus\set{x}$, the following proposition is dual to Proposition~\ref{pairs kappa}.  

\begin{prop}\label{pairs kappa dual} 
Let $x\in\Sha$.  
Then $T(x)$ is the minimum element in the set of closed sets $T$ obeying ${}^\perp F(x) \join T = {}^\perp F^*(x)$.
\end{prop}

Propositions~\ref{pairs kappa} and~\ref{pairs kappa dual} combine to prove the following piece of Theorem~\ref{2afs kappa}.  
\begin{prop}\label{pairs kappa explicit}
If $(\Sha,\t,\ont,\int)$ is a two-acyclic factorization system, then $\Pairs(\t)$ is a $\kappa$-lattice.
Specifically,
\[
\kappau(T(x),T(x)^\perp)=({}^\perp F(x),F(x))\quad
\text{and} \quad
\kappad({}^\perp F(x),F(x))=(T(x),T(x)^\perp). \]
\end{prop}

To prove that $\Pairs(\t)$ is well separated, we first point out a simple lemma.

\begin{lemma} \label{ws lemma}
Let $(\Sha, \t, \ont, \int)$ be a two-acyclic factorization system and let $(X_1, Y_1)$ and $(X_2, Y_2) \in \Pairs(\t)$. If $X_1 \cap Y_2 = \emptyset$ then $(X_1, Y_1) \leq (X_2, Y_2)$. 
\end{lemma}

\begin{proof}
We show the contrapositive. Suppose that $X_1 \not\subseteq X_2$. 
Then there is some $p \in X_1 \setminus X_2$. 
Since $p \not \in X_2$, there is some $r \in Y_2$ with $p \to r$. 
Factor this arrow as $p \onto q \into r$. Then $q \in X_1 \cap Y_2$.
\end{proof}

\begin{prop} \label{pairs is ws}
If $(\Sha, \t, \ont, \int)$ be a two-acyclic factorization system, then $\Pairs(\t)$ is well separated.
\end{prop}

\begin{proof}
For $z_1$ and $z_2 \in \Pairs(\t)$ with $z_1 \not\leq z_2$, we need to show that there is a completely join irreducible element $j$ of $\Pairs(\t)$ with $z_1 \geq j$ and $\kappau(j) \geq z_2$.  Let $z_1 = (X_1, Y_1)$ and $z_2 = (X_2, Y_2)$. By the contrapositive of Lemma~\ref{ws lemma}, the assumption that $z_1 \not\leq z_2$ implies that $X_1 \cap Y_2 \neq \emptyset$; let $q \in X_1 \cap Y_2$. Put $j = (T(q), T(q)^{\perp})$. By Proposition~\ref{JIrr and MIrr Characterize}, $j$ is completely join irreducible; by Proposition~\ref{pairs kappa}, $\kappa(j) = ({}^{\perp} F(q), F(q))$. Because $X_1$ and $Y_2$ are a $\ont$-downset and an $\int$-upset respectively, we have $X_1 \supseteq T(q)$ and $F(q) \subseteq Y_2$, so  $z_1 \geq j$ and $\kappau(j) \geq z_2$.
\end{proof}

To complete the proof of Theorem~\ref{2afs kappa}, we need to show that, if $L \cong \Pairs(\t)$ for a two-acyclic factorization system $(\Sha, \t, \ont, \int)$, then $(\Sha, \t, \ont, \int)$ is isomorphic to 
$(\JIrr^cL, \t_L, \ont_L, \int_L)$, for $\t_L$, $\ont_L$, and $\int_L$ as defined just before Theorem~\ref{2afs iff ws kappa}. 

If $L \cong \Pairs(\t)$, then Proposition~\ref{JIrr and MIrr Characterize} provides bijections $\Sha \leftrightarrow \JIrr^cL \leftrightarrow \MIrr^cL$ given by $x \leftrightarrow (T(x),  T(x)^{\perp}) \leftrightarrow ({}^{\perp}F(x), F(x))$.  
We now establish that these bijections turn $\ont$ and $\int$ into $\ont_L$ and $\int_L$. 

\begin{prop}\label{bijections turn}  
Suppose $x,y \in \Sha$.  Then $x \onto y$ if and only if $(T(x), {} T(x)^\perp)  \onto_L (T(y), {} T(y)^\perp)$ and $x \into y$ if and only if $({}^\perp F(x), F(x)) \into_L ({}^\perp F(y), F(y))$.
\end{prop}

\begin{proof}
We prove the claim about $\ont$; the claim about $\int$ is analogous. Unwinding definitions, we must show that $x \onto y$ if and only if $\{ x' : x \onto x' \} \supseteq {\{ x' : y \onto x' \}}$. This holds because $\ont$ is a partial order.
\end{proof}

We can complete the proof of Theorem \ref{2afs kappa} by showing that the same bijections turn $\t$ into $\t_L$ as well.   
By the definition of a factorization system, we have $\t = \Mult(\ont, \int)$.
By Propositions~\ref{pairs kappa dual} and~\ref{pairs is ws}, $\Pairs(\t)$ is a well separated $\kappa$-lattice so, by Proposition~\ref{well-sep kappa pairs}, we have $\t_L = \Mult(\ont_L, \int_L)$. 
Proposition~\ref{bijections turn} shows that $\ont = \ont_L$ and $\int = \int_L$, so we also have $\t = \t_L$ as required.
This completes the proof of Theorem~\ref{2afs kappa}.

We have now proved all the parts of Theorem~\ref{2afs iff ws kappa}: Theorem~\ref{kappa 2afs} combined with Proposition~\ref{well-sep kappa pairs}, shows that a well separated $\kappa$-lattice gives rise to a $2$-acyclic factorization system.
Theorem~\ref{2afs kappa} shows that every $2$-acyclic factorization system gives a well separated $\kappa$-lattice, and shows that these two constructions are related in the required manner.

\subsection{Finite semidistributive lattices} \label{semi sec} 
In this section, we recall and prove some basic facts about finite lattices that show that FTFSDL is a special case of Theorem~\ref{2afs iff ws kappa}.
We also use Proposition~\ref{posets only 2-afs} to give a formulation of the FTFSDL which references only $\ont$ and $\int$, and not $\t$.

\begin{prop}\label{Jx Mx}
If $L$ is a finite lattice and $x\in L$, then \[x=\Join \{ j \in \JIrr L : j \leq x \} =\Meet \{ m \in \MIrr L : m \geq x \}.\]
\end{prop}
\begin{proof}
We argue the first equality; the other is dual.
It is enough to show that $x$ is the join of a set of join-irreducible elements.
If $x$ is not join-irreducible, then it is the join of a set $S$ of elements strictly lower than $x$ in $L$.
By induction in
$L$, each element of $S$ is the join of a set of join-irreducible elements, and thus $x$ is also.
\end{proof}

We use the following well-known characterization of semidistributivity in finite lattices.
See, for example, \cite[Theorem~3-1.4]{semi} or \cite[Theorem~2.56]{FreeLattices}.
(Note that these  sources disagree on which map is called~$\kappau$ and which is called~$\kappad$.)
Since the proof is short, we include it. 

\begin{theorem}\label{semi char}
Suppose $L$ is a finite lattice.  
\begin{enumerate}[\qquad\rm1.]
\item \label{meet kappa}   
$L$ is meet semidistributive if and only if for every join-irreducible element $j\in\JIrr L$, the set ${\set{x\in L: j\meet x=j_*}}$ has a maximum element $\kappau(j)$.
\item \label{join kappa}
$L$ is join semidistributive if and only if for every meet-irreducible element $m\in\MIrr L$, the set ${\set{x\in L:m\join x=m^*}}$ has a minimum element $\kappad(m)$.
\end{enumerate}
If $L$ is semidistributive then $\kappau$ is a bijection from $\JIrr L$ to $\MIrr L$ with inverse~$\kappad$.
\end{theorem}

\begin{proof}
Suppose $L$ is meet semidistributive.  Suppose $j$ is a join-irreducible element.
Meet semidistributivity implies that the join $\Join\{ y : j \meet y = j_* \}$ is itself an element of $\{ y : j \meet y = j_* \}$.
Thus this set has a maximum element (i.e., an element greater than all the other elements in the set).  We denote it $\kappau(j)$.   

Conversely, suppose that ${\set{x\in L: j\meet x=j_*}}$ has a maximum element $\kappau(j)$ for every $j\in\JIrr L$.
Let $x$, $y$, and $z$ be elements of $L$ with $x\meet y=x\meet z$.
In any case, $x\meet(y\join z)\ge x\meet y$.
If $x\meet(y\join z)>x\meet y$, let $j$ be minimal among elements that are $\le x\meet(y\join z)$ and $\not\le x\meet y$.
(It is in choosing this minimal element $j$ that we make use of the hypothesis of finiteness.) 
If $j$ covers two distinct elements $k_1$ and $k_2$, then $j=k_1\join k_2$, but $k_1\join k_2\le x\meet y$ since minimality of $j$ implies $k_1\le x\meet y$ and $k_2\le x\meet y$.
Thus $j$ is join-irreducible.
Minimality of $j$ also implies that $j_* \le x\meet y\le y$.
Similarly, $j_* \le z$.
However, $y\join z\ge x\meet(y\join z)\ge j$, contradicting the existence of $\kappau(j)$.
We conclude that $x\meet(y\join z)=x\meet y$.

We have established the first numbered assertion; the second is dual.

We next check that the map $\kappa$ takes $\JIrr L$ to $\MIrr L$.
If $\kappau(j)$ is not \mi, let $X$ be a set of elements with $\Meet X=\kappau(j)$ but $\kappau(j)\not\in X$.
Then every element $x\in X$ has $x>\kappau(j)$, and thus $x\ge j$.
Thus $\kappau(j)=\Meet X\ge j$, contradicting the definition of $\kappau(j)$.
We conclude that $\kappau(j)\in\MIrr L$.

Finally, we must check that $\kappau$ and $\kappad$ are inverse.
Write $\kappau(j)^*$ for the unique element covering $\kappau(j)$.
If $L$ is semidistributive and $j\in\JIrr L$, then by definition of $\kappau(j)$, we have $j\le\kappau(j)^*$ and $j_* \le \kappau(j)$.
Thus $j$ is a minimal element of ${\set{x\in L:\kappau(j)\join x=\kappau(j)^*}}$, so that $\kappad(\kappau(j))=j$.
The dual argument shows that $\kappad$ maps $\MIrr L$ to $\JIrr L$ and that $\kappau(\kappad(m))=m$ for all $m\in\MIrr L$.
\end{proof}

Proposition~\ref{Jx Mx} and Theorem~\ref{semi char} combine to establish the following statement.
\begin{cor}\label{fin kappa}
A finite lattice is a $\kappa$-lattice if and only if it is semidistributive.
\end{cor}

\begin{prop}\label{fin kappa is ws}
If $L$ is a finite $\kappa$-lattice, then $L$ is well separated.
\end{prop}
\begin{proof}
Suppose $z_1\not\le z_2$ in $L$.
The set $\set{x\in L:x\le z_1,\,x\not\le z_2}$ is not empty, because it contains $z_1$.
Thus the set has a minimal element $j$.
The element $j$ is join-irreducible; if $j=\Join X$ with $j\not\in X$, then $z_2$ is an upper bound for $X$, yielding the contradiction $j\le z_2$.  
Also $j_*\le z_2$, so $j\meet z_2=j_*$, and thus $\kappa(j)\ge z_2$.
\end{proof}

Corollary \ref{fin kappa}, Proposition~\ref{fin kappa is ws}, and the finite case of Theorem~\ref{2afs iff ws kappa} combine to prove Theorem \ref{FTFSDL1}, the Fundamental Theorem of Finite Semidistributive Lattices.

Using Proposition \ref{posets only 2-afs}, we can restate the FTFSDL 
referring only to $\ont$ and $\int$.  
We can define $\Pairs(\Mult(\ont,\int))$ directly as the set of pairs $(X,Y)$ with $Y = X^\perp$ and $X={}^\perp Y$ in the sense of Proposition~\ref{posets only 2-afs}, partially ordered by containment in the first component, or equivalently reverse containment in the second component.

\begin{theorem}[FTFSDL, restated]\label{FTFSDL2}  
A finite poset $L$ is a semidistributive lattice if and only if it is isomorphic to $\Pairs(\ont,\int)$ for partial orders $\ont$ and $\int$ on a set $\Sha$, satisfying the conditions of Proposition~\ref{posets only 2-afs}.
In this case, $(\Sha,\ont,\int)$ is isomorphic to $(\JIrr L,\ont_L,\int_L)$, where $i \onto_L j$ if and only if $i\ge j$ in $L$ and $i\into_L j$ if and only if $\kappau(i)\ge\kappau(j)$ in $L$.
The map
\[x\mapsto( \{ j \in \JIrr L: j \leq x \},\  \kappad \left( \{ m \in \MIrr L : m \geq x \} \right) )\]
is an isomorphism from $L$ to $\Pairs(\ont_L,\int_L)$, with inverse $(X,Y)\mapsto\Join X=\Meet(\kappau(Y))$.
\end{theorem}

\section{The infinite case}\label{inf sec}
In this section, we discuss the infinite case further.  
We begin by explaining some choices we have made in the infinite case.

One choice we have made is to only consider lattices that are \emph{complete}.
This is for two reasons: 
First, one of our main motivations is the lattice of torsion classes (Section~\ref{rep ssec}) for a finite-dimensional algebra, and this lattice is complete.   
Second, we have seen in Proposition~\ref{step} that for any binary relation $\t$ on a set $\Sha$, the lattice $\Pairs(\t)$ is complete. 

We have similarly focused our attention on completely \ji and completely \mi elements.
Crucial here is the fact that $j$ is completely \ji if and only if there is an element $j_*$ such that $x<j$ if and only if $x\le j_*$.
If $j$ is \ji but not completely so, there is no such an element.

\subsection{Conditions on infinite complete lattices}  \label{inf summary}
We now discuss some conditions on complete lattices, each of which, in the finite case, is either equivalent to semidistributivity or trivially true.
The main theorem of this section describes the relationships between these conditions.
In Section~\ref{infinite main proof} we prove the main theorem, and in Section~\ref{counter sec}, we present numerous counterexamples to show that the claims of the theorem cannot be strengthened.

We begin with some discreteness conditions, which are easily seen to hold in all finite lattices. 
Recall that a complete lattice $L$ is \newword{bi-spatial} if for all $x \in L$, we have $x = \Join_{j \leq x,\ j \in \JIrr^cL} j = \Meet_{m \geq x,\ m \in \MIrr^cL} m$.
We say that $L$ is \newword{weakly atomic}, if, for all $x < y$, there exist $u$ and $v$ with $x \leq u \covered v \leq y$.    
We say that \newword{meets in $L$ are cover-determined} if whenever $x < y$ and $x \meet z < y \meet z$, there is a cover $x \leq u \covered v \leq y$ with $u \meet z < v \meet z$.
We say that \newword{joins in $L$ are cover-determined} if the dual condition holds.

A condition which implies all of the previous discreteness conditions is that a lattice be \newword{bi-algebraic}.
(See Propositions~\ref{cov det cov sep}, \ref{alg cov det}, and~\ref{alg gen}.)
An element $x$ of a lattice $L$ is called \newword{compact} if, for every subset $A$ of $L$, if $\bigjoin A \geq x$ then there is a finite subset $F$ of $A$ 
such that $\bigjoin F \geq x$. A lattice $L$ is called \newword{algebraic} if every element $x$ is the join of the set of compact elements which are $\leq x$. 
A lattice $L$ is called \newword{bi-algebraic} if it is algebraic and if the dual lattice is also algebraic. 
It is clear that any finite lattice is bi-algebraic. 
As we will discuss in Section~\ref{rep ssec},
the lattice of torsion classes for any algebra is bi-algebraic.

We continue by describing some conditions that, in the finite case, are equivalent to semidistributivity.
Recall that $L$ is join semidistributive if, for any nonempty finite subset $X$ of $L$ such that $x \join y = z$ for all $x \in X$, we have ${\left( \Meet_{x \in X} x \right) \join y = z}$.
A lattice $L$ is \newword{completely join semidistributive} if it is complete and if, for \emph{every} nonempty subset $X$ of $L$ such that $x \join y = z$ for all $x \in X$, we have ${\left( \Meet_{x \in X} x \right) \join y = z}$.   
\newword{Complete meet semidistributivity} is defined dually, and $L$ is called completely semidistributive if it is completely join semidistributive and completely meet semidistributive.
Our interest in complete semidistributivity arises in part from the lattice of torsion classes for a finite-dimensional algebra, as discussed in Section~\ref{rep ssec}.

Recall that a $\kappa$-lattice is a complete lattice $L$ that is bi-spatial, and has special bijections $\kappau$ and $\kappad$ between $\JIrr^cL$ and $\MIrr^cL$.
Recall also that Corollary~\ref{fin kappa} says that a finite lattice is a $\kappa$-lattice if and only if it is semidistributive.
Finally, recall that a $\kappa$-lattice is well separated if whenever  $z_1 \not\leq z_2$, there exists $j \in\JIrr^cL$ with $z_1 \geq j$ and $\kappa(j)\geq z_2$.

The following theorem relates all of these conditions on a complete lattice.

\begin{theorem} \label{infinite main}
Let $L$ be a complete lattice. 
The implications shown by solid arrows hold without additional hypotheses. 
The dashed implications hold under the additional hypothesis that $L$ is completely semidistributive.

%

\smallskip
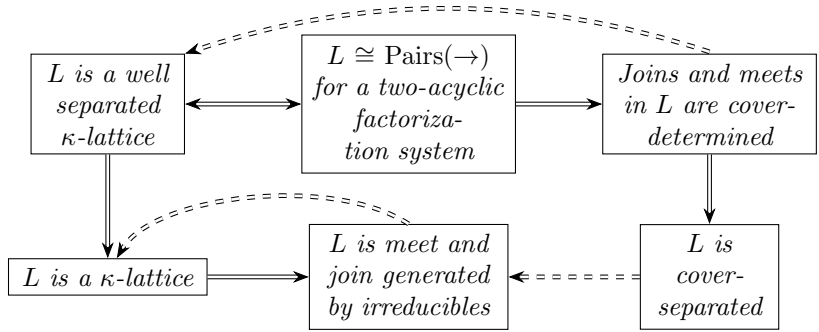
\begin{tikzpicture}[xscale=4.5,yscale=1.9,double distance=1pt,arrows={-Stealth}]
\node [shape=rectangle,draw,text width=3.0cm,align=center] (a) at (0,0) {$L\cong\Pairs(\t)$ for a two-acyclic factorization system} ;
\node [shape=rectangle,draw,text width=3.0cm,align=center] (b) at (0,-1) {$L$ is a well separated $\kappa$-lattice};
\node [shape=rectangle,draw,text width=3.0cm,align=center] (c) at (0,-1.7) {$L$ is a $\kappa$-lattice};
\node [shape=rectangle,draw,text width=3.0cm,align=center] (d) at (1,0) {Joins and meets in $L$ are cover-determined};
\node [shape=rectangle,draw,text width=3.0cm,align=center] (e) at (1,-1) {$L$ is weakly atomic};
\node [shape=rectangle,draw,text width=3.0cm,align=center] (f) at (1,-1.7) {$L$ is bi-spatial};
\node [shape=rectangle,draw,text width=2.4cm,align=center] (g) at (2,0) {$L$ is bi-algebraic};
\draw[Stealth-Stealth, double] (a) -- (b);
\draw[double] (b) -- (c);
\draw[double] (d) -- (e);
\draw[double, dashed] (e) -- (f);
\draw[double] (g) -- (d);
\draw[double] (g.south) to [bend left=22] ([xshift=0.0mm]f.east);
\draw[double] ([yshift=-0.1cm]a.east) to  [bend right=0]  ([yshift=-0.1cm]d.west);
\draw[double, dashed] ([yshift=0.1cm]d.west) to  [bend right=0]  ([yshift=0.1cm]a.east);
\draw[double] ([yshift=-0.06 cm]c.east) to [bend right=0] ([yshift=-0.06 cm]f.west);
\draw[double, dashed] ([yshift=0.06 cm]f.west)  to [bend right=0] ([yshift=0.06 cm]c.east);
\draw[-,dotted] ([xshift=-0.05cm,yshift=+0.1cm]a.north west) -- ([xshift=0.05cm, yshift=0.1cm]d.north east)-- ([xshift=0.05cm, yshift=-0.12cm]d.south east) --([xshift=0.05cm,yshift=-0.1cm]a.south east) --
  ([xshift=0.05cm,yshift=-0.1cm]b.south east) --([xshift=-0.05cm,yshift=-0.1cm]b.south west) -- ([xshift=-0.05cm,yshift=0.1cm]a.north west);
\end{tikzpicture}

\end{theorem}

\begin{remark}
For finite lattices, the three conditions in the left column are all equivalent to each other, and are equivalent to being semidistributive.
All finite lattices obey the conditions in the middle and right column.
\end{remark}

\begin{remark}
We remind the reader of Theorem~\ref{kappa 2afs}: If $L$ is a $\kappa$-lattice, then $L \cong \Pairs(\t_L)$ and $\Pairs(\t_L)$ obeys all the conditions of a two-acyclic factorization system except that some $\t$ arrows may not factor as an $\ont$ followed by an $\int$. We have not incorporated the statement $L \cong \Pairs(\t_L)$ into the diagram because $\t_L$ is only defined when $L$ is a $\kappa$-lattice.
\end{remark}

\begin{remark}
We regard completely semidistributive lattices obeying the three equivalent conditions inside the region marked with a dotted line in the top left of the diagram 
as the ``good'' lattices.
All finite semidistributive lattices are in this class, as are lattices of torsion classes for finite-dimensional algebras, and we hope to construct lattices completing weak order on infinite Coxeter groups which will likewise obey these conditions. 
\end{remark}

\begin{remark}
Unfortunately, in the infinite case, none of the conditions we have mentioned imply semidistributivity (complete or otherwise).
See Example~\ref{obnoxious} for a truly frustrating counterexample.
We do not have a good replacement hypothesis which would imply semidistributivity.
\end{remark}

%
%

\subsection{Proof of Theorem~\ref{infinite main}} \label{infinite main proof}

We now establish the implications of Theorem~\ref{infinite main}.
We begin with the arrows which do not require complete semidistributivity.
We already established in Theorem~\ref{2afs iff ws kappa} that $L$ is a well separated $\kappa$-lattice if and only if it is isomorphic to $\Pairs(\t)$ for a two-acyclic factorization system $(\Sha, \to, \ont, \int)$.

\begin{proposition} \label{prop-cov-det} Let $(\Sha,\t,\ont,\int)$ be a two-acyclic factorization
  system. Then joins and meets in $\Pairs(\t)$ are cover-determined.
\end{proposition}

\begin{proof}
We prove the statement for meets. 
 Let $X<Y$ in $\Pairs(\t)$ and let $Z \in \Pairs(\t)$ be such that $X \meet Z < Y \meet Z$. 
 We first prove the result in the special case that $Z \leq Y$, so $Y \meet Z = Z$.
In this case, all of the objects $X$, $Y$, $Z$ and $X\meet Z$ lie in the interval $[X \meet Z, Y]$. 
By Theorem~\ref{thm-intervals}, this interval  also corresponds to a two-acyclic factorization system. 
We therefore may assume that $X \meet Z = (\emptyset,\Sha)$ and $Y = (\Sha,\emptyset)$. 
The hypothesis $X \meet Z < Y \meet Z$ now simplifies to $Z\neq(\emptyset,\Sha)$. 
We identify each pair with its first element, so our hypotheses now are that $X$, $Z$ are closed sets with $X \cap Z = \emptyset$ and $Z \neq \emptyset$. We want to show that there is a cover $X \leq C_1 \covered C_2$ with $C_1 \cap Z \subsetneq C_2 \cap Z$.

Since $Z \neq \emptyset$, we can find $p \in Z$. Since $X \cap Z = \emptyset$, we have $p \not\in X$ and thus there is $r \in X^{\perp}$ with $p \to r$. 
Factor this arrow as $p \onto q \into r$. Since $Z$ is an $\ont$-downset and $X^{\perp}$ is an $\int$-upset, we have $q \in X^{\perp} \cap Z$. 
Recall the notations $F(q) = \{ s\in\Sha : s \into q \}$ and $F^*(q) = F(q) \setminus \{ q \}$. 
We have $X \subseteq {}^{\perp} F(q) \covered {}^{\perp} F^*(q)$. 

Clearly, $q \not \in {}^{\perp} F(q)$ so $q \not \in {}^{\perp} F(q) \cap Z$.
Since $F(q) \setminus F^*(q) = \{ q \}$, we have $q \in  {}^{\perp} F^*(q)$, so  $q  \in {}^{\perp} F^*(q) \cap Z$.
This shows that $ {}^{\perp} F(q) \cap Z \neq {}^{\perp} F^*(q) \cap Z$, so ${}^{\perp} F^*(q) \covered {}^{\perp} F(q)$ is the desired cover.
This completes the proof in the case that $Z \leq Y$.
 
 We now tackle the general case.  Let $X<Y$ in $\Pairs(\t)$ and let $Z \in \Pairs(\t)$ be such that $X \meet Z < Y \meet Z$. 
 Put $Z' = Y \meet Z$. Note that $Z' \leq Y$ and note that $X \meet Z' = X \meet Y \meet Z = X \meet Z < Y \meet Z = Y \meet Z'$, so the hypotheses of the Proposition apply to $(X, Y, Z')$. So, by our earlier work, there is a cover $X \leq C_1 \covered C_2 \leq Y$ with $C_1 \meet Z' < C_2 \meet Z'$. Suppose for the sake of contradiction that $C_1 \meet Z = C_2 \meet Z$. Then $C_1 \meet Z' = C_1 \meet Z \meet Y = C_2 \meet Z \meet Y = C_2 \meet Z'$, which is indeed a contradiction.
 
%
%

\end{proof}

\begin{prop}\label{cov det cov sep}
If either joins in $L$ are cover-determined or meets in $L$ are cover-determined, then $L$ is weakly atomic.
\end{prop}

\begin{proof} 
We consider the case that meets are cover-determined; the case of joins is dual.
Let $x < y$ and let $1$ be the maximum element of $L$. 
Since meets are cover-determined and $x \meet 1 < y \meet 1$, there exist $u$ and $v$ with $x \leq u \covered v \leq y$. 
\end{proof}

\switchmargin
\begin{prop}\label{alg cov det}
If $L$ is algebraic then meets in $L$ are cover-determined. If $L$ is bi-algebraic then meets and joins in $L$ are cover-determined. 
\end{prop}

\begin{proof}
Suppose that $x < y$ and $x \meet z < y \meet z$. 
Since $L$ is algebraic, $y \meet z$ is the join of the compact elements below it, so there must be a compact element $k$ which is below $y \meet z$ and not below $x \meet z$. 
Since $z \geq y \meet z \geq k$ and $x \meet z \not\geq k$, we must have $x \not\geq k$.

Let $P$ be the set $\{ w : x \le w\le y\text{ and } w \not\geq k \}$. 
The set $P$ is nonempty, because we checked in the previous paragraph that $x \not\geq k$ and, clearly, $x\le x\le y$.
Let $C$ be a totally ordered subset of $P$. 
We claim that $\bigjoin C \in P$. 
It is obvious that $x\le\bigjoin C \le y$, so we just need to check that $\bigjoin C \not\geq k$. 
By the compactness condition, if $\bigjoin C \geq k$, then there is some finite subset $F$ of $C$ with $\bigjoin F = k$.  
But $\bigjoin F = f$ where $f$ is the largest element of $F$, and $f \in F \subseteq P$ so, by definition, $f \not\geq k$. 
This contradiction completes the verification that $\bigjoin C \in P$.

The argument of the previous paragraph shows that $P$ is a nonempty poset in which every totally ordered subset has an upper bound.  
Thus, by Zorn's lemma, there is an element $u \in P$ such that any $w>u$ is not in $P$. 
Put $v = u \join k$ and observe that $x\le u < v\le y$. 
We claim that $u \covered v$. If not, suppose that $u < w < v$. 
Then $w \not \in P$ and, since $x\le u<w<v\le y$, we must have $w \geq k$. 
But then $w$ is an upper bound for $u$ and $k$ which is smaller than $u \join k=v$, a contradiction.

Finally, $v\wedge z \geq k$, but $u\wedge z\not\geq k$, so $u\wedge z < v\wedge z$, as desired. \end{proof}

The following is~\cite[Theorem 1-4.25]{CLaD}.

\begin{prop}\label{alg gen}
If $L$ is algebraic then $L$ is dually spatial. If $L$ is bi-algebraic, then $L$ is bi-spatial.  
\end{prop}

The remaining undashed implications say that a well separated $\kappa$-lattice is a $\kappa$-lattice and that a $\kappa$-lattice is bi-spatial.
Both of these implications are true by definition.

We now turn to the implications which require complete semidistributivity.
We begin with the bottom row.

\begin{prop}\label{gen semi kappa}
Let $L$ be a completely semidistributive lattice which is bi-spatial. 
Then $L$ is a $\kappa$-lattice.
\end{prop}

\begin{proof}
This is proved exactly as in Theorem~\ref{semi char}.
\end{proof}

Before proving the other dashed implications, 
we need the following lemma.
This result is essentially \cite[Proposition~2.20]{DIRRT}.
\begin{lemma}\label{ji lab}
If $L$ is a completely semidistributive lattice and $u\covered v$ in $L$, then the set $\{ t \in L: t \join u = v \}$ contains a minimum element $\ell$, which is completely \ji and has $\ell_*\le u$.
\end{lemma}
\begin{proof}
Let $T=\{ t \in L: t \join u = v \}$.
Since $L$ is completely semidistributive, the element $\ell =\Meet T$ has $\ell\join u = v$, so it is the desired minimum element of $T$.
Suppose $\ell= \Join S$ for some set $S\subseteq L$.
We must show that $\ell\in S$.
In order to do this, we will show that $\Join (S \setminus \{ \ell \}) \leq u$; since $\ell \not \leq u$, this shows that $\Join(S \setminus \set{ \ell }) \neq \Join S$. 

If $s\in S\setminus\set{\ell}$, then $u \leq s \join u \leq \ell \join u = v$  so, since $u \covered v$, we must have $s \join u = u$ or $s \join u =v$. 
If $s \join u = v$ then $s$ is an element of $T$ that is strictly less than $\ell=\Meet T$.
By this contradiction, we see that $s \join u = u$.
We have shown that $s\le u$ 
for every $s\in S\setminus\set{\ell}$ and thus  $\Join (S \setminus \{ \ell \}) \leq u$ as promised.
We have shown that $\ell$ is completely \ji.

We have $u\le\ell_*\join u\le\ell\join u=v$, but the second inequality must be strict by the definition of $\ell$.
Since $u\covered v$, we see that $\ell_*\join u=u$, so that $\ell_*\le u$.
\end{proof}

\begin{remark}
The map in Lemma \ref{ji lab} from covers of a complete semidistributive
lattice to completely join-irreducible elements is known as the \newword{join-irreducible labelling}. It is well-known in the finite case. An analogous result in the
representation-theoretic setting was established in \cite{BCZ} (combining
Theorems 1.0.2, 1.0.3, and 1.0.5). \end{remark}

We next establish the downward dashed implication.

\begin{prop} \label{Jx Mx Infinite}
If $L$ is a completely semidistributive, weakly atomic lattice and $x\in L$, then \[x=\Join \{ j \in \JIrr^cL : j \leq x \} =\Meet \{ m \in \MIrr^cL : m \geq x \}.\]
\end{prop}

\begin{proof}
We prove that $x=\Join \{ j \in \JIrr^cL : j \leq x \}$.
The other equality is dual.

Let $x'=\Join \{j \in \JIrr^cL  :  j\leq x \}$.
Then $x' \leq x$, so suppose for the sake of contradiction that $x'<x$.
Since $L$ is weakly atomic, there exist $u$ and $v$ with $x' \leq u \covered v \leq x$. 
Let $\ell$ be the minimum element of $\{ t : t \join u = v \}$, which exists and is completely \ji by Lemma~\ref{ji lab}. 

Since $\ell\join u = v\neq u$, we see that $\ell \not\leq u$ and hence $\ell\not \leq x'$. 
But $\ell$ is completely join-irreducible and $\ell\le v\le x$, so we have contradicted the definition of $x'$ as ${x'=\Join \{ j \in \JIrr^c L:  j \leq x \}}$.  We therefore reject the supposition that ${x'<x}$.
\end{proof}

Our final task is to prove the dashed implication in the top row.
In fact, we will prove the following proposition.
\begin{prop}\label{kappa 2afs j/m det}   
Suppose $L$ is a completely semidistributive lattice. 
If either joins are cover-determined or meets are cover-determined, then $L$ is a well separated $\kappa$-lattice.
\end{prop}

To prove the proposition, we will use an alternative characterization of the well separated property.
In a $\kappa$-lattice $L$, suppose $i,j\in\JIrr^cL$.
If there exists $\ell\in\JIrr^cL$ with $i\ge\ell$ and $\kappa(\ell)\ge\kappa(j))$, then $i\not\le\kappa(j)$.
(If $i\le\kappa(j)$, then $\ell\le i\le\kappa(j)\le\kappa(\ell)$, contradicting the definition of $\kappa(\ell)$.)  
We say that $L$ is \newword{weakly separated} if the converse holds:  
If $i,j\in\JIrr^cL$ have $i\not\le\kappa(j)$, then there exists $\ell\in\JIrr^cL$ with $i\ge\ell$ and $\kappa(\ell)\ge\kappa(j))$.

\begin{lemma}\label{weak-sep} A $\kappa$-lattice $L$ is well separated if and only if it is weakly separated.
\end{lemma}
\begin{proof}
Clearly any well separated $\kappa$-lattice is weakly separated.  
We show the converse. 
Suppose $L$ is weakly separated and $z_1\not\le z_2$ in $L$.  
We will show that there exists $\ell\in\JIrr^cL$ with $z_1\ge\ell$ and $\kappa(\ell)\ge z_2$.
Since $L$ is spatial (as part of the definition of a $\kappa$-lattice), there exists $i\in\JIrr^cL$ with $i\le z_1$ but $i\not\le z_2$.
Since $i\not\le z_2$ and because $L$ is dually spatial, there exists $j\in\JIrr^cL$ such that $z_2\le\kappa(j)$ and $i\not\le\kappa(j)$.
By weak separation, there exists $\ell\in\JIrr^cL$ with $i\ge\ell$ and $\kappa(\ell)\ge\kappa(j)$.
Thus $z_1\ge i\ge\ell$ and $\kappa(\ell)\ge\kappa(j)\ge z_2$ as desired.
\end{proof}

\begin{proof}[Proof of Proposition~\ref{kappa 2afs j/m det}]
We consider the case that joins are cover-determined; the case of meets is dual. 
The implications we have already established show that $L$ is a $\kappa$-lattice.
It remains to establish well separation.
By Lemma~\ref{weak-sep}, it is enough to establish weak separation.

Suppose $i,j\in\JIrr^cL$ have $i\not\leq\kappau(j)$. 
Then $i\join\kappau(j)>\kappau(j)$ and $(i\meet\kappau(j)) \join\kappau(j)=\kappau(j)$.  
Since joins are cover-determined in $L$, there exist $u,v\in L$ with $i\meet\kappau(j)\le u \covered v\leq i$ and $u \join\kappau(j)\neq v\join\kappau(j)$. 
By Lemma~\ref{ji lab}, the set $\{ s\in L : s \join u = v \}$ has a minimum element $\ell$, which is completely join-irreducible and has $\ell_*\le u$.
We have $\ell \leq v \leq i$.

Suppose for the sake of contradiction that $\kappau(\ell) \not\geq \kappa(j)$.
We claim first that $\kappau(j)\join\ell_*\ge\ell$.
Indeed, if $\kappau(j)\join\ell_*\not\ge\ell$, then $(\kappau(j)\join\ell_*)\meet\ell<\ell$, so $(\kappau(j)\join\ell_*)\meet\ell=\ell_*$.
By the definition of $\kappau(\ell)$, we have $\kappau(j)\join\ell_*\le\kappau(\ell)$, contradicting $\kappau(\ell) \not\geq \kappa(j)$ and thus proving the claim.
By the claim, $\kappa(j)\join u\ge\ell$ as well.
Since $\ell\join u=v$, also $\kappa(j)\join v=\kappa(j)\join\ell\join u$,
which equals $\kappa(j)\join u$ because $\kappa(j)\join u\ge\ell$.
But we already know $u \join\kappau(j)\neq v\join\kappau(j)$, and by this contradiction, we see that $\kappau(\ell)\geq \kappa(j)$.
We have proven that $L$ is weakly separated, as desired.
\end{proof}

\subsection{Counterexamples}\label{counter sec}  
We now present a number of counter-examples, to sharpen the distinctions between the various conditions on complete lattices, and to demonstrate that the results in Theorem~\ref{infinite main} cannot be strengthened. We start with three examples to illustrate the relationship between semidistributivity, complete semidistributivity and being a well separated $\kappa$-lattice. 
We note that none of Examples~\ref{SDNotKappa}, \ref{NotCompleteSD} and~\ref{CoversNotSemiD} are algebraic.
Later, in Example~\ref{obnoxious}, we will give an example of a well separated $\kappa$-lattice which is bi-algebraic (and hence obeys all the other discreteness conditions by Theorem~\ref{infinite main}) but is still not semi-distributive. 

\begin{eg}[A semidistributive lattice that is not completely semidistributive and is not a $\kappa$-lattice]\label{SDNotKappa}
  Consider the lattice with elements $0$, $1$, $X$ and $Y_i$ for
 $i\in \mathbb Z$, 
  and relations that $0$ is the minimum, $1$ is the maximum 
and $Y_i \leq Y_j$ for $i \leq j$. This lattice is complete and semidistributive. 
For each $Y_i$, we have $X \join Y_i = 1$. However, $X \join \Meet_i Y_i = X \join 0 = X$.
Also, $X$ is completely join irreducible with $X_* = 0$, but there is no maximal element $\kappa(X)$ in the set $\{ Y : X \meet Y = 0 \}$. 
\end{eg}

\begin{eg}[A semidistributive $\kappa$-lattice that is not completely semidistributive  and not well separated] \label{NotCompleteSD}
Consider the lattice whose elements $L$ are called $0$, $1$, $X_i$ and $Y_i$, for $i \in \ZZ$, with relations that $0$ is the minimum, $1$ is the maximum, and $X_i \geq X_j$ and $Y_i \geq Y_j$ for $i \geq j$.  This lattice is complete and semidistributive.
However, it is not completely semidistributive for the same reason as in Example
\ref{SDNotKappa}: any $X_i$ can play the role of $X$ from that example.
Every element is both completely \ji and completely \mi except $0$ and $1$.
Also, $0=\Meet L$ and $1=\Join L$.
We have $\kappau(X_i) = X_{i-1}$ and $\kappau(Y_i) = Y_{i-1}$, and $\kappad$ is the inverse map, so $L$ is a $\kappa$-lattice. 
Also, $X_0 \not\geq Y_0$, yet there is no completely \ji element $J$ with $X_0 \geq J$ and $\kappa(J) \geq Y_0$, so $L$ is not well separated.
\end{eg}

\begin{eg}[A $\kappa$-lattice that is not semidistributive and not well separated]\label{CoversNotSemiD}
  Consider the lattice $L$ with elements $0$, $1$ and $X_i$, $Y_i$ and $Z_i$ for$i \in \mathbb Z$,  
  and relations that $0$ is the minimum, $1$ is the maximum, and $X_i \leq X_j$, $Y_i \leq Y_j$ and $Z_i \leq Z_j$ for $i \leq j$. As in Example \ref{NotCompleteSD}, every element is both completely \ji and completely \mi except $0$ and $1$, and 
we have $\kappau(X_i) = X_{i-1}$, $\kappau(Y_i) = Y_{i-1}$ and $\kappau(Z_i) = Z_{i-1}$. 
Thus $L$ is a $\kappa$-lattice.
However, $X_0 \meet Y_0 = X_0 \meet Z_0 = 0$ while $X_0 \meet (Y_0 \join Z_0) = X_0 \meet 1 = X_0$, so $L$ is not semidistributive.
This lattice is not well separated for the same reason as Example~\ref{NotCompleteSD}.
\end{eg}

\begin{eg}[$\kappa$-lattices not isomorphic to $\Pairs(\t)$ for a two-acyclic factorization system] \label{NotPairs}
Suppose $L$ is the lattice in Example~\ref{NotCompleteSD} or the lattice in Example~\ref{CoversNotSemiD}.
Thus $L$ is a $\kappa$-lattice.
By Theorem~\ref{2afs iff ws kappa}, if $L$ is isomorphic to $\Pairs(\t)$ for a two-acyclic factorization system $(\Sha, \t, \ont, \int)$, then this system is isomorphic to $(\JIrr^cL,\t_L,\ont_L,\int_L)$.
However, $\t_L$ is not $\Mult(\ont_L, \int_L)$.
Indeed, we have $\kappau(X_1) = X_0 \not\geq Y_1$ so $X_1 \to Y_1$, but the only compositions $X_1 \onto P \into Q$ are those of the form $X_1 \onto X_j \into X_k$ for $1 \geq j \geq k$, so we cannot factor $X_1 \to Y_1$ into an $\ont$ arrow and an $\int$ arrow.
\end{eg}

We next give an example where meets and joins are cover-determined, but the lattice is not dually spatial because it has no completely meet irreducible elements.  
It is therefore also not a $\kappa$-lattice (well separated or otherwise).
By Theorem~\ref{infinite main}, this example cannot be completely semidistributive; surprisingly, it is distributive!

\begin{eg}[A lattice where meets and joins are cover-determined, but $\MIrr^cL$ is empty] \label{vector space example}
Let $\Omega$ be an infinite set.
Let $L$ consist of the finite subsets of $\Omega$, as well as the set $\Omega$ itself, with $L$ ordered by containment.
It is easy to check that $L$ is a complete lattice and meets and joins are cover-determined.
However, $L$ has no meet irreducible elements, so it is not dually spatial.
\end{eg}

For completely semidistributive lattices, Theorem~\ref{infinite main} tells us that having joins and meets cover-determined implies being weakly atomic, which in turn implies being bi-spatial.
We now present two very similar examples, showing that these implications cannot be reversed.
As Theorem~\ref{infinite main} shows must be the case, these are $\kappa$-lattices but are not well separated. 

\begin{figure}[h]
\[ \begin{gathered}\xymatrix@R0.5pc@C0.5pc{
&&& y^{++} \ar@{-}[dl] & \\
&&y^+ \ar@{-}[dr]  \ar@{.}[dl] && \\
&x^{++} \ar@{-}[dl] & & y^{\circ} \ar@{.}[ddll]  \ar@{-}[dr] &\\
x^+ \ar@{-}[dr] & & && y^-  \ar@{-}[dl]    \\
&  x^{\circ}  \ar@{-}[dr] && y^{--} \ar@{.}[dl]  & \\ 
&& x^{-} \ar@{-}[dl] && \\
& x^{--}  &&& \\
 }\end{gathered}  \]
\caption{Portions of the lattices from Examples~\ref{comp SD not cov sep} and~\ref{comp SD not cov det}} \label{induced}
\end{figure} 

\begin{eg}[A completely semidistributive $\kappa$-lattice that is bi-spatial but not weakly atomic] \label{comp SD not cov sep}
We define a poset $L$ whose ground set is $5$ copies of the interval $[0,1]$ in the real numbers, with elements denoted by formal symbols $x^{++}$, $x^+$, $x^{\circ}$, $x^-$, and $x^{--}$ for each $x\in[0,1]$.  
For any $0 \leq x < y \leq 1$, the poset $L$ induces the order depicted on the $10$ elements shown in Figure~\ref{induced}.   
The solid lines in the figure represent cover relations in $L$; the dashed lines are order relations in $L$ that are not covers.
It is easy to check from Figure~\ref{induced} that $L$ is a lattice. 
To see that it is actually a complete lattice, note that any set of the form $\set{x^{++}:x\in A}$ has a join (either $(\sup A)^{++}$ or $(\sup A)^+$, where $\sup A$ is the supremum), and similarly, sets consisting of elements $x^+$, $x^{\circ}$, $x^-$ or $x^{--}$ have joins.
Thus we can compute the join of any subset of $L$ as the join of at most five elements, and similarly for meets.

Complete meet semidistributivity is verified in the following table.
For any $u<v$, the table gives the unique maximal $p$ with $v \meet p = u$, where $u$ is given in the row headings and $v$ in the column headings, assuming $x<y$.  
Complete join semidistributivity is dual.
\[
\begin{array}{|r||c|c|c|c|c|c|c|c|c|c|}
\hline 
 & x^{--} & x^{-} & x^{\circ} & x^{+} & x^{++} & y^{--} & y^{-} & y^{\circ} & y^{+} & y^{++} \\ \hline \hline
x^{++} &&&&&1^{++}&&&&x^{++}&x^{++} \\ \hline
x^{+} &&&&1^{++}&x^{+}&&&&x^{+}&x^{+} \\ \hline
x^{\circ} &&&1^{++}&1^{\circ}&1^{\circ}&&&x^{++}&x^{\circ}&x^{\circ} \\ \hline
x^{-} &&1^{++}&1^{-}&1^{-}&1^{-}&x^{++}&x^{++}&x^{-}&x^{-}&x^{-} \\ \hline
x^{--} &1^{++}&x^{--}&x^{--}&x^{--}&x^{--}&x^{--}&x^{--}&x^{--}&x^{--}&x^{--} \\ \hline
\end{array}
\]
It is easy to check that it is bi-spatial; this,
in combination with the fact that it is completely semidistributive, allows
us to conclude by Theorem \ref{infinite main} that it is a $\kappa$-lattice.

However, this lattice is not weakly atomic: There are no covers in the interval $[0^{\circ},1^{\circ}]$. It follows from Theorem \ref{infinite main} that it is not well separated. 
\end{eg}

\begin{eg}[A completely semidistributive $\kappa$-lattice which is weakly atomic but where joins and meets are not cover-determined] \label{comp SD not cov det} 
Take Example~\ref{comp SD not cov sep} and remove the elements of the form $x^{\circ}$. 
As before, we can verify that this is a complete, completely semidistributive lattice, and it is now weakly atomic. Theorem \ref{infinite main} therefore shows that it is a $\kappa$-lattice.

However, neither meets nor joins are cover-determined. 
For all $x$, we have $x^- \join 0^+ = x^{--} \join 0^+ = x^+$. 
So, for $0 \leq x<z \leq 1$, we have $x^- \join 0^+ < z^{--} \join 0^+$ but, for any cover $x^- < y^{--} \covered y^- < z^{--}$, we have $y^- \join 0^+ = y^{--} \join 0^+$. Since meets and joins are not cover-determined, Theorem~\ref{infinite main} shows that the lattice is not well separated. 
\end{eg}

Well separated $\kappa$-lattices obey all the discreteness conditions in the middle column of Theorem~\ref{infinite main}. However, they need not be algebraic, as the next example shows. 

\begin{eg}[A well separated, completely semidistributice $\kappa$-lattice which is not algebraic] \label{EGKappaNotPairs}  
Consider the following factorization system. 
The elements of $\Sha$ are called $\{ x_1, x_2, x_3, \ldots, y, z \}$.  
For $i \geq j$, we have $x_i \to x_j$ and, for all $i$, we have $y \to x_i$. 
Finally, we have $z \to y$ and the reflexive relations at $y$ and $z$. 
The relations $\ont$ and $\int$ are the transitive closure of the relations shown on the left hand side of Figure~\ref{PairsNotAlgebraic}.  
In words, every $\t$-arrow is also an $\ont$-arrow except that $z\not\onto y$ and every $\t$-arrow is also an $\int$-arrow except that $y\not\into x_i$ for all $i$. (Note that, unlike $\ont$ and $\int$, $\t$ is \emph{not} transitive, since $z\not\to x_i$.)
The closed sets are $X[j] := \{ x_i : i \leq j \}$, $Z[j] := X[j] \cup \{ z \}$, $X[\infty] := \{ x_i \}$, $Y = X[\infty] \cup \{ y \}$ and the whole set $\Sha$, so the lattice $\Pairs(\t)$ is depicted on the right hand side of Figure~\ref{PairsNotAlgebraic}.

We see that $\bigjoin_{j<\infty} Z[j] = \Sha \geq Y$. However, for any finite set of indices $\{ j_1, \ldots, j_r \}$, we have $\bigjoin_{s=1}^r Z[j_s] = Z[\max(j_1, j_2, \ldots, j_r)] \not\geq Y$. So $Y$ is not compact. The join of all the compact elements $\leq Y$ is $X[\infty]$, so $\Pairs(\t)$ is not algebraic.
It is also straightforward to check that $\Pairs(\t)$ is completely semidistributive.
%
\end{eg}

\begin{figure}
\mbox{
  \begin{tikzpicture}[yscale=1.5]
\node (z) at (0,0) {$z$};
\node (y) at (0,-1) {$y$};
\node (dots) at (0,-2) {$\vdots$};
\node (x3) at (0,-3) {$x_3$};
\node (x2) at (0,-4) {$x_2$};
\node (x1) at (0,-5) {$x_1$};
\draw[right hook->]  (z) to (y);
\draw[->>]  (y) to (dots);
\draw[right hook->>]  (dots) to (x3);
\draw[right hook->>]  (x3) to (x2);
\draw[right hook->>]  (x2) to (x1);
\end{tikzpicture}}
\qquad\qquad\quad
\mbox{
  \begin{tikzpicture}[yscale=1.8]
\node (Xinfty) at (0,-1) {$X[\infty]$} ;
\node (Y) at (1, -0.5) {$Y$};
\node (Zinfty) at (2,0) {$\Sha$} ;
\node (dots1) at (0,-1.5) {$\vdots$};
\node (dots2) at (2,-0.5) {$\vdots$};
\node (X2) at (0,-2) {$X[2]$} ;
\node (X1) at (0,-3) {$X[1]$} ;
\node (X0) at (0,-4) {$X[0]$} ;
\node (Z2) at (2,-1) {$Z[2]$} ;
\node (Z1) at (2,-2) {$Z[1]$} ;
\node (Z0) at (2,-3) {$Z[0]$} ;
\draw (Zinfty) to (Y) ;
\draw (Y) to (Xinfty) ;
\draw (Zinfty) to (dots2) ;
\draw (Xinfty) to (dots1) ;
\draw (dots2) to (Z2) ;
\draw (dots1) to (X2) ;
\draw (X2) to (X1) ;
\draw (X1) to (X0) ;
\draw (Z2) to (Z1) ;
\draw (Z1) to (Z0) ;
\draw (Z2) to (X2);
\draw (Z1) to (X1);
\draw (Z0) to (X0);
  \end{tikzpicture}
}
\caption{The factorization system and lattice $\Pairs(\t)$ in Example~\ref{EGKappaNotPairs}} \label{PairsNotAlgebraic}
\end{figure}

\begin{eg}[A completely semidistributive lattice which violates all other conditions in Theorem~\ref{infinite main}]\label{JIrr empty}  
Let $L$ be the interval $[0,1]$ in the real numbers with the standard total order. 
This is a complete and completely semidistributive (even distributive!) lattice, but it has no completely join-irreducible or completely meet-irreducible elements at all.
Thus $L$ is neither spatial nor dually spatial, and thus Theorem~\ref{infinite main} shows that this lattice also violates the other conditions there.
\end{eg}

We conclude with an example that shows that semidistributivity is independent of all of the other conditions in Theorem~\ref{infinite main}.
This example is particularly frustrating, as we do not know what additional condition on the two-acyclic factorization system would rule out this lattice.

\begin{eg}[A two-acyclic factorization system $(\Sha, \t, \ont, \int)$ such that the $\kappa$-lattice $\Pairs(\t)$ is not semidistributive.]  \label{obnoxious}  
We describe a lattice $L$ by giving a two-acyclic factorization system so, by Theorem~\ref{infinite main}, it is a well separated $\kappa$-lattice, joins and meets are cover determined, it is cover separated and bi-spatial. 
We will also check that $L$ is bi-algebraic. However, $L$ is not semi-distributive.

Let $\Sha$ consist of elements $y$, $z$, and $x_i$ for $i\in\ZZ$. 
Define  $\to$ by $x_i \to x_j$ for $i \geq j$, by $y \to x_{2i}$ and $z\to x_{2i+1}$ for $i\in\ZZ$, as well as $y \to y$ and $z \to z$. 
We define $(\ont,\int)$ to be $\Fact(\t)$. Explicitly, 
\begin{itemize}
\item $x_i \onto x_j$ if and only if $i \geq j$, 
\item $x_i \into x_j$ if and only if $i \geq j$ and $i \equiv j \bmod 2$,
\item $y \into x_{2i}$ and $z\into x_{2i+1}$ for all $i\in\ZZ$, and
\item $y \into y$, $y \onto y$, $z \into z$ and $z \onto z$.
\end{itemize}
The relations $\ont$ and $\int$ are shown in the left picture of Figure~\ref{non SD eg}.
\begin{figure}
\mbox{
  \begin{tikzpicture}[yscale=1.5]
  \node (y) at (0,0) {$y$};
  \node (z) at (1.5,0) {$z$};
  \node (dl) at (0, -1) {$\mvdots$};
  \node (dr) at (1.5,-1) {$\mvdots$};
  \node (x2) at (0,-2) {$x_2$};
  \node (x1) at (1.5,-3) {$x_1$};
  \node (x0) at (0,-4) {$x_0$};
  \node (xm1) at (1.5,-5) {$x_{-1}$};
  \node (xm2) at (0,-6) {$x_{-2}$};
  \node (dbl) at (0,-7) {$\mvdots$};
  \node (dbr) at (1.5,-7) {$\mvdots$};
  \draw[right hook->] (y) -- (dl);
  \draw[right hook->] (z) -- (dr);
  \draw[right hook->] (dl) --(x2);
  \draw[right hook->] (dr) -- (x1);
  \draw[right hook->] (x2) -- (x0);
  \draw[right hook->] (x0) -- (xm2);
  \draw[right hook->] (xm2) -- (dbl);
  \draw[right hook->] (x1) -- (xm1);
  \draw[right hook->] (xm1) -- (dbr);
  \draw[->>] (x2) -- (x1);
  \draw[->>] (x1) -- (x0);
  \draw[->>] (x0) -- (xm1);
  \draw[->>] (xm1)--(xm2);
  \draw[->>] (dr.south west) -- (x2);
  \draw[->>] (xm2) -- (dbr);
  \end{tikzpicture}} \qquad\qquad
\mbox{
  \begin{tikzpicture}[xscale=1.7]
    \node (S) at (0,0) {$\Sha$};
    \node (Yi) at (-1.5,-1) {$Y[\infty]$};
    \node (Zi) at (1.5,-1) {$Z[\infty]$};
    \node (Xi) at (0,-2) {$X[\infty]$};
    \node  (ld) at (-1.5,-3) {$\mvdots$};
    \node  (cd) at (0,-3) {$\mvdots$};
    \node (rd) at (1.5,-3) {$\mvdots$};
    \node (X2) at (0,-4) {$X[2]$};
    \node (X1) at (0,-5) {$X[1]$};
    \node (X0) at (0,-6) {$X[0]$};
    \node (Xm1) at (0,-7) {$X[-1]$};
    \node (Xm2) at (0,-8) {$X[-2]$};
    \node (cld) at (0,-9) {$\mvdots$};
    \node (lld) at (-1.5,-9) {$\mvdots$};
    \node (rld) at (1.5,-9) {$\mvdots$};
    \node (Xmi) at (0,-11) {$X[-\infty]$}; 
    \node (Z1) at (1.5,-4) {$Z[1]$};
    \node (Y0) at (-1.5,-5) {$Y[0]$};
    \node (Zm1) at (1.5,-6) {$Z[-1]$};
    \node (Ym2) at (-1.5,-7) {$Y[-2]$};
    \node (Zm3) at (1.5,-8) {$Z[-3]$};
    \node (Ymi) at (-1.5,-10) {$Y[-\infty]=\{y\}$};
    \node (Zmi) at (1.5,-10) {$Z[-\infty]=\{z\}$};
    \draw (S)--(Yi) -- (ld) -- (Y0)--(Ym2) -- (lld) --(Ymi);
    \draw (Ymi) -- (Xmi);
    \draw (S)--(Zi) -- (rd) -- (Z1) -- (Zm1) -- (Zm3) -- (rld) -- (Zmi);
    \draw (Zmi) --(Xmi);
    \draw (Xi)--(cd) --(X2) --(X1) --(X0)--(Xm1)--(Xm2)--(cld)--(Xmi);
    \draw (Yi)--(Xi)--(Zi);
    \draw (ld)--(X2); \draw (Z1)--(X1);\draw (Y0)--(X0);\draw (Zm1)--(Xm1);
    \draw (Ym2)--(Xm2); \draw (Zm3)--(cld);
\end{tikzpicture}}
\caption{A factorization system for which $\Pairs(\t)$ is not semidistributive} \label{non SD eg}
  \end{figure}

We see that if $p \to q$, then either $p \into q$ or $p \onto q$ (or both), so $\Mult(\ont, \int) \supseteq \t$.
By Proposition~\ref{Mult Fact facts}.\ref{mf} we conclude that $(\Sha, \t, \ont, \int)$  is a factorization system. 
The union of the relations $\t$, $\ont$, and $\int$ is acyclic, so in particular $(\Sha, \t, \ont, \int)$ is two-acyclic. 

For $j\in\ZZ\cup\set{\pm\infty}$, we put $X[j] = \{ x_i : i \leq j \}$, $Y[j] =\{ y \} \cup \{ x_i : i \leq j \}$ and $Z[j] =\{ z \} \cup \{ x_i : i \leq j \}$. 
The closed sets are $X[i]$, $Y[2i]$ and $Z[2i+1]$, for $i\in\ZZ\cup\set{\pm\infty}$, as well as the whole set $\Sha$. 
The lattice of closed sets is shown as the right picture of Figure~\ref{non SD eg}.

The compact elements are all elements except for $\Sha$, $X[\infty]$, $Y[\infty]$ and $Z[\infty]$; it is easy to check that this makes the lattice algebraic. Likewise, the co-compact elements are all but $X[-\infty]$, $Y[-\infty]$ and $Z[-\infty]$ and this lattice is co-algebraic.

Finally, we check that $L$ is not semidistributive. 
We have $X[0] \meet Y[-\infty] = X[0] \meet Z[-\infty] = X[-\infty]$, but $X[0] \meet (Y[-\infty] \join Z[-\infty]) = X[0] \meet \Sha = X[0]$. 
\end{eg}

\section{Factorization systems for intervals} \label{intervals sec}
Let $L$ be a finite semidistributive lattice. 
By the FTFSDL, it can be described as $\Pairs(\t)$ for a finite two-acyclic factorization system $(\Sha,\t,\ont,\int)$.  
An interval in $L$ is also a finite semidistributive lattice, so it is also describable by a finite two-acyclic factorization system.
In this section, we show that this can be done very explicitly.
More generally, we show that, if $L \cong \Pairs(\t)$ for a possibly infinite, two-acyclic factorization system $(\Sha,\t,\ont,\int)$, then every interval of $L$ can be described as $\Pairs(\t')$ for some two-acyclic factorization system $(\Sha',\t',\ont',\int')$.

We begin by describing lower intervals in $\Pairs(\t)$.
We use notation such as $\t|_X$ for the restriction of relations to subsets $X\subseteq\Sha$.

\begin{prop}\label{prop-lower-intervals}  
Suppose $(\Sha,\t,\ont,\int)$ is a two-acyclic factorization system and let $(X,Y)\in\Pairs(\t)$. 
Then $(X, \t|_X, \ont|_X, \int')$ is a two-acyclic factorization system, where $\int'$ is given by $\Fact(\t|_X)=(\ont|_X,\int')$.
The map $(U,V) \mapsto (U, V \cap X)$ is an isomorphism from the interval $\bigl[(\emptyset, \Sha), (X,Y)\bigr]$  in $\Pairs(\t)$ to the lattice $\Pairs(\t|_X)$.
The closure operators associated to $(\Sha,\t,\ont,\int)$ and $(X, \t|_X, \ont|_X, \int')$ agree on subsets of $X$.
\end{prop}

\begin{proof}  
Write $\Fact(\t|_X)=(\ont', \int')$.
We have $(\ont', \int')\supseteq (\ont|_X, \int|_X)$ because $(\Sha,\t,\ont,\int)$ is a factorization system.
We now check that $\ont|_X \supseteq \ont'$. 
Suppose that $p$ and $q \in X$ with $p \not\onto q$; we will check that $p \not \! \ont' q$. 
Then there is some $s \in \Sha$ with $p \not\to s$ and $q \to s$. 
Factor $q \to s$ as $q \onto r \into s$; since $X$ is an $\ont$-downset by Proposition~\ref{TorsionImpliesDownset},
we have $r \in X$. 
If $p \to r$ then there is a factorization $p\onto t\into r$ and thus $p\onto t\into s$, contradicting the assumption that $p\not\to s$.  
We conclude that $p \not\to r$. 
Now $r \in X$ with $p \not\to r$ and $q \to r$, so $p\not\onto' q$.

To show that $(X, \t|_X, \ont|_X, \int')$ is a factorization system, it remains to show that $\Mult(\ont|_X,\int')=\t|_X$.
By Proposition~\ref{Mult Fact facts}.\ref{mf}, we have $\Mult(\ont|_X,\int')\subseteq\t|_X$.
Suppose that $x$ and $z \in X$ and $x \to z$. 
Then $x \onto y \into z$ for some $y \in \Sha$. 
Since $X$ is a downset for $\ont$, we have $y \in X$.
Using our above observation that $\int'\supseteq\int|_X$, we have $x \onto y \into' z$, so $x$ is related to $z$ by $\Mult(\ont|_X,\int')$.

We now check the two-acyclicity of $(X, \t|_X, \ont|_X, \int')$: 
The relation $\ont|_X$ inherits acyclicity from $\ont$.
Suppose $p \onto q\into' p$.  
Since $\Mult(\ont|_X,\int')=\t|_X$, we have $q\to p$, so there exists $z\in\Sha$ such that $q \onto z \into p$.
Thus $p\onto z\into p$, so $z=p$ by the brick condition on $(\Sha,\t,\ont,\int)$.
Thus $q\onto p\onto q$, so that $p=q$ by the acyclicity of $\onto$.

To complete the proof of two-acyclicity\, suppose $p\into'q\into'p$.  
Again, since $q\into'p$, there exists $z\in\Sha$ such that $q \onto z \into p$.  
Since $X$ is a downset for $\ont$, we have $z\in X$.
Thus by the observation that $\int'\supseteq\int|_X$, we have $q \onto z \into' p$, so $q\onto z\into'q$.
By the brick condition on $(X, \t|_X, \ont|_X, \int')$, established above, we have $z=q$, so $q\into p$.
Similarly, since $p\into'q$, there exists $w\in X$ with $p\onto w\into q$, and thus $p\onto w\into'q$, so $p\onto w\into'p$, and thus $w=p$, so that $p\into q$.
By the acyclicity of $\int$, we see that $p=q$.

Writing $\triangle$ for orthogonality operations with respect to $\t|_X$, we now establish that ${}^{\triangle} (U^{\triangle}) = {}^{\perp} (U^{\perp})$ for any $U\subseteq X$.  
Unpacking definitions yields
\[ ^{\triangle}\!\bigl(U^{\triangle}\bigr) = {}^{\perp}\bigl( U^{\triangle} \bigr) \cap X =  {}^{\perp}{\bigl(} U^{\triangle} \cup Y {\bigr)} =  {}^{\perp} {\Big(} (U^{\perp} \cap X) \cup Y {\Big)}  . \]
It is clear that $U^{\perp} \supseteq (U^{\perp} \cap X) \cup Y$, and thus ${}^{\perp} (U^{\perp} ) \subseteq {}^{\triangle} (U^{\triangle})$.  

For the reverse containment, suppose $p \not\in {}^{\perp} (U^{\perp})$.  
Since $U\subseteq X$, we have $U^\perp\supseteq X^\perp=Y$ and thus ${}^{\perp} (U^{\perp})\subseteq Y^\perp=X$.  
If $p\not\in X$, then $p\not\in {}^{\triangle} (U^{\triangle})$, so we may as well assume $p\in X$.
Since $p \not\in {}^{\perp} (U^{\perp})$, there exists $r \in U^{\perp}$ with $p \to r$. 
Factor $p \to r$ as $p \onto q \into r$. 
Since $X$ is an $\ont$-downset and $p\in X$, we have $q \in X$.
If $q\not\in U^\perp$, then there exists $u\in U$ with $u\rightarrow q$. 
Since $q\into r$, this implies $u\rightarrow r$, contradicting $r\in U^\perp$.
Thus $q\in U^\perp$.
Since also $q\in X$, we have $q\in U^{\triangle}$.
Since $p\to q$, we see that $p\not\in {}^{\triangle}(U^{\triangle})$, as desired.

Since the closure operations coincide on subsets of $X$, the lattices of closed sets coincide.
Thus $(U, U^{\perp}) \mapsto (U, U^{\triangle})$ is an isomorphism from $\bigl[(\emptyset, \Sha), (X,Y)\bigr]$ to $\Pairs(\to')$. 
The equality $U^{\triangle} = U^{\perp} \cap X$ gives the formulation of this result in the statement of the Theorem.
\end{proof}

We explicitly state the dual result for upper intervals of $\Pairs(\t)$.

\begin{prop}\label{prop-upper-intervals} 
Suppose $(\Sha,\t,\ont,\int)$ is a two-acyclic factorization system and let $(X,Y)\in\Pairs(\t)$. 
Then $(Y, \t|_Y, \ont', \int|_Y)$ is a two-acyclic factorization system, where $\ont'$ is given by $\Fact(\t|_Y)=(\ont',\int|_Y)$.
The map $(U,V) \mapsto (U \cap Y, V)$ is an isomorphism from the interval $\bigl[(X,Y),(\Sha,\emptyset)\bigr]$ in $\Pairs(\t)$ to the lattice $\Pairs(\t|_Y)$.
\end{prop}

Concatenating Propositions~\ref{prop-lower-intervals} and~\ref{prop-upper-intervals}, we obtain:  

\begin{theorem}\label{thm-intervals} 
Suppose $(\Sha,\t,\ont,\int)$ is a two-acyclic factorization system and let $\bigl[(X_1,Y_1),(X_2, Y_2)\bigr]$ be an interval in $\Pairs(\t)$.
Let $\t'$ be the restriction of $\t$ to $X_2 \cap Y_1$ and let $(\ont', \int') = \Fact(\t')$.  
Then $(X_2 \cap Y_1, \t', \ont', \int')$ is a two-acyclic factorization system.
The map $(U,V) \mapsto (U \cap Y_1, V \cap X_2)$ is an isomorphism from $\bigl[(X_1,Y_1),(X_2, Y_2)\bigr]$ to $\Pairs(\to')$. 
\end{theorem}

\begin{proof}  
By Proposition~\ref{prop-lower-intervals}, $(X_2, \t|_{X_2}, \ont|_{X_2}, \int'')$ is a two-acyclic factorization system, where $\int''$ is defined by $\Fact(\t|_{X_2}) = (\ont|_{X_2}, \int'')$.   Furthermore, $(X_1, X_2 \cap Y_1)$ is in $\Pairs(\t|_{X_2})$. 
The relation $\to'$ is the restriction of $\t|_{X_2}$ to $X_2 \cap Y_1$, so Proposition~\ref{prop-upper-intervals} states that $(X_2 \cap Y_1, \t', \ont', \int')$ is a two-acyclic factorization system.
By Proposition~\ref{prop-lower-intervals}, the interval  $\bigl[(X_1,Y_1),(X_2, Y_2)\bigr]$ in $\Pairs(\t)$ is isomorphic to the interval $\bigl[(X_1, X_2 \cap Y_1), (X_2, \emptyset) \bigr]$ in $\Pairs(\t|_{X_2})$.
By Proposition~\ref{prop-upper-intervals}, the latter is isomorphic to $\Pairs(\to')$. 
The isomorphisms given in the theorems are $(U,V) \mapsto (U, V \cap X_2) \mapsto (U \cap Y_1, V \cap X_2)$. 
\end{proof}

We conclude with a corollary of Theorem~\ref{thm-intervals} which will be useful in the next section.
\begin{lemma} \label{cor-element-moves}  
Suppose $(\Sha,\t,\ont,\int)$ is a two-acyclic factorization system and suppose $(X_1,Y_1)\le(X_2, Y_2)$ in $\Pairs(\t)$. 
  Then $(X_2, Y_2)$ covers $(X_1, Y_1)$ if and only if $X_2 \cap Y_1$ is a singleton. 
\end{lemma}

\begin{proof}
By definition, $(X_2, Y_2)$ covers $(X_1, Y_1)$ if and only if $\bigl[ (X_1,Y_1), (X_2, Y_2) \bigr]$ has precisely two elements.  
Equivalently (by Theorem~\ref{2afs iff ws kappa}), the lattice $\Pairs(\t')$ of Theorem~\ref{thm-intervals} has $|\JIrr^c(\Pairs(\t'))|=1$. In other words, ${|X_2 \cap Y_1|=1}$.
\end{proof}

\begin{remark}
See \cite[Theorem~3.3(b)]{DIRRT} for a similar result about covers in the lattice of torsion classes.
\end{remark}

\section{Covers and canonical join representations}\label{cjc sec}  
In this section, we describe cover relations in the lattice of maximal orthogonal pairs of a two-acyclic factorization system. 
We then use the description of covers to describe canonical join representations in the finite case. 
Throughout the section, $(\Sha,\t,\ont,\int)$ will be a two-acyclic factorization system.
All references to orthogonality operations $\perp$ refer to this system, and references to closures and closed sets refer to the closure operator $X\mapsto\overline X={}^\perp(X^\perp)$.

\subsection{Covers}  
Given $c\in X$, define $\Del(X,c)=X\setminus\set{x\in X:x\onto c}$.
This notation suggests ``deleting $c$ from $X$'', as $\Del(X,c)$ is the largest $\ont$-downset contained in $X$ but not containing $c$. 
Given a closed set $X$, define $\Cov(X)$ to be the set of elements $c\in X$ such that $\Del(X,c)$ is closed and $\overline {\Del(X,c)\cup\{c\}}= X$. 
We will prove the following theorem. 

\begin{theorem} \label{new cov thm}
Let $(\Sha,\t,\ont,\int)$ be a two-acyclic factorization system and let $X$ be a closed set.
Then $\Del(X,c)\covered X$ in the lattice of closed sets for all $c\in\Cov(X)$, and every element covered by $X$ is $\Del(X,c)$ for a unique $c\in\Cov(X)$.
\end{theorem}

Proposition~\ref{step} lets us restate Theorem~\ref{new cov thm} in terms of maximal orthogonal pairs:
Given $(X,Y)\in\Pairs(\t)$, the elements covered by $(X,Y)$ are precisely the pairs $(\Del(X,c),\Del(X,c)^\perp)$ for $c\in\Cov(X)$, with exactly one $c$ for each cover.

We will also prove some alternative descriptions of $\Cov(X)$. 
The first is general and explicit, but unilluminating. 
The second is general but comes at the cost of computing $\Fact(\t|_X)$. 
The third applies only to special cases, including the finite case, but provides some insight into the general case.

\begin{proposition} \label{reform}
$\Cov(X)$ consists of those elements $c$ of $X$ which have the following properties: \begin{itemize}
\item $c$ is $\int$-maximal in $X$,
\item for all $y \in \Sha$, if there exists $x \in X$ such that $x \into y$, then there exists $x'\in \Del(X,c)\cup \{c\}$ such that $x'\into y$.  
\end{itemize}
\end{proposition}

\begin{theorem}\label{cov thm fact}
Let $(\Sha,\t,\ont,\int)$ be a two-acyclic factorization system and let $X$ be a closed set.
Write $\Fact(\t|_X)=(\ont|_X,\int')$ as in Proposition~\ref{prop-lower-intervals}.
Then $\Cov(X)$ is the set of $\int'$-maximal elements of $X$.
\end{theorem}

\begin{theorem}\label{cov thm fin}
Let $(\Sha,\t,\ont,\int)$ be a two-acyclic factorization system and let $X$ be a closed set.
Let $\C(X)$ be the set of $\int$-maximal elements of $X$. 
Then every element of $\Cov(X)$ is an $\ont$-maximal element of~$\C(X)$.
If $\Sha$ is finite, or more generally if $X$ is finite, then $\Cov(X)$ is precisely the set of $\ont$-maximal elements of $\C(X)$.
\end{theorem}

We now proceed to prove Theorem \ref{new cov thm}. Our proof uses the following lemma. 

\begin{lemma} \label{cover-lemma}
If $(X,Y)\in\Pairs(\t)$ and $c$ is $\int$-maximal in $X$, then
\[ \Del(X,c) = X \setminus\set{x\in X: x \to c} = {}^{\perp} (Y \cup \{ c \}) \]
Consequently $\Del(X,c)$ is closed. 
\end{lemma}

\begin{proof}
The first equality says that $x\onto c$ if and only if $x\to c$.
If $x\onto c$, then $x\onto c\into c$, so $x\to c$.
If $x\to c$ then $x \onto z \into c$ for some $z\in\Sha$.
Proposition~\ref{TorsionImpliesDownset} says that $X$ is a downset for $\ont$, so $z \in X$.
Since $c$ is $\int$-maximal in $X$, we have $z = c$, so $x \onto c$ as desired.
For the second equality, note that ${}^{\perp} (Y \cup \{ c \}) = \left( {}^{\perp} Y \right) \cap \left( {}^{\perp} \{ c \} \right) = X \cap \left( {}^{\perp} \{ c \} \right) = X \setminus \{ x \to c \}$.
The set $\Del(X,c)$ is closed because any set of the form ${}^\perp S$ is closed (Proposition~\ref{perp is closed}).  
\end{proof}

\begin{proof}[Proof of Theorem \ref{new cov thm}]
  For $c\in \Cov(X)$, we have that $\Del(X,c)$ is a closed set, and any strictly larger closed set
  contains $c$, so contains the closure of $\Del(X,c)\cup\{c\}$, which is
  $X$. This shows $\Del(X,c)\lessdot X$.
  
  Suppose now that $(X',Y')\covered (X,X^\perp)$ in the lattice of maximal orthogonal pairs. 
  From Lemma~\ref{cor-element-moves}, we have $X \cap Y' = \{ c \}$ for some $c \in X$. We will establish that $c\in \Cov(X)$ and $X'=\Del(X,c)$. 

If there exists $c'\in X\setminus\set{c}$ with $c' \into c$, then since $Y'$ is an $\int$-upset, we have $c' \in Y'$.  
Thus because $X \cap Y'$ is the singleton $\set{c}$, we see that $c$ is $\int$-maximal in $X$.
Thus by Lemma~\ref{cover-lemma}, $\Del(X,c)$ is closed.

Since $X'$ is a downset for $\ont$ and $c\not \in X'$, we must have
$X'\subseteq \Del(X,c)$. It therefore follows from the fact that
$X'\lessdot X$ that $X'=\Del(X,c)$. Since we now know that
$\Del(X,c)\lessdot X$, the closure of $\Del(X,c)\cup \{c\}$ must be $X$.
We have proved that $c\in\Cov(X)$.

The uniqueness of $c$ such that $X'=\Del(X,c)$ follows because $c$ is the unique $\ont$-minimal element of $\set{x\in X: x \onto c}$ by the acyclicity of $\ont$. \end{proof}

We now prove Proposition \ref{reform} and Theorems~\ref{cov thm fact} and ~\ref{cov thm fin}.

\begin{proof}[Proof of Proposition \ref{reform}]
Suppose $c$ satisfies the conditions of the proposition. Since $c$ is $\int$-maximal in $X$, by Lemma \ref{cover-lemma}, $\Del(X,c)$ is closed.

Let $y\in \Sha$. Then $y\not\in X^\perp$ if and only if there exists some $x \in X$ with $x\to y$. 
By factoring $x \to y$ as $x \onto z \into y$, and noting that $X$ is a downset for $\onto$, we see that the existence of $x \in X$ with $x \to y$ is equivalent to the existence of $x \in X$ with $x \into y$.
By the hypothesis, this is equivalent to the existence of $x'\in \Del(X,c)\cup \{c\}$ with $x'\into y$, which is equivalent to $y \not \in (\Del(X,c)\cup\{c\})^\perp$.  
This shows $X$ and $\Del(X,c)\cup\{c\}$ have the same right perpendiculars, so their closures agree. We have now shown that $c\in \Cov(X)$.

Conversely, suppose that $c\in \Cov(X)$. 
The argument in the previous paragraph also proves that $c$ satisfies the second condition of the proposition.
The fact that $c$ in $\int$-maximal in $X$ was established in the proof of Theorem \ref{new cov thm}.
  \end{proof}

Theorem~\ref{cov thm fact} follows from the following proposition, in light of Proposition~\ref{prop-lower-intervals}.
   
\begin{prop}\label{cover-Sha-prop} 
Let $(\Sha,\t,\ont,\int)$ be a two-acyclic factorization system. 
Then $\Cov(\Sha)$ is the set of $\int$-maximal elements of $\Sha$.
\end{prop}
\begin{proof} 
We showed in the proof of Theorem~\ref{new cov thm} that each element of $\Cov(X)$ is $\int$-maximal in $X$.
Conversely, if $c$ is $\int$-maximal in $\Sha$, then since $\Sha^\perp=\emptyset$, Lemma \ref{cover-lemma} says that $\Sha \setminus \set{x\in\Sha: x \onto c}$ is ${}^\perp\set{c}$.
Since $c$ is $\int$-maximal, $\set{c}=F(c)$,
so Proposition~\ref{JIrr and MIrr Characterize} says in particular that $(\Sha \setminus \set{x\in\Sha: x \onto c},\set{c})$ is a maximal orthogonal pair.
Since $\set{c}$ is a singleton, this pair is covered by $(\Sha,\emptyset)$.
\end{proof} 

\begin{proof}[Proof of Theorem~\ref{cov thm fin}]
The proof of Theorem~\ref{new cov thm} established that each element $c$ of $\Cov(X)$ is in $\C(x)$.
If $c'\neq c$ is also in $\C(X)$ and $c'\onto c$, then $X\supsetneq\Del(X,c')\supsetneq\Del(X,c)$, because $\ont$ is a partial order.
By Lemma~\ref{cover-lemma}, these are all closed sets, contradicting the fact that $\Del(X,c)\covered X$.
Thus $c$ is $\ont$-maximal in $\C(X)$.

Now suppose $X$ is finite.
We will show that every $\ont$-maximal element $c$ of $\C(X)$ is in $\Cov(X)$.
If $X$ does not cover $\Del(X,c)$, then $X$ must cover some $X'$ with $X' \supsetneq \Del(X,c)$.
By Theorem~\ref{new cov thm}, the set $X'$ is $\Del(X,c')$ for some $c' \in \Cov(X)$.  
In particular, $c'\in\C(X)$.
But then the containment $\Del(X,c') \supsetneq \Del(X,c)$ implies that $c' \onto c$ with $c' \neq c$, contradicting the fact that $c$ is $\ont$-maximal in $\C(x)$.
We conclude that $X$ covers $\Del(X,c)$, so $c\in\Cov(X)$ by Theorem~\ref{new cov thm}.
\end{proof}

We pause to point out a lemma that will be useful later.
\begin{lemma}\label{cov diamond}
Let $(\Sha,\t,\ont,\int)$ be a two-acyclic factorization system and let $X$ be a closed set.
If $c\in\Cov(X)$, then $\Del(X,c)\meet T(c)=T_*(c)$ and $\Del(X,c)\join T(c)=X$.
\end{lemma}
\begin{proof}
The first assertion follows from the definitions.
The second follows from the definitions and the fact that $\Del(X,c)\covered X$.
\end{proof}

\subsection{Canonical join representations}
For $S$ and $S'$ subsets of $L$, we write $S \lleq S'$ if, for every $s \in S$, there exists $s' \in S'$ with $s \leq s'$. 
The relation $\lleq$ is a preorder, and becomes a partial order when restricted to antichains. 

We say that  $S$ is the \newword{canonical join representation} of $x\in L$ if
\begin{enumerate}
\item $x = \Join S$
\item  $S\lleq S'$ for any set $S'$ such that $x=\Join S'$ and
\item $S$ is an antichain.
\end{enumerate}

\begin{remark}
We make several remarks on this definition: Since $\lleq$ is a partial order on antichains, $x$ has at most one canonical join representation.
If $L$ is finite and $S$ obeys the first two conditions to be the canonical join representation of $x$, then the set of maximal elements of $S$ is the canonical join representation of $x$.
Finally, if $S$ is the canonical join representation of $x$ and $S' \subsetneq S$, then $\Join S' < x$.
\end{remark}

The following well-known fact can be found, for example in \cite[Theorem~3-1.4]{semi} or \cite[Theorem~2.24]{FreeLattices}.

\begin{theorem}\label{cjr semi}
Suppose $L$ is a finite lattice.  
Then $L$ is join semidistributive if and only if every element of $L$ has a canonical join representation.
\end{theorem}

Thus, if $(\Sha, \t, \ont, \int)$ is a finite two-acyclic factorization system, we deduce that each element of $\Pairs(\t)$ has a canonical join representation.
We will now give an explicit description of this representation. 
Recall the notation $T(x)$ for $\{ x' : x \onto x'\}$.

\begin{theorem} \label{thm-cjr-description}  
Let $(\Sha, \t, \ont, \int)$ be a finite two-acyclic factorization system.
If $X$ is a closed set, then the canonical join representation of $X$ is $\{T(c) : c \in \Cov(X) \}$. 
\end{theorem}

\begin{proof}
Let $S=\set{T(c):c\in\Cov(X)}$.
We first verify that $\Join S = X$. 
For each $c \in\Cov(X)$, we have $T(c) \subseteq X$, so  $\Join S  \subseteq X$. 
Suppose for the sake of contradiction that this containment were strict. 
Then there would be some $X'$, covered by $X$, with $\Join S \subseteq X' \subsetneq X$.
(This is the one use of finiteness in this proof.)  
By Theorem~\ref{new cov thm}, $X'=\Del(X,c)$ for some $c \in \Cov(X)$ and then $T(c) \not\subseteq X'$, a contradiction. 

We next show that, if $X = \Join S'$, then $S \lleq S'$. 
We must show that, for all $c \in \Cov(X)$, there is some $X' \in S'$ with $T(c) \subseteq X'$. 
We have $T(c) \subseteq X'$ if and only if $c \in X'$ so suppose, for the sake of contradiction that $c$ is not in any $X' \in S'$. 
Since $c \not \in X'$, we have $X' \cap \{ x\in\Sha: x \onto c \} = \emptyset$, and we deduce that the closed set $\Del(X,c)$ contains $X'$ for all $X' \in S'$. 
But then $\Join S' \subseteq \Del(X,c) \subsetneq X$, a contradiction.

Theorem~\ref{cov thm fin} implies in particular that $\Cov(X)$ is an antichain in the partial order $\ont$.
Therefore also $S$ is an antichain.   
\end{proof}

By Proposition~\ref{prop-lower-intervals}, the conclusion of Theorem ~\ref{thm-cjr-description} also holds when the hypothesis that $\Sha$ is finite is replaced with the weaker hypothesis that $X$ is finite.

For ease of exposition, let $L$ be a finite lattice.
We say that $S\subseteq L$ \newword{joins canonically} if $S$ is the canonical join representation of an element (which is necessarily $\Join S$).
When $L$ is join semidistributive,  the collection of all $S\subset L$ such that $S$ joins canonically is a simplicial complex \cite[Proposition~2.2]{arcs}.
This simplicial complex is called the \newword{canonical join complex}.
Its vertex set is $\JIrr L$. 

A simplicial complex is called \newword{flag} if all of its minimal non-faces are edges, or equivalently, if it is the clique complex of its $1$-skeleton.
The following theorem of Barnard~\cite[Theorem~1.1]{Barnard} characterizes finite semidistributive lattices among all finite join semidistributive lattices in terms of their canonical join complexes.

\begin{theorem}\label{flag}
A finite join semidistributive lattice is semidistributive if and only if its canonical join complex is flag.
\end{theorem}

In particular, to understand the canonical join complex of a finite semidistributive lattice $L$, it suffices to know which $2$-element subsets of $\JIrr L$ join canonically. We will address this after proving the following lemma, which gives an alternate description of the $\t_L$ relation:
\begin{lemma}\label{j kappa}   
Let $L$ be a $\kappa$-lattice and let $i,j\in\JIrr^cL$.
Then $i\to_Lj$ if and only if $i\join j_* \ge j$, where $j_*$ is the unique element of $L$ covered by $j$.
\end{lemma}
\begin{proof}
Since $j_*\le\kappau(j)$, we have $i\le\kappau(j)$ if and only if $i\join j_*\le\kappau(j)$.
Since $j\not\le\kappau(j)$, if $i\join j_*\le\kappau(j)$, then $i\join j_*\not\ge j$.
On the other hand, if $i\join j_*\not\le\kappau(j)$, then since $j_*\le i\join j_*$, by definition of $\kappa$, we have $i\join j_*\ge j$.
\end{proof}

We now explain how to detect which $2$-element subsets of $\JIrr L$ join canonically in a  finite semidistributive lattice $L$, and thus describe the canonical join complex of $L$.

\begin{theorem}\label{cjc to}
Let $L$ be a finite semidistributive lattice, identified with $\Pairs(\t_L)$ for $\t_L$ as in Theorem~\ref{FTFSDL1}.
The faces of the canonical join complex of $L$ are the subsets $S$ of $\JIrr L$ such that $x\not\to_L y$ for all distinct $x,y\in S$.
\end{theorem}

\begin{proof}
In light of Theorem~\ref{flag}, we need only prove that for $i,j\in\JIrr L$, the set $\set{i,j}$ joins canonically if and only neither $i\to_L j$ nor $j\to_L i$.
If $i\to_L j$ then $i\join j_*=i\join j$ by Lemma~\ref{j kappa}.  
Since $\set{i,j_*}\lleq\set{i,j}$, the set $\set{i,j}$ does not join canonically.
Conversely, suppose $\set{i,j}$ does not join canonically and write $x=i\join j$.
If $x=\Join S$ is the canonical join representation of $x$, then $S\lleq\set{i,j}$. 
Since $x=\Join S$ is irredundant, $S\not\supseteq\set{i,j}$, so without loss of generality, $j\not\in S$.
By the definition of $\lleq$, and since $j\not\in S$, every $s\in S$ has either $s\le i$ or $s\le j_*$. 
Thus $S\lleq\set{i,j_*}$, so that $x=\Join S\le i\join j_*\le i\join j=x$.
We see that $ i\join j_*=i\join j$, so that $i\to_L j$.
\end{proof}

\begin{remark}\label{flag insight}
\textit{A priori}, it is not unreasonable to hope to find a Fundamental Theorem of Finite Join-Semidistributive Lattices along the same lines as FTFSDL.
However, in view of Theorem~\ref{flag}, we expect that such a hypothetical FTFJDL would not simply involve binary relations between the elements of $\Sha$, but rather relations between $k$-tuples of elements for $k>2$. 
\end{remark}

\section{Quotient lattices}  \label{quot sec} 
In this section, we prove Theorem~\ref{thm finite quotients} and Corollary~\ref{con pairs finite}, which describe all lattice congruences and quotients of a finite semidistributive lattice.
We prove these results as special cases of results on lattices of maximal orthogonal pairs for not-necessarily-finite two-acyclic factorization systems.

To begin the section, we give background and terminology necessary to understand the results and their generalizations.
Further background and proofs, from a point of view compatible with the present treatment, can be found in \cite[Section~9-5]{regions9} and \cite[Section~2]{DIRRT}.  

A \newword{lattice homomorphism} is a map $\eta: L_1 \to L_2$ with
$\eta(x \meet y) = \eta(x) \meet \eta(y)$ and $\eta(x \join y) = \eta(x) \join \eta(y)$ for all $x,y\in L_1$.
Equivalently, $\eta(\Join X)=\Join(\eta X)$ for all finite, non-empty 
subsets $X\subseteq L_1$.
A \newword{complete lattice homomorphism} has $\eta(\Join X)=\Join(\eta X)$ for infinite subsets $X$ as well.

A \newword{congruence} on a lattice $L$ is an equivalence relation $\Theta$ on $L$ that respects the meet and join operations in the following sense:    
If $x_1\equiv y_1$ and $x_2\equiv y_2$ mod $\Theta$, then $x_1\meet x_2\equiv y_1\meet y_2$ and $x_1\join x_2\equiv y_1\join y_2$ mod $\Theta$.
Equivalently, if $I$ is a finite, non-empty indexing set and $(x_i:i\in I)$ and $(y_i:i\in I)$ are elements of $L$ such that $x_i\equiv y_i\bmod \Theta$ for all $i\in I$, then $\Meet_{i\in I}x_i\equiv\Meet_{i\in I}y_i$ and $\Join_{i\in I}x_i\equiv\Join_{i\in I}y_i$ mod $\Theta$.
If the same condition holds for all infinite indexing sets, then $\Theta$ is a \newword{complete congruence}.
The \newword{quotient} $L/\Theta$ is the set of congruence classes,
with the obvious lattice structure.
The map sending an element to its congruence class is a surjective lattice homomorphism from $L$ to $L/\Theta$.
On the other hand, given a (complete) surjective lattice homomorphism $\eta:L_1\to L_2$, the fibers of $\eta$ are the congruence classes of a (complete) congruence $\Theta$ on $L_1$, and $\eta$ induces an isomorphism from $L_1/\Theta$ to $L_2$.
Every congruence class of a complete lattice congruence is an interval in $L$.

The set of all equivalence relations on a set $X$ forms a lattice (the lattice of set partitions of $X$), where $\Theta_1$ is smaller than $\Theta_2$ if $\Theta_1$ \newword{refines} $\Theta_2$ (in other words, if $x \equiv y \bmod \Theta_1$ implies $x \equiv y \bmod \Theta_2$).
The set $\Con(L)$ of all congruences on $L$ is a sublattice of the lattice of set partitions of~$L$.
While the lattice of set partitions is not even semidistributive (for $|X|>2$), the lattice $\Con(L)$ is distributive for any lattice $L$ \cite[Theorem~149]{Gratzer}. 

In particular, when $L$ is finite, we can understand $\Con(L)$ by way of the FTFDL.
Given a congruence $\Theta$ on a finite lattice $L$ and an edge $a\covered b$ in the Hasse diagram of $L$, we say that $\Theta$ \newword{contracts} $a\covered b$ if $a\equiv b$ mod $\Theta$.
There is a unique finest congruence contracting $a\covered b$ (the meet in $\Con(L)$ of all congruences contracting $a\covered b$), and we write $\con(a,b)$ for this congruence.
The following proposition is well known; see, for example, \cite[Proposition~9.5.14]{regions9}.

\begin{prop} \label{JIrr of ConL}
Let $L$ be a finite lattice and $\Theta$ a congruence on $L$. The following are equivalent:
\begin{enumerate}[\qquad\rm(i)]
\item $\Theta$ is a join-irreducible element in the lattice $\Con(L)$.
\item $\Theta$ is of the form $\con(a,b)$ for some 
  cover $a \covered b$ of $L$.
\item $\Theta$ is of the form $\con(j_*,j)$ for some join-irreducible element $j$.
\item $\Theta$ is of the form $\con(m,m^*)$ for some meet-irreducible element $m$.
\end{enumerate}
\end{prop}

Since $\Con(L)$ is a distributive lattice, the FTFDL states that it is isomorphic to the lattice of downsets of $\JIrr (\Con(L))$. 
Proposition~\ref{JIrr of ConL} states that $j \mapsto \con(j_*,j)$ gives a surjection from $\JIrr L$ to $\JIrr (\Con(L))$.
We define the \newword{forcing preorder} on $\JIrr L$ by saying that $j$ \newword{forces} $j'$ if and only if $\con(j_*,j) \geq \con(j'_*,j')$ or, in other words, if every congruence contracting $j_{\ast} \covered j$ also contracts $j'_{\ast} \covered j'$.
Forcing is a preorder on $\JIrr L$, meaning it is reflexive and transitive, but perhaps not antisymmetric. 
Thus $\Con(L)$ is isomorphic to the poset of downsets for the forcing preorder on $\JIrr L$. 

On a finite lattice, congruences are determined by the set of covers that they contract.
The same is not true for infinite complete lattices. 
For example, consider the $\NN\union\set{\infty}$ with the obvious partial order.
There are two complete congruences that contract all covers, one with a single congruence class and one with two classes $\NN$ and $\set{\infty}$.
We will consider the congruences that \emph{are} determined by the covers they contract.

Given a congruence $\Theta$ on a lattice~$L$, if $x\le y$ and $x\equiv y\bmod\Theta$, then for all $u,v$ with $x\le u\covered v\le y$, we have $u\equiv v\bmod\Theta$.
The congruence $\Theta$ is \newword{cover-determined} if the converse holds as well:
If $x\le y$ and $u\equiv v\bmod\Theta$ for all $u,v$ with $x\le u\covered v\le y$, then $x\equiv y\bmod\Theta$.

We have given the lattice-theoretic background to Theorems~\ref{thm finite quotients} and its generalization.
We now need two more definitions related to two-acyclic factorization systems.

Let $(\Sha, \t, \ont, \int)$ be a two-acyclic factorization system.   
Given $x$, $y\in\Sha$, we write $x\forces y$ and say $x$ \newword{directly forces} $y$ if and only if either  
\begin{enumerate}[\qquad\rm(i)]
\item \label{forcs1}
$x$ is $\ont$-minimal in $F(y) = \{ x'\in\Sha : x'\into y\}$, or
\item \label{forcs2}
$x$ is $\int$-maximal in $T(y) = \{ x'\in\Sha : y \onto x' \}$.
\end{enumerate}  

\begin{eg}\label{forcs ex}
This example continues Example~\ref{2afs ex}.
The relation $\forcs$ on $\Sha=\set{\labels}$ is shown in Figure~\ref{forcs fig}. 
(Arrows of the form $x\forces x$ are omitted.)
For example, $F(\7)=\set{\7,\5,\3}$ and all three elements of $F(\7)$ are $\ont$-minimal in $F(\7)$, so $\5\forces\7$ and $\3\forces\7$.
Also, $T(\7)=\set{\6,\7}$ and both elements of $T(\7)$ are $\int$-maximal in $T(\7)$, so $\6\forces\7$.
\begin{figure}
  \mbox{
    \begin{tikzpicture}[->,decoration = {snake,pre length=3pt,post length=3pt,amplitude=1.5pt,segment length = 10pt}]
      \node (1) at (1,1) {$\1$};  
      \node (2) at (2,0) {$\2$};  
      \node (4) at (3,1) {$\4$};  
      \node (6) at (3,0) {$\6$};  
      \node (7) at (3,-1) {$\7$}; 
      \node (5) at (4,0) {$\5$};  
      \node (3) at (5,-1) {$\3$}; 
      \draw[decorate] (1)--(2); \draw[decorate] (1)--(4); \draw[decorate] (2)--(4); \draw[decorate] (6)--(4);
      \draw[decorate] (6)--(7); \draw[decorate] (5)--(7); \draw[decorate] (3)--(5); \draw[decorate] (3)--(7);
      \draw[decorate] (4)--(5); \draw[decorate] (7)--(2);
      \end{tikzpicture}}
        
\caption{The $\forcs$ relation for Example~\ref{2afs ex} (Figure~\ref{2afs fig})} 
\label{forcs fig}
\end{figure}

\end{eg}

Let $x \to z$, so $\{ y\in\Sha : x \onto y \into z \}$ is nonempty. We define an $\ont$-minimal element of $\{ y\in\Sha : x \onto y \into z \}$ to be an \newword{image} of $x \to z$ and define an $\int$-maximal element to be a \newword{co-image} of $x \to z$.
We will say that a factorization system $\Sha$ satisfies the \newword{image condition} if, for every $x$ and $z \in \Sha$ with $x \to z$, there is at least one image and at least one co-image of $x \to z$.

Recall that a \newword{$\forcs$-upset} means a subset $\Che$ of $\Sha$ such that, if $x \forces y$ and $y \in \Che$, then $x \in \Che$. 
We will prove the following generalization of Theorem~\ref{thm finite quotients}.

\begin{theorem}\label{thm lattice quotients}
Suppose $(\Sha, \t, \ont, \int)$ is a two-acyclic factorization system that obeys the image condition and $\Che\subseteq\Sha$ is a $\forcs$-upset.
Then the restriction of $(\Sha, \t, \ont, \int)$ to $\Che$ is a two-acyclic factorization system that obeys the image condition.
The map $(X,Y) \mapsto (X \cap \Che, Y \cap \Che)$ is a surjective complete lattice homomorphism between the corresponding lattices of maximal orthogonal pairs.
The complete congruence associated to $(X,Y) \mapsto (X \cap \Che, Y \cap \Che)$ is cover-determined.
Every cover-determined congruence on $\Pairs(\t)$ arises from a $\forcs$-upset in this way.
\end{theorem}

We note that it is not obvious that $(X \cap \Che, Y \cap \Che)$ is a maximal orthogonal pair in $\Pairs(\t|_\Che)$; this is checked as Proposition~\ref{prop intersection is max orth}.

Since every finite factorization system satisfies the image condition and every congruence on a finite lattice is complete and cover-determined, Theorem~\ref{thm finite quotients} is an immediate corollary of Theorem~\ref{thm lattice quotients}.

We now proceed to prove Theorem~\ref{thm lattice quotients}.
Following that proof, we will prove Corollary~\ref{con pairs finite} as an immediate corollary of an infinite version (Corollary~\ref{con pairs infinite}).

\medskip
\noindent
\textbf{Conventions:}   
For the rest of the section, we will assume that $(\Sha, \t, \ont, \int)$ is a two-acyclic factorization system that obeys the image condition.
We will also assume that $\Che \subset \Sha$ is a $\forcs$-upset. 
The symbols $\t'$, $\ont'$, and $\int'$ will stand for the restrictions of $\t$, $\ont$ and $\int$ to $\Che$. 
\medskip

Throughout the proof, the image condition will almost always be used through the following lemma:

\begin{lemma}\label{image lemma}
If $y$ is an image of $x \to z$, then $y \forces z$.
If $y$ is a co-image of $x \to z$ then $y \forces x$.
\end{lemma}

\begin{proof}
We prove the first claim; the second is dual. 
Suppose that $y \not\forces z$.
By hypothesis, $y \into z$, so $y$ must not be $\ont$-minimal in $\{ y' : y' \into z \}$. 
Thus there exists $y' \neq y$ with $y \onto y' \into z$.
But then $x \onto y' \into z$, contradicting the fact that $y$ is an image.
\end{proof}

We prove Theorem~\ref{thm lattice quotients} as a series of propositions.

\begin{proposition} \label{prop is factorization}
$(\Che, \t', \ont', \int')$ is a two-acyclic factorization system that obeys the image condition.
\end{proposition}

\begin{proof}  
Two-acyclicity is obvious, and it is clear that $\Mult(\ont', \int') \subseteq \t'$ and $\Fact(\t') \supseteq (\ont', \int')$. 
We now prove the latter containments are equalities.

In order to show that $\Mult(\ont', \int') \supseteq \t'$, we suppose $x,z \in \Che$ have $x \to z$ and show that there exists $y \in \Che$ with $x \onto y \into z$. 
The image condition says there is an image $y$ of $x \to z$. 
By Lemma~\ref{image lemma}, $y \forces z$, and since $\Che$ is a $\forcs$-upset, we have $y \in \Che$, as desired. 
(We pause to note that $y$ is also an image of $x \to z$ in $\Che$.)  

To show that $\Fact(\t') \subseteq (\ont', \int')$, let $\Fact(\t') = (\ont'', \int'')$.
We show that $\ont''\subseteq \ont'$;  the proof that $\int''\subseteq\int'$ is dual.
We first suppose that $x_1 \onto'' x_2$ for some $x_1$ and $x_2 \in \Che$ and show that $x_1 \onto x_2$. 
In other words, the hypothesis is that $\{ y \in \Che : x_1 \to y \} \supseteq \{ y \in \Che : x_2 \to y \}$, and the desired conclusion is that $\{ z \in \Sha : x_1 \to z \} \supseteq \{ z \in \Sha : x_2 \to z \}$. 
Suppose $x_2 \to z$ and let $y$ be a co-image
of $x_2 \to z$. 
By Lemma~\ref{image lemma}, $y \forces x_2$, and thus $y\in\Che$ because $\Che$ is a $\forcs$-upset.
Since $y$ is a co-image of $x_2\to z$, in particular, $x_2\onto y\into z$.
By the hypothesis, $x_1 \onto y\into z$, so that $x_1 \to z$ as desired. 
(Again, we note that $y$ is also a co-image of $x_2\to z$ in $\Che$.)

We have shown that $(\Che, \t', \ont', \int')$ is a two-acyclic factorization system.
In the process, we have shown how to find both an image and a co-image for an arbitrary $\t'$-arrow, so $(\Che, \t', \ont', \int')$ obeys the image condition.
\end{proof}

\begin{prop} \label{prop intersection is max orth}
If $(X,Y)\in\Pairs(\t)$, then $(X\cap\Che,Y\cap\Che)\in\Pairs(\t')$. 
\end{prop}

\begin{proof}  
We write $X'$ for $X\cap\Che$ and $Y'$ for $Y\cap\Che$.
We use the symbol $\triangle$ for $\perp$ with respect to $(\Che, \t', \ont', \int')$.
Since $(X,Y)$ is a maximal orthogonal pair for $(\Sha, \t, \ont, \int)$, we have $Y'\subseteq (X')^\triangle$ and $X'\subseteq{}^\triangle(Y')$.
We will argue that $Y' \supseteq (X')^{\triangle}$;  the argument that $X'\supseteq{}^\triangle(Y')$ is dual.

To show that $Y' \supseteq (X')^{\triangle}$, we must show that, if $z \in \Che$ and $z \not \in Y$, then there exists $y \in X'$ with $y \to z$. 
Since $z \not \in Y$, by definition, there is some $x \in X$ with $x \to z$. 
Let $y$ be an image of $x \to z$, so that $y \forces z$ by Lemma~\ref{image lemma}. 
But $\Che$ is a $\forcs$-upset, so this implies that $y \in \Che$. 
Also, $x \onto y$, so $y \in X$ by Proposition~\ref{TorsionImpliesDownset}, and thus $y \in X'$.
Since $y\into z$, we have $y \to z$, as desired.
\end{proof}

\begin{prop} \label{prop map is quotient}
The map $(X, Y) \mapsto (X\cap\Che,Y\cap\Che)$ is a surjective complete lattice homomorphism from $\Pairs(\t)$ to $\Pairs(\t')$.
\end{prop}

\begin{proof}  
Identifying $\Pairs(\t)$ with the lattice of closed sets, the meet is intersection.
The assertion that the map respects arbitrary meets is $\bigcap_i (X_i \cap \Che) = \left( \bigcap X_i \right) \cap \Che$, which is obvious.
The proof that the map respects joins 
is dual (identifying $\Pairs(\t')$ with the dual notion of the lattice of closed sets by considering the second entries in pairs).

We now show that this map is surjective. 
Let $(X', Y')$ be a maximal orthogonal pair for $\t'$. 
Let $X = {}^\perp Y'$ (using this $\perp$ operator in the sense of $(\Sha, \t, \ont, \int)$). 
Then $X$ is closed (in the sense of $\t$), and unfolding the definitions immediately shows that $X \cap \Che = X'$. 
Let $Y = X^{\perp}$. 
Then $(X, Y)$ is a maximal orthogonal pair, so $(X \cap \Che, Y \cap \Che)$ is a maximal orthogonal pair by Proposition~\ref{prop intersection is max orth}.
Since $X \cap \Che  = X'$ and $(X', Y')$ is a maximal orthogonal pair, we also have $Y \cap \Che = Y'$.
\end{proof}

The following is \cite[Proposition~2.6]{DIRRT}.

\begin{prop}\label{cov det sep}
A complete lattice congruence $\Theta$ on a complete lattice $L$ is cover-determined if and only if the quotient $L/\Theta$ is weakly atomic.
\end{prop}

Theorem~\ref{infinite main} says that $\Pairs(\t')$ is weakly atomic, and thus Proposition~\ref{cov det sep} implies the following proposition.

\begin{prop}\label{Che cov det}
The complete congruence associated to $(X, Y) \mapsto (X\cap\Che,Y\cap\Che)$ is cover-determined.
\end{prop}

To prove the rest of Theorem~\ref{thm lattice quotients}, we need two simple, known lemmas.  
The second is essentially the easy direction of~\cite[Lemma~3.2]{DayBounded}, but appears in that source without proof.

\begin{lemma} \label{lem kappa and collapsing}
Let $L$ be a $\kappa$-lattice, let $\eta : L \to L'$ be a lattice homomorphism, and let $j$ be a completely join irreducible element of $L$.
Then $\eta(j) = \eta(j_*)$ if and only if $\eta(\kappa(j)) = \eta(\kappa(j)^*)$. 
\end{lemma}

\begin{proof}
For brevity, put $m = \kappa(j)$. 
We have $j \join m = m^*$ and $j_* \join m = m$, so if $\eta(j)=\eta(j_*)$ then $\eta(m) = \eta(m^*)$. 
We have $m \meet j = j_*$, so if $\eta(m)=\eta(m^*)$ then $\eta(j) = \eta(j_*)$. 
\end{proof}

\begin{lemma}\label{Day lem} 
Suppose $L$ is a lattice and $\Theta$ is a lattice congruence on $L$.
Suppose we have elements $a$, $b$, $c$, $d$ of $L$ with $a \leq b$ and $a \equiv b \bmod \Theta$. 
Suppose that either $a \leq c \leq d \leq b \join c$ or else $a \meet d \leq c \leq d \leq b$. Then $c \equiv d \bmod \Theta$.
\end{lemma}

\begin{proof}
We consider the case $a \leq c \leq d \leq b \join c$; the other case is similar.
Then $a \equiv b \bmod \Theta$, so $c=a \join c \equiv b \join c \bmod \Theta$.  
Since $c \leq d \leq b \join c$, we have $c = c \meet d \equiv (b \join c) \meet d =  d \bmod \Theta$.
\end{proof}

\medskip
\noindent
\textbf{Additional conditions:}
We continue to assume that $(\Sha, \t, \ont, \int)$ is a two-acyclic factorization system that obeys the image condition.
In addition, for the rest of the section, we will assume that $\Theta$ is a complete lattice congruence on $\Pairs(\t)$.
Furthermore, we define 
\[\Che_\Theta = \{ x \in \Sha :  (T_*(x), T_*(x)^\perp) \not\equiv (T(x), T(x)^\perp) \bmod \Theta  \}.\]
Lemma~\ref{lem kappa and collapsing} and Proposition~\ref{pairs kappa explicit} combine to imply that 
\[\Che_\Theta = \{ x \in \Sha :  ({}^{\perp} F(x), F(x))\not\equiv({}^{\perp} F^*(x), F^*(x)) \bmod \Theta  \}.\]
We will want both descriptions of $\Che_\Theta$.
\medskip

\begin{prop} \label{all quotients forcs upset} 
$\Che_\Theta$ is a $\forcs$-upset.
\end{prop} 

\begin{proof}
We prove the equivalent assertion that $\Sha\setminus\Che_\Theta$ is a $\forcs$-downset.  
Suppose $x\not\in\Che_\Theta$, meaning that $(T_*(x), T_*(x)^\perp)\equiv(T(x), T(x)^\perp)$, and suppose $x \forces y$.

We suppose that $x$ is $\ont$-minimal in $F(y)$; the case where $x$ is $\int$-maximal in $T(y)$ is similar. 
We will apply the first case of Lemma~\ref{Day lem} with ${a = (T_*(x), T_*(x)^\perp)}$, $b = (T(x), T(x)^\perp)$, $c=({}^{\perp} F(y), F(y))$ and $d = ({}^{\perp} F^*(y), F^*(y))$. 
To that end, we must confirm that $(T_*(x), T_*(x)^\perp) \leq ({}^{\perp} F(y), F(y))$ and $ ({}^{\perp} F^*(y), F^*(y)) \leq (T(x), T(x)^\perp) \join ({}^{\perp} F(y), F(y))$. 

To show that $(T_*(x), T_*(x)^\perp) \leq ({}^{\perp} F(y), F(y))$, we must show that $T_*(x) \subseteq {}^{\perp} F(y)$. 
Suppose for the sake of contradiction that there exists $x'\in T_*(x)$ with $x'\not\in{}^\perp F(y)$.
That is, there exist $x$ and $y$ in $\Sha$ with $x\onto x'\neq x$ and $x'\to y'\into y$.
Factor the arrow $x'\to y'$ as $x' \onto z \into y'$, so that $x\onto z\into y$.
Since $x\neq x'$, also $x\neq z$, contradicting the $\ont$-minimality of $x$ in $F(y)$.

To show that $ ({}^{\perp} F^*(y), F^*(y)) \leq (T(x), T(x)^\perp) \join ({}^{\perp} F(y), F(y))$, we must show that $F^*(y)\supseteq T(x)^\perp\cap F(y)$. 
The only element of $F(y)$ not in $F^*(y)$ is $y$, so we simply must show that $y\not\in T(x)^\perp$. 
But $x \in F(y)$, so in particular $x \to y$ and thus $y \not \in T(x)^{\perp}$. 

Now Lemma~\ref{Day lem} says that $({}^{\perp} F(y), F(y)) \equiv ({}^{\perp} F^*(y), F^*(y)) \bmod \Theta$, so that $y\not\in\Che_\Theta$, using the second description of $\Che_\Theta$.
We have proved that $\Sha\setminus\Che_\Theta$ is a $\forcs$-downset, as desired.%
\end{proof}

In light of Proposition~\ref{all quotients forcs upset}, Proposition~\ref{prop is factorization} says the restriction of $(\Sha,\t,\ont,\int)$ to $\Che_\Theta$ is a two-acyclic factorization system satisfying the image condition, and Proposition~\ref{prop map is quotient} says the map $(X,Y)\mapsto(X\cap\Che_\Theta,Y\cap\Che_\Theta)$ is a complete surjective lattice homomorphism from $\Pairs(\t)$ to the lattice of maximal orthogonal pairs for the restriction.

\begin{prop} \label{all quotients inf}
If $\Theta$ is cover-determined, then $\Theta$ is the congruence associated to $(X,Y) \mapsto {(X \cap \Che_\Theta, Y \cap \Che_\Theta)}$.
\end{prop} 

The following lemma will be useful in proving Proposition~\ref{all quotients inf}.
The lemma depends neither on the image condition nor on the completeness of $\Theta$.
It follows immediately from Lemma~\ref{cov diamond} and the definition of a congruence. 
\begin{lemma}\label{cov diamond cong}
Suppose $(\Sha, \t, \ont, \int)$ is a two-acyclic factorization system and $\Theta$ is a congruence on $\Pairs(\t)$.
If $(X,Y)\in\Pairs(\t)$ and $c\in\Cov(X)$, then $(\Del(X,c),\Del(X,c)^\perp)\equiv(X,Y)$ if and only if $(T_*(c), T_*(c)^{\perp})\equiv (T(c), T(c)^{\perp})$ mod~$\Theta$.
\end{lemma}

\begin{proof}[Proof of Proposition~\ref{all quotients inf}]
Proposition~\ref{Che cov det} says that the complete congruence on $\Pairs$ associated to $(X,Y) \mapsto {(X \cap \Che_\Theta, Y \cap \Che_\Theta)}$ is cover-determined.
Since $\Theta$ is also cover-determined, we only need to show that if $(X_1, Y_1)\covered(X_2, Y_2)$ in $\Pairs(\t)$, then $(X_1, Y_1) \equiv (X_2, Y_2) \bmod \Theta$ if and only if $X_1 \cap \Che_\Theta = X_2 \cap \Che_\Theta$. 

Suppose $(X_1, Y_1)\covered(X_2, Y_2)$ in $\Pairs(\t)$.
Theorem~\ref{new cov thm} says that there exists $c$ such that $X_1=\Del(X_2,c)$.
Since $X_1\cap\Che_\Theta$ is a downset for the restriction of $\ont$ to $\Che_\Theta$, we have $X_1 \cap \Che_\Theta = X_2 \cap \Che_\Theta$ if and only if $c\not\in\Che_\Theta$.
By definition, $c\not\in\Che_\Theta$ if and only if $(T_*(c), T_*(c)^{\perp})\equiv (T(c), T(c)^{\perp}) \bmod \Theta$.
By Lemma~\ref{cov diamond cong}, this is if and only if $(X_1, Y_1) \equiv (X_2, Y_2) \bmod \Theta$.
\end{proof}

We have completed the proof of Theorem~\ref{thm lattice quotients}.
We now proceed to generalize and prove Corollary~\ref{con pairs finite}. 

Given a complete lattice $L$, let $\Concc(L)$ be the set of complete cover-determined congruences on $L$.
Then $\Concc(L)$ is a complete lattice under refinement order.  
More specifically, \cite[Proposition~2.7]{DIRRT} says that it is a complete meet-sublattice of the lattice of complete congruences on $L$, which is a complete meet-sublattice of $\Con(L)$.
It is thus a complete meet-semilattice, and since it also has a maximal element, it is a complete lattice by a standard argument.

Given a cover relation $a\covered b$ in $L$, we use the notation $\concc(a,b)$ for the finest complete cover-determined congruence having $a\equiv b$.
This is the meet, in $\Concc(L)$, of all complete cover-determined congruences with $a\equiv b$;  this meet has $a\equiv b$ because $\Concc(L)$ is a meet-sublattice of the lattice of partitions of~$L$.
A cover-determined complete congruence $\Theta$ is the join, in $\Concc(L)$, of the congruences $\concc(a,b)$ for all covers $a\covered b$ contracted by $\Theta$.

When $L$ is $\Pairs(\t)$ for a two-acyclic factorization system, Lemma~\ref{cov diamond cong} implies that a cover-determined congruence $\Theta$ is the join of the congruences $\concc(a,b)$ for all covers $a\covered b$ contracted by $\Theta$ such that $b$ is $(T(x),T(x)^\perp)$ and $a$ is $(T_*(x),T_*(x)^\perp)$.
We are led to define the \newword{cover-determined forcing preorder} on $\JIrr^c(\Pairs(\t))$:
We say that $(T(x),T(x)^\perp)$ forces $(T(y),T(y)^\perp)$ if and only if every cover-determined congruence contracting the cover $(T_*(x),T_*(x)^\perp)\covered(T(x),T(x)^\perp)$ also contracts the cover $(T_*(y),T_*(y)^\perp)\covered(T(y),T(y)^\perp)$.

We are now prepared to state and prove the generalization of Corollary~\ref{con pairs finite}.

\begin{cor} \label{con pairs infinite}  
Suppose $(\Sha, \t, \ont, \int)$ is a two-acyclic factorization system that obeys the image condition.
The map $x\mapsto(T(x), T(x)^{\perp})$ is an isomorphism from the transitive closure of $\forcs$ on $\Sha$ to the cover-determined forcing preorder on $\JIrr^c(\Pairs(\t))$.
This map induces an isomorphism from the poset of $\forcs$-downsets under containment to the lattice $\Concc(\Pairs(\t))$.
\end{cor}

\begin{proof}
By Proposition~\ref{JIrr and MIrr Characterize}, the map $x\mapsto(T(x), T(x)^{\perp})$ is a bijection from $\Sha$ to $\JIrr^c (\Pairs(\t))$.  
Theorem~\ref{thm lattice quotients} implies that for $x,y\in\Sha$, $(T(x),T(x)^\perp)$ forces $(T(y),T(y)^\perp)$ in the cover-determined forcing preorder if and only if there is a directed path in the relation $\forcs$ from $x$ to $y$.
Theorem~\ref{thm lattice quotients} also implies that cover-determined complete congruences are completely specified by the set of elements $x\in\Sha$ such that $(T_*(x), T_*(x)^{\perp})\covered(T(x), T(x)^{\perp})$ is contracted, and that the sets arising are precisely the $\forcs$-downsets.
The refinement order on cover-determined complete congruences is containment order on these $\forcs$-downsets.
\end{proof}

\section{Relation to other types of finite lattices} \label{lattice types sec}

\subsection{Distributive lattices} \label{FTFDL and FTFDSL}  
A lattice $L$ is distributive if $x \meet (y \join z) = (x \meet y) \join (x \meet z)$ and $x \join (y \meet z) = (x \join y) \meet (x \join z)$ for all $x$, $y$ and $z \in L$.
This condition is clearly stronger than semidistributivity, so every distributive lattice is semidistributive and thus, every finite distributive lattice is of the form $\Pairs(\t)$ where $(\Sha, \t, \ont, \int)$ is a two-acyclic factorization system. 

We now state three results that summarize the relationship between the Fundamental Theorem of Finite Distributive Lattices (FTFDL---Theorem~\ref{FTFDL}) and the Fundamental Theorem of Finite Semidistributive Lattices (FTFSDL---Theorem~\ref{FTFSDL1}).
These statements are easily proved, and we omit the proofs. 
Theorem~\ref{FTFDL redux} is stated in a way that allows convenient comparison with Theorem~\ref{FTFSDL2}.

\begin{prop}
Suppose $P$ is a set and $\t$ is a partial order on $P$.
Then $(P, \t, \t, \t)$ is a two-acyclic factorization system.
The map $(X,Y)\mapsto X$ is an isomorphism from $\Pairs(\t)$ to $\Downsets(P)$, with inverse $X\mapsto(X,P\setminus X)$.
\end{prop}

\begin{theorem}[FTFDL, rephrased]\label{FTFDL redux}
A finite poset $L$ is a distributive lattice if and only if it is isomorphic to $\Pairs(\t)$ for some partial order $\t$ on a finite set $P$.
In this case, $(P,\t)$ is isomorphic to $(\JIrr L,\ge)$, where $\ge$ is the partial order induced from~$L$. 
The map $x\mapsto(\set{j\in\JIrr L: j\leq x },\ \set{j\in\JIrr L:j\not\le x}))$ is an isomorphism from $L$ to $\Pairs(\t)$, with inverse $(X,Y)\mapsto\Join X$.
\end{theorem}

\begin{prop}\label{dist char}
Let $(\Sha, \t, \ont, \int)$ be a finite two-acyclic factorization system. Then the following are equivalent:
\begin{enumerate}[\qquad\rm1.]
\item $\Pairs(\t)$ is distributive. \label{cond1}
\item $\t$ is a partial order. \label{cond5}
\item $\ont = \int$. \label{cond2}
\item $\t = \ont$. \label{cond3}
\item $\t = \int$. \label{cond4}
\end{enumerate}
\end{prop}

Using Proposition~\ref{dist char}, we see that if $L$ is finite and distributive then $x \forcs y$ if and only if $x=y$.
From there, by Corollary~\ref{con pairs finite}, we recover the well known fact that $\JIrr(\Con(L))$ 
is an antichain.

\subsection{Congruence uniform lattices}\label{cu sec}
We recall our discussion of lattice congruences 
from Section~\ref{quot sec}.
As we described, the set of congruences on a finite lattice $L$ forms a distributive lattice $\Con(L)$. Since $\Con(L)$ is distributive, it is isomorphic to the poset of downsets of $\JIrr(\Con(L))$. 
Proposition~\ref{JIrr of ConL} described surjective maps to $\JIrr(\Con(L))$ from $\JIrr L$, from $\MIrr L$, and from the set of covers of~$L$.
When $L$ is semidistributive, Lemma~\ref{cov diamond cong} implies that surjection from $\mathrm{Covers}(L)$ to $\JIrr(\Con(L))$ factors (as a composition of surjections) through $\JIrr L$.
Dually, the map factors through $\MIrr L$.
Also when $L$ is semidistributive, $\kappau$ gives a bijection between $\JIrr L$ and $\MIrr L$.
These maps are shown in the diagram below, and they all commute:
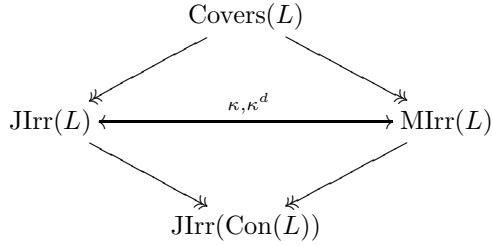
\[ \xymatrix{
& \mathrm{Covers}(L) \ar@{->>}[dl] \ar@{->>}[dr] & \\
\JIrr L \ar@{<->}[rr]^{\kappau, \kappad} \ar@{->>}[dr] && \MIrr L \ar@{->>}[dl] \\
& \JIrr(\Con(L)) & \\ 
} . \]

A finite lattice $L$ is \newword{conguence uniform} if the map $j \mapsto \con(j_*,j)$ from $\JIrr L$ to $\JIrr(\Con(L))$ and the map $m \mapsto \con(m,m^*)$ from $\MIrr L$ to $\JIrr(\Con(L))$ are both bijections.
In this case, $L$ is semidistributive (see \cite[Lemma~4.2]{DayBounded} and \cite[Theorem~5.1]{DayBounded}), with $\kappau$ and $\kappad$ being  the unique maps that make the diagram above commute.
Thus a finite lattice $L$ is conguence uniform if and only if it is semidistributive and one (and therefore both) of the maps $j \mapsto \con(j_*,j)$ and $m \mapsto \con(m,m^*)$ is a bijection.
Since the forcing preorder on $\JIrr L$ is the pullback of the partial order on $\JIrr(\Con(L))$, we can alternatively say that $L$ is congruence uniform if it is semidistributive and the forcing preorder on $\JIrr L$ is a partial order.
Thus we have the following corollary of the FTFSDL and Corollary~\ref{con pairs finite}.

\begin{cor}\label{cu cor}
Suppose $(\Sha,\t,\ont,\int)$ is a finite two-acyclic factorization system.
Then $\Pairs(\t)$ is congruence uniform if and only if $\forcs$ is acyclic.  
\end{cor}

As mentioned above, every finite congruence uniform lattice is also semidistributive.
Thus the concatenation of Corollary~\ref{cu cor} and Theorem~\ref{FTFSDL1} can be thought of as a Fundamental Theorem of Finite Congruence Uniform Lattices.

\begin{eg}\label{non cong unif ex}
This example continues Examples~\ref{2afs ex} and~\ref{forcs ex}.
The relation $\forcs$ shown in Figure~\ref{forcs fig} is not acyclic.
Indeed, the lattice pictured in Figure~\ref{Pairs fig} is a well known example
of a semidistributive lattice that is not congruence uniform.
\end{eg}

Another characterization of congruence uniformity, due to Day, uses the notion of doubling that we now recall.
Let $\2$ be the two-element lattice with elements $1<2$.
Let $L$ be a finite lattice and let $I$ be an interval in $L$.
Let  $L[I]$ be the set $(L\setminus I)\cup(I\times\2)$ and let $\pi : L[I] \to L$ be the obvious projection.
The \newword{doubling} of $I$ in $L$ is $L[I]$ equipped with the partial order that $x \leq y$ if 
\begin{enumerate}[\qquad\rm(i)]
\item $\pi(x) \leq \pi(y)$ and 
\item  if $x = (a,i)$ and $y = (b,j) \in I \times \2$, then $i \leq j$.
\end{enumerate}
The following is \cite[Corollary~5.4]{DayBounded}.

\begin{theorem}\label{cu double}
A finite lattice $L$ is congruence uniform if and only if there is a sequence of lattices $L_0,L_1,\ldots,L_k=L$ such that $L_0$ has exactly one element and for all $j=1,\ldots,k$, there is an interval $I_{j-1}$ in $L_{j-1}$ such that $L_j\cong L_{j-1}[I_{j-1}]$.
\end{theorem}

In light of Corollary~\ref{cu cor} and Theorem~\ref{cu double}, it is natural to ask how doubling happens in a two-acyclic factorization system.
This question is answered by the following theorem, whose proof we omit.

\begin{theorem}\label{double fact sys} 
Suppose that $(\Sha,\t,\ont,\int)$ is a finite two-acyclic factorization system and suppose that $(X_1,Y_1)\le(X_2,Y_2)$ in $\Pairs(\t)$.
Then the doubling $\Pairs(\t)[(X_1,Y_1),(X_2,Y_2)]$ is a semidistributive lattice isomorphic to $\Pairs(\t')$ for a two-acyclic factorization system $(\Sha\cup\set{a},\t',\ont',\int')$ for some $a\not\in\Sha$.
The relations $\t'$, $\ont'$, and $\int'$ agree with $\t$, $\ont$, and $\int$ on $\Sha$. 
In addition, there are the reflexive relations $a \to' a$, $a \onto' a$, and $a \into' a$ and the following relations:
\begin{enumerate}[\qquad\rm(i)]
\item $x\to' a$ for all $x\in\Sha\setminus X_2$ and $a\to' y$ for all $y\in\Sha\setminus Y_1$. 
\item $a\onto' y$ for all $y\in X_1$ and $x\onto' a$ for all $x\in\Sha\setminus X_2$ such that $x\onto z$ for all $z\in X_1$.
\item $x\into' a$ for all $x\in Y_2$ and $a\into' y$ for all $y\in\Sha\setminus Y_1$ such that $z\into y$ for all $z\in Y_2$.
\end{enumerate}
\end{theorem}

Part of Theorem~\ref{double fact sys} is the following proposition, which is well known.
(See, for example, \cite[Theorem~B]{DNT}.)
\begin{prop}\label{double sd}
If $L$ is a finite semidistributive lattice and $I$ is an interval in~$L$, then $L[I]$ is a semidistributive lattice.
\end{prop}

\subsection{General finite lattices} \label{MarkSec}
One may wonder whether there is a Fundamental Theorem of Finite Lattices, giving a canonical way of realizing an arbitrary lattice as something like $\Pairs(\t)$.
Such a result appears a paper by Barbut on analyzing results of opinion surveys \cite{Barb65}; it was rediscovered by Markowsky~\cite{Mark75}.
We restate the result in a manner that allows easy comparison to FTFSDL.
Whereas FTFSDL realizes a finite semidistributive lattice as $\Pairs(\t)$ for a relation $\t$ on a single set $\Sha$, the theorem realizes an arbitrary finite lattice as $\Pairs(\t)$ for a relation $\t$ between two sets $\El$ and $\Er$.
($\El$ and $\Er$ are the Cyrillic letters ``el" and ``er", which we hope are mnemonic for ``left" and ``right".)
Given a relation $\t$ between sets $\El$ and $\Er$, for any subset $X$ of $\El$, put $X^{\perp} = \{ y \in \Er : \ x \not\to y \ \forall x \in X \}$ and, for any subset $Y$ of $\Er$, put ${}^{\perp} Y = \{ x \in \El : \ x \not\to y \ \forall y \in Y \}$. 
We define $(X,Y)$ to be a maximal orthogonal pair if $Y = X^{\perp}$ and $X = {}^{\perp} Y$ and we write $\Pairs(\t)$ for the set of maximal orthogonal pairs, ordered by containment on the first element. 
We partially order $\Pairs(\t)$ by containment in the first entries of pairs (or equivalently reverse containment in the second entries).

We write $\Fact(\t)$ for the pair $(\ont,\int)$ of preorders $\ont$ and $\int$ on $\El$ and $\Er$ respectively defined as follows:
$x_1 \onto x_2$ if $x_2 \to y$ implies $x_1 \to y$;  
and $y_1 \into y_2$ if $x \to y_1$ implies $x \to y_2$. 
(If $\El=\Er$, then this coincides with the earlier definition of $\Fact(\t)$ in connection with FTFSDL.)  
Given $x \in\El$, a \newword{right companion} of $x$ is an element $y \in\Er$ such that $x \to y$ and such that if $x \onto x' \to y$ then $x= x'$.   
Similarly, for $y \in \Er$, a \newword{left companion} of $y$ is an element $x \in\El$ such that $x \to y$ and such that if $x \to y' \into y$ then $y = y'$.
We say that $\t$ is \newword{companionable} if every $x\in\El$ has a right companion and every $y \in\Er$ has a left companion.

\begin{theorem}[FTFL  \cite{Barb65,Mark75}]\label{FTFL} 
A finite 
poset $L$ is a lattice if and only if it is isomorphic to $\Pairs(\t)$ for a companionable relation $\t$ between finite sets $\El$ and~$\Er$.
In this case, $(\El, \Er, \t)$ is isomorphic to $(\JIrr L,\MIrr L,\t_L)$, where $j \to_L m$ if and only if $j \not\leq m$.
Writing $\Fact(\t_L)=(\ont_L,\int_L)$, we see that $\ont_L$ is the partial order induced on $\JIrr L$ as a subset of $L$, while $\int_L$ is the partial order induced on $\MIrr L$ as a subset of $L$.
The map
\[x\mapsto( \{ j \in \JIrr L: j \leq x \},\  \{ m \in \MIrr L : m \geq x \} )\]
is an isomorphism from $L$ to $\Pairs(\t_L)$, with inverse $(X,Y)\mapsto\Join X=\Meet Y$.
\end{theorem}

Theorem~\ref{FTFL} shows that a finite lattice can always be written as $\Pairs(\t)$ by taking $\El = \JIrr L$, $\Er = \MIrr L$ and $j \to m$ if $j \not\leq m$, whether or not that lattice is semidistributive.
What is special in the semidistributive case is, first, that $\JIrr L$ and $\MIrr L$ can be identified with a single set $\Sha$ such that each element is its own left and right companion and, second, that we can recognize those $(\El, \Er, \t)$ which come from semidistributive lattices by axioms reminiscent of abelian categories.
In Section~\ref{ext sect}, we will see a different case in which $\El$ and $\Er$ can be identified.

The following proof spells out how Theorem~\ref{FTFL} is a restatement of Markowsky's theorem in~\cite{Mark75}.

\begin{proof}[Proof of Theorem~\ref{FTFL}]
By Proposition~\ref{Jx Mx}, $\JIrr L$ and $\MIrr L$ are join- and meet-generating subsets of $L$, so they can play the roles of $X$ and $Y$ in \cite[Theorem~5]{Mark75}.
By \mbox{\cite[Theorem~5(a)]{Mark75}}, the given map from $L$ to $\Pairs(\t_L)$ is an isomorphism.  

Now, let $(\El, \Er, \t)$ be two sets and a relation such that $L \cong \Pairs(\t)$. 
By \mbox{\cite[Theorem~9]{Mark75}}, we have $(\El, \Er, \t) \cong (\JIrr L, \MIrr L, \not\leq)$ if and only if   
\begin{enumerate}
\item For all $x \in \JIrr L$, if $\Delta \subseteq \JIrr L$ is such that $\{ y \in \Er : x \to y \}$ equals ${\{ y \in \Er : \exists x' \in \Delta \ \mbox{with}\ x' \to y \}}$, then $x \in \Delta$ and
\item For all $y \in \JIrr L$, if $\Gamma \subseteq \MIrr L$ is such that $\{ x \in \El : x \to y \}$ equals ${\{ x \in \El : \exists y' \in \Gamma \ \mbox{with}\ x \to y' \}}$, then $y \in \Gamma$.
\end{enumerate}
We will show that the first condition is equivalent to saying that every $x \in \El$ has a right companion; analogously, the second condition is is equivalent to saying that every $y \in \Er$ has a left companion.

So, suppose that $x$ has a right companion $y$ and suppose that $\Delta \subseteq \JIrr L$ is such that $\{ y \in \Er : x \to y \} = \{ y \in \Er : \exists x' \in \Delta \ \mbox{with}\ x' \to y \}$. By the definition of a right companion, $y$ is contained in the left hand side, so there is some $x' \in \Delta$ with $x' \to y$. If $x \not \onto x'$, there is some $y'$ with $x' \to y'$ and $x \not\to y'$, in which case $y'$ is in the right hand side but not the left, a contradiction. So $x \onto x'$ and $x' \to y$; then the definition of a right companion shows that $x=x'$, so we have shown that $x \in \Delta$.

Conversely, suppose that $x$ does not have a right companion.   
If $\Delta$ is the set ${\{ x' : x \onto x',\ x \neq x' \}}$, then  $\{ y \in \Er : x \to y \} = \{ y \in \Er : \exists x' \in \Delta \ \mbox{with}\ x' \to y \}$ but $x \not\in \Delta$.
\end{proof}

The idea that a relation $\t$ from $\El$ to $\Er$ gives rise to a
lattice, which can equally well be viewed as a lattice of closed sets in
$\El$ or in $\Er$, goes back to Birkhoff \cite{Birkhoff}. This is sometimes
described as a \newword{Galois connection}, see \cite{Ore}.

Subsequent to Barbut and Markowsky's results showing that there is a canonical way to realize
any finite lattice in this fashion, similar ideas were taken up by Wille
under the title of \newword{formal concept analysis} \cite{GW,W}. Related
subsequent developments can be found in \cite{CLaD}; a recent application of these ideas is in \cite[Section 5]{C+}

\subsection{Extremal lattices} \label{ext sect} 
Let $L$ be a finite lattice.
Let $x_0 < x_1 < \cdots < x_n$ be a chain of elements in $L$. Then, for each $1 \leq i \leq n$, there must be some join-irreducible element $j$ with $j \leq x_i$ and $j \not \leq x_{i-1}$, so $|\JIrr L| \geq n$. Similarly, $|\MIrr L | \geq n$. 
A lattice is called \newword{extremal} if it has a chain of length $N$ where $|\JIrr L| = |\MIrr L| = N$. 
An \newword{acyclic reflexive relation} 
is a reflexive relation that, when interpreted as a directed graph, is acyclic except for $1$-cycles.

We will want the following notation: given a bijection $\mu$ from $\JIrr L$ to $\MIrr L$, define a relation $\to^{\mu}$ on $\JIrr L$ by $i\to^{\mu} j$ if and only if $i\to_L \mu(j)$.
(Recall that, for $j \in \JIrr L$ and $m \in \MIrr L$, we write $j \to_L m$ if $j \not\leq m$.)
The following is a restatement of a characterization of extremal lattices due to Markowsky~\cite{Mark92}.

\begin{theorem}[FTFEL] \label{FTFEL}  
A finite poset $L$ is an extremal lattice if and only if it is isomorphic to $\Pairs(\t)$ for an acyclic reflexive relation $\t$ on a finite set~$\Sha$. 
In this case, $(\Sha, \t)$ is unique up to isomorphism. 
More precisely, in the case where $L$ is an extremal lattice, there is a unique bijection $\mu$ from $\JIrr L$ to $\MIrr L$
for which $\to^{\mu}$ is a reflexive relation, and $(\Sha, \t)$ is isomorphic to $(\JIrr L, \to^{\mu})$.
\end{theorem}

\begin{proof}  
We prove the first two assertions of the theorem by combining \cite[Theorem~13]{Mark92} with other results of Markowsky that appear here as Theorem~\ref{FTFL}. 
In our notation, \cite[Theorem~13]{Mark92} says a finite lattice $L$ is extremal if and only if there is a bijection $\mu$ from $\JIrr L$ to $\MIrr L$ such that $\t^\mu$ is an acyclic reflexive relation.

Theorem~\ref{FTFL} says that in any case, $L$ is isomorphic to $\Pairs(\t_L)$.  
If there exists a bijection $\mu$ from $\JIrr L$ to $\MIrr L$, then there is an isomorphism from $\Pairs(\t_L)$ to $\Pairs(\t^\mu)$ sending $(X,Y)$ to $(X,\mu^{-1}(Y))$.  
We conclude that if $L$ is extremal, then there is a bijection $\mu$ from $\JIrr L$ to $\MIrr L$ such that $\t^\mu$ is an acyclic reflexive relation, and $L$ is isomorphic to $\Pairs(\t^\mu)$.

Conversely, if there exists an acyclic reflexive relation $\t$ on a set $\Sha$ such that $L$ is isomorphic to $\Pairs(\t)$, then $(\Sha,\Sha,\t)$ is a companionable relation; each $x\in\Sha$ is its own left companion and its own right companion.
Thus Theorem~\ref{FTFL} (FTFL) says that $(\Sha,\Sha,\t)$ is isomorphic to $(\JIrr L,\MIrr L,\t_L)$.
This isomorphism features bijections between $\Sha$ and $\JIrr L$ and between $\Sha$ and $\MIrr L$.
Composing these bijections, we obtain a bijection $\mu$ from $\JIrr L$ to $\MIrr L$ such that $\t^\mu$ coincides with $\t$.
Again appealing to \cite[Theorem~13]{Mark92}, we see that $L$ is extremal.
Furthermore, we have established the uniqueness of $(\Sha,\t)$ up to isomorphism.

The final assertion of the theorem is that the bijection $\mu$ is the unique bijection that makes $\t^\mu$ reflexive.
This more precise statement of uniqueness is not in Markowsky's work, but can be seen more generally as follows:
Let $A$ and $B$ be two finite sets with a relation $\t$ between them, and let $\mu$ and $\nu$ be bijections between $A$ and $B$. 
Define a relation $\t^\mu$ on $A$ with $x\to^\mu y$ if and only if $x\to \mu(y)$. 
If $\t^{\mu}$ is reflexive and $\t^{\nu}$ is an acyclic reflexive relation, then we claim that $\mu = \nu$.

Suppose to the contrary that $\mu \neq \nu$. Then we can find some $r \geq 2$ and some $a_1$, $a_2$, \ldots, $a_r \in A$ with $\mu(a_j) = \nu(a_{j+1})$, with indices periodic modulo $r$. Since $\mu$ is reflexive, we have $a_j \to \mu(a_j)$, and thus we have $a_1 \to^{\nu} a_2 \to^{\nu} a_3 \to^{\nu} \cdots \to^{\nu} a_r \to^{\nu} a_1$, contradicting that $\t^{\nu}$ was assumed to be an acyclic reflexive relation.
\end{proof}

If $L$ is a finite extremal lattice, then the relation $(\Sha, \t)$ of Theorem~\ref{FTFEL} is obtained from the relation $(\JIrr L, \MIrr L, \t_L)$
of Theorem~\ref{FTFL} by identifying $\JIrr L$ with $\MIrr L$ by the bijection $\mu$. 
If, instead, $L$ is a semidistributive lattice then $(\Sha, \t)$ is obtained analogously from $(\JIrr L, \MIrr L, \t_L)$ by identifying $\JIrr L$ and $\MIrr L$ by the bijection $\kappa$. In this latter case, $(\Sha, \t)$ need not be acyclic.

The condition of being extremal neither implies nor is implied by being semidistributive:  
We present examples of lattices that satisfy each condition without satisfying the other, followed by two propositions that describe how the conditions can be satisfied simultaneously.
In the examples, if any sort of arrow is drawn between two vertices, we mean that the $\t$ relation holds between them, and we have decorated this arrow to indicate
whether the $\ont$ or $\int$ relation also holds where $(\ont, \int):= \Fact(\t)$. Conveniently, if $x \onto y$ or $x \into y$, then $x \to y$, so we can always draw diagrams in this way.
All relations are assumed to be reflexive, although this is not indicated pictorially.

\begin{eg} \label{ex1} 
The quadruple $(\Sha, \t, \ont, \int)$ shown on the left of Figure~\ref{semi fig} obeys the axioms of a two-acyclic factorization system, but $\t$ has $3$-cycles.
Thus the corresponding lattice of closed sets, shown on the right of Figure~\ref{semi fig}, is semidistributive but not extremal.
\end{eg}

\begin{figure}
  $$
  \begin{tikzpicture}
  \node (1) at (-1,0) {\strut$a$};
  \node (2) at (0,1) {$b$};
  \node (3) at (0,-1) {$c$};
  \node (4) at (1,0) {\strut$d$};
  \draw[right hook->] (1) -- (2);
  \draw[->>] (2) -- (4);
  \draw[->>] (3) -- (1);
  \draw[right hook->] (4) --(3);
  \draw[<->] (2) --(3);
  \draw[draw=none] (-1,-1.5) rectangle (2,1.5);
\end{tikzpicture}
\begin{tikzpicture}
  \node (empty) at (0,0) {$\emptyset$};
  \node (1) at (-1,1) {$\{a\}$};
  \node (13) at (-1,2) {$\{a,c\}$};
  \node (4) at (1,1) {$\{d\}$};
  \node (42) at (1,2) {$\{b,d\}$};
  \node (1234) at (0,3) {$\{a,b,c,d\}$};
  \draw (empty) -- (1) --(13) --(1234)--(42) --(4)--(empty);
  \end{tikzpicture}$$
  
\caption{A two-acyclic factorization system and the closed sets of the corresponding semidistributive lattice, which is not extremal}
\label{semi fig}
\end{figure}

\begin{eg}  \label{ex2}
The quadruple $(\Sha, \t, \ont, \int)$ shown on the left of Figure~\ref{ext fig} is an acyclic reflexive relation. 
However, the middle $\t$-arrow cannot be factored into an $\ont$-arrow and an $\int$-arrow, so the corresponding lattice of closed sets, also pictured, is extremal without being semidistributive.  
\end{eg}


\begin{figure}
  $$    
  \begin{tikzpicture}
    \node (1) at (0,0) {\strut$a$};
    \node (2) at (1,0) {\strut$b$};
    \node (3) at (2,0) {\strut$c$};
    \node (4) at (3,0) {\strut$d$};
    \draw[right hook->] (1) -- (2);
    \draw [->] (2)--(3);
    \draw [->>] (3)--(4);
    \draw [draw=none] (0,-2) rectangle (4,2);
    \end{tikzpicture}
 \begin{tikzpicture}
    \node (empty) at (0,0) {$\emptyset$};
    \node (1) at (-1.5, 1) {$\{a\}$};
    \node (2) at (0,1) {$\{b\}$};
    \node (4) at (1.5,1) {$\{d\}$};
\node (12) at (-1.5,2) {$\{a,b\}$};
    \node (14) at (-0.3,2) {$\{a,d\}$};
\node (34) at (1.5,2) {$\{c,d\}$};
\node (234) at (1,3) {$\{b,c,d\}$};
\node (1234) at (0,4) {$\{a,b,c,d\}$};
\draw (empty) -- (1);
\draw (empty) -- (2);
\draw (empty) -- (4);
\draw (1) -- (12);
\draw (1) -- (14);
\draw (2) -- (12);
\draw (4) -- (14);
\draw (4) -- (34);
\draw (34) -- (234);
\draw (234) -- (1234);
\draw (12) -- (1234);
\draw (14) -- (1234);
\draw (2) -- (234);
  \end{tikzpicture} $$
 \caption{An acyclic reflexive relation and the closed sets forming the
   corresponding extremal lattice, which is not semidistributive}
   \label{ext fig}
\end{figure}

\begin{prop}\label{ext semid}  
Let $\t$ be an acyclic reflexive relation on a finite set $\Sha$ and let $(\ont, \int) = \Fact(\t)$. 
Then $\Pairs(\t)$ is semidistributive if and only if $\t = \Mult(\ont, \int)$, in which case $(\Sha, \t, \ont, \int)$ is 
a two-acyclic factorization system.
\end{prop}

\begin{proof}  
Write $L$ for $\Pairs(\t)$.
By Theorem~\ref{FTFEL}, we can take $\Sha$ to be $\JIrr L$ and $\t$ to be derived from the bijection $\mu$ in Theorem~\ref{FTFEL}. 

Suppose $L$ is semidistributive so that, by Theorem~\ref{FTFSDL1}, $L$ is isomorphic to $\Pairs(\t_L)$ for the finite two-acyclic factorization system $(\JIrr L,\t_L,\ont_L,\int_L)$.
Since in particular $i\not\le\kappa(i)$ for all $i\in\JIrr L$, we see that $\kappa$ is the bijection $\mu$ from Theorem~\ref{FTFEL}.
Thus $\t=\t^{\kappa}=\t_L$ has the desired properties.

Conversely, suppose that $\t=\Mult(\ont, \int)$.  
By definition, $(\ont, \int) = \Fact(\t)$, so $(\Sha, \t, \ont, \int)$ is a factorization system.
Since $\ont\subseteq\t$ and $\int \subseteq \t$, the two-acyclicity of $(\Sha, \t, \ont, \int)$ follows from the acyclicity of $\t$.
\end{proof}

\begin{prop}\label{semid ext}
Let $(\Sha, \t, \ont, \int)$ be a finite two-acyclic factorization system. 
Then $\Pairs(\t)$ is extremal if and only if $\t$ is an acyclic reflexive relation.  
\end{prop}

\begin{proof}
By Theorem~\ref{FTFEL}, if $\t$ is an acyclic reflexive relation, then $\Pairs(\t)$ is extremal. 
Conversely, if $\Pairs(\t)$ is extremal, then it is isomorphic to $\Pairs(\t')$ for an acyclic reflexive relation $\t'$ on a set $\Sha'$.
By Proposition~\ref{ext semid}, $(\Sha', \t',\Mult(\t'))$ is a two-acyclic factorization system.
By Theorem~\ref{FTFSDL1}, $(\Sha', \t', \Mult(\t'))$ and $(\Sha, \t, \ont, \int)$ are isomorphic, so $\t$ is an acyclic reflexive relation.
\end{proof}

It was shown in \cite{TW} that if a lattice is extremal and semidistributive,
then it is also left modular (and therefore trim, since by definition a
lattice is trim if it is extremal and left modular). A central topic in \cite{TW} is the representation of lattices which are trim but not necessarily
semidistributive as maximal orthogonal pairs for a suitable relation.

\section{Motivating examples}\label{conn sec}
In this section, we connect the FTFSDL and its generalizations 
 to the two main examples that motivated it, namely posets of regions of hyperplane arrangements and lattices of torsion classes of finite-dimensional algebras.
Posets of regions, and their quotients, have motivated much of the interest in the combinatorial study of semidistributivity, while lattices of torsion classes provided clues leading to the definition of a two-acyclic factorization system.

Our discussion here does two things:  
In Section~\ref{P(A,B) sec}, we apply the FTFSDL and a known characterization of semidistributivity of the poset of regions to construct two-acyclic factorization systems for a class of posets of regions that includes the simplicial case.
In Section~\ref{rep ssec}, given a finite-dimensional algebra $A$, we construct a two-acyclic factorization system whose lattice of maximal orthogonal pairs is isomorphic to the lattice $\tors(A)$ of torsion classes of $A$.
As a consequence, $\tors(A)$ is a well separated $\kappa$-lattice.
When $\tors(A)$ is finite, this recovers the finite case of \cite[Theorem~1.3]{DIRRT}, namely that $\tors(A)$ is semidistributive.
We point out that~\cite[Theorem~1.3]{DIRRT} states that $\tors(A)$ is completely semidistributive even when it is infinite.
We do not have a combinatorial hypothesis which would let us prove this result, since a well separated $\kappa$-lattice need not be semidistributive (see Example~\ref{obnoxious}).

\subsection{Posets of regions} \label{P(A,B) sec}
Background on posets of regions can be found in \cite{regions9}.
Here we give the basic definitions.
A \newword{(real, central) hyperplane arrangement} $\cA$ is a finite collection of linear hyperplanes in $\RR^n$.
We assume throughout the adjectives real and central, but will not repeat them.
The \newword{regions} of $\cA$ are the closures of the connected components of the complement $\RR^n\setminus(\bigcup_{H\in\cA}H)$.
Fixing one region $B$ to be the \newword{base region}, each region $R$ is specified by its \newword{separating set} $S(R)$, the set of hyperplanes in $\cA$ that separate $R$ from $B$.
The \newword{poset of regions} $\Pos(\cA,B)$ of $\cA$ with respect to $B$ is the set of regions, partially ordered by containment of their separating sets.

The regions of a finite hyperplane arrangement are full-dimensional polyhedral cones.
A facet $F$ of $R$ is a \newword{lower facet} of $R$ if the hyperplane defining it is in the separating set of $R$.
Otherwise, $F$ is an \newword{upper facet}.
A region $R$ of $\cA$ is \newword{tight} with respect to $B$ if for every pair of lower facets of $R$, the intersection of the two facets is a codimension-$2$ face, and if for every pair of upper facets of $R$, the intersection of the two facets is a codimension-$2$ face.
The arrangement $\cA$ is \newword{tight} with respect to $B$ if each of its regions is tight with respect to $B$.
The arrangement is \newword{simplicial} if each of its maximal cones has exactly $n$ facets.
It is immediate that a simplicial arrangement is tight with respect to any choice of $B$.

The following is \cite[Theorem~9-3.8]{regions9} combined with \cite[Corollary~9-3.9]{regions9}.

\begin{theorem}\label{tight semi}
The poset of regions $\Pos(\cA,B)$ is a semidistributive lattice if and only if $\cA$ is tight with respect to $B$. 
If particular, if $\cA$ is simplicial, $\Pos(\cA,B)$ is a semidistributive lattice for any choice of $B$.
\end{theorem}

A \newword{rank-two subarrangement} of $\cA$ is a subset $\cA'$ of $\cA$ with $|\cA'|\ge2$ such that there exists a codimension-$2$ subspace $U$ having $\cA'=\set{H\in\cA:U\subseteq H}$. 
Given a rank-two subarrangement $\cA'$ there is a unique $\cA'$-region $B'$ containing the $\cA$-region~$B$.
The \newword{basic hyperplanes} of $\cA'$ are the two hyperplanes that define the facets of $B'$.

Given distinct hyperplanes $H$ and $H'$, there is a unique rank-two subarrangement $\cA'$ containing $H$ and $H'$.
We say $H'$ \newword{cuts} $H$ if $H'$ is basic in $\cA'$ and $H$ is not basic in $\cA'$.
For each hyperplane $H\in\cA$, consider the subset $H\setminus(\bigcup_{H'}H'\cap H)$, where the union is taken over all $H'\in\cA$ such that $H'$ cuts $H$.
The closures of the connected components of this subset are called the \newword{shards} in $H$.
The set of shards of $\cA$ is the union of the sets of shards of all of the hyperplanes of $\cA$.
The decomposition of the hyperplanes of $\cA$ into shards depends strongly on the choice of $B$, and to emphasize that, we sometimes refer to shards of $\cA$ \emph{with respect to $B$}.

Given a shard $\Sigma$ of $\cA$ contained in a hyperplane $H$, an \newword{upper region} of $\Sigma$ is a region $R$ whose intersection with $\Sigma$ is $(n-1)$-dimensional and which has $H\in S(R)$.
A \newword{lower region} of $\Sigma$ is a region $R$ having an $(n-1)$-dimensional intersection with $\Sigma$, with $H\not\in S(R)$.
\cite[Proposition~9-7.8]{regions9} asserts that, when $\cA$ is tight with respect to $B$, the collection of upper regions of $\Sigma$ contains a unique minimal region (in the sense of $\Pos(\cA,B)$).
Furthermore, this region is \ji in $\Pos(\cA,B)$.
We write $J(\Sigma)$ for this element.
Every \ji element is $J(\Sigma)$ for a unique shard~$\Sigma$.  
Dually, there is a unique maximal region $M(\Sigma)$ among lower regions of~$\Sigma$, this region is \mi in $\Pos(\cA,B)$, and every \mi element arises in this way.

Let $\Sha(\cA,B)$ stand for the set of shards of $\cA$ with respect to $B$.
Define relations $\t$, $\ont$, and $\int$ on $\Sha(\cA,B)$ as follows.
For $\Sigma,\Sigma'\in\Sha(\cA,B)$, set $\Sigma\to\Sigma'$ if and only if $J(\Sigma)\not\le M(\Sigma')$, set $\Sigma\onto\Sigma'$ if and only if $J(\Sigma)\ge J(\Sigma')$, and set $\Sigma\into\Sigma'$ if and only if $M(\Sigma)\ge M(\Sigma')$.

\begin{theorem}\label{PAB fact sys} 
If $\cA$ is tight with respect to $B$, then $(\Sha(\cA,B),\t,\ont,\int)$ is a two-acyclic factorization system.
The map
\[R\mapsto( \{\Sigma\in\Sha(\cA,B):J(\Sigma) \leq R \},\  \left( \{\Sigma\in\Sha(\cA,B):M(\Sigma) \geq R \} \right) )\]
is an isomorphism from $\Pos(\cA,B)$ to $\Pairs(\t)$, with inverse 
\[(X,Y)\mapsto\Meet_{\Sigma\in X}J(\Sigma)=\Meet_{\Sigma\in Y}M(\Sigma).\]
\end{theorem}
\begin{proof}
Theorem~\ref{tight semi} says that $\Pos(\cA,B)$ is semidistributive, and thus is a $\kappa$-lattice by Theorem~\ref{semi char}.
Given a join-irreducible element $J(\Sigma)$ of $\Pos(\cA,B)$, write $J_*(\Sigma)$ for the unique element covered by $J(\Sigma)$.
The unique element $M^*(\Sigma)$ covering $M(\Sigma)$ is an upper region of $\Sigma$, so $J(\Sigma)\le M^*(\Sigma)$.
Thus $M(\Sigma)$ is a maximal element of the set ${\set{R\in\Pos(\cA,B):J(\Sigma)\meet R=J_*(\Sigma)}}$.
We conclude that $\kappau(J(\Sigma))=M(\Sigma)$.
Now Theorem~\ref{FTFSDL1} (FTFSDL) says that $(\Sha(\cA,B),\t,\ont,\int)$ is a two-acyclic factorization system and gives the desired isomorphism.
\end{proof}

\subsection{Representation theory of finite-dimensional algebras} \label{rep ssec}

Let $k$ be a field and let $A$ be a finite-dimensional algebra over $k$. 
A full subcategory of the category of finite-dimensional $A$-modules is called a \newword{torsion class} 
if it is closed under extensions and quotients; meaning that 
\begin{enumerate}
\item If $X$, $Y$ and $Z$ are  finite-dimensional $A$-modules with $0 \to X \to Y \to Z \to 0$ is a short exact sequence, then any torsion class which contains $X$ and $Z$ must contain $Y$ and
\item If $X$ and $Y$ are  finite-dimensional $A$-modules and $X \to Y$ is a surjection, then any torsion class which contains $X$ must contain $Y$. \label{QuotCond}
\end{enumerate}

Torsion-free classes are defined similarly to torsion classes except that condition~(\ref{QuotCond}) is replaced by 
\begin{enumerate}
\item[(2$'$)]  If $X$ and $Y$ are  finite-dimensional $A$-modules and $X \to Y$ is an injection, then any torsion-free
class which contains $Y$ must contain $X$. 
\end{enumerate}
To each torsion class $\mathcal T$ there is an associated torsion-free class 
\[ \mathcal F= \cT^{\perp} = {\{M: \Hom(L,M)=0 \textrm { for all }L\in \mathcal T\}}. \]
Dually, the torsion class corresponding to a torsion-free 
class $\mathcal F$ can be obtained by 
\[ \mathcal T= {}^{\perp} \mathcal F = \{L: \Hom(L,M)=0 \textrm { for all } M\in\mathcal F\} . \]
Inclusion of torsion classes corresponds to reverse inclusion of torsion-free classes. 
Given such a torsion pair $(\mathcal T,\mathcal F)$, for any module $M$, there is a maximal submodule of $M$ which lies in $\mathcal T$. 
This is denoted $t_{\mathcal T} M$, or $tM$ if the intended torsion class is clear.  
The quotient $M/t_{\mathcal T}M$ is the maximal quotient of $M$ which is contained in~$\mathcal F$. 
A good reference for basic facts about torsion classes is \cite[Section VI.1]{ASS}.

We note that conditions (1) and (2) above describe torsion classes as the closed sets for a closure operator on $\operatorname{mod}(A)$.   
Thus, the torsion classes form a complete lattice $\tors(A)$.  
This closure operation is finitary, meaning that the closure of an arbitrary set is the union of the closures of its finite subsets.
(Sometimes such a closure is called ``algebraic''.  See \cite[Definition~4-1.1]{alg}.)
It is known that the closed sets of a finitary closure form an algebraic lattice.
(See, for example, \mbox{\cite[Lemma~4-1.4]{alg}} and \cite[Lemma~4-1.16]{alg}.)
Likewise, (1) and (2$'$) also define a finitary closure operator, so the lattice of torsion-free classes is algebraic. Since the lattice of torsion-free classes is isomorphic to the dual of $\tors(A)$, this shows that $\tors(A)$ is bi-algebraic. 
This result was first shown in \cite[Theorem 3.1(b)]{DIRRT} by a different argument.

The lattice $\tors(A)$ is completely semidistributive \cite[Theorem 3.1(a)]{DIRRT}.

We now explain how to construct a set $\Sha$ and relations $\ont$, $\int$, $\t$
in terms of the representation theory of $A$, so that
we recover the lattice of torsion classes of $A$ as $\Pairs(\t)$.
This construction was one of the inspirations for this project.  

An $A$-module is called a \newword{brick} if its endomorphism ring is a division algebra
(i.e., if all its non-zero endomorphisms are invertible). 
We write $\Bricks(A)$ for the collection of isomorphism classes of bricks.  

For $\br{L}$, $\br{M}\in \Bricks(A)$, we define $\br{L}\to \br{M}$ iff $\Hom(L,M)\ne 0$.
We define $\br{L} \onto \br{M}$ iff $M$ is filtered by quotients of $L$ and define $\br{L}\into \br{M}$ iff $L$ is filtered by submodules of $M$.
Note that a surjective morphism of modules 
from $L$ to $M$ implies $\br{L} \onto \br{M}$ but
the converse does not hold, and similarly for injective maps and $\int$. 

\begin{eg}
Let $A$ be the path algebra of the Kronecker quiver with two arrows $a$ and $b$.
We write a representation $V$ of $A$ as $V_1 \rightrightarrows V_2$.
Let $M$ be the representation $k \rightrightarrows k$ where $a=b=1$ and let $N$ be the representation $k^2 \rightrightarrows k$ where $a = \begin{sbm} 1\\ 0 \end{sbm}$ and $b=\begin{sbm} 0 \\ 1 \end{sbm}$. We clearly cannot have a surjection $M \to N$ as $N$ has larger dimension.
We cannot even have a surjection $M^{\oplus r} \to N$, as every map $M \to N$ has image lying in $(k \begin{sbm} 1\\1 \end{sbm} \rightrightarrows k) \ \subset N$. Note that this submodule of $N$ is isomorphic to $M$, and that the quotient of $N$ by this submodule is the simple module $k \rightrightarrows 0$, which is a quotient of $M$. Thus, $N$ is filtered by quotients of $M$ and $[M] \onto [N]$. 
\end{eg}

We will need one useful fact about bricks. 
This is a special case of \cite[Lemma 1.7(1)]{AsaiSemibricks}. 
We give the short proof from \cite{AsaiSemibricks}.
    \begin{lemma}\label{AsaiLemma} If $S$ is a brick, and $M$ is in the torsion class
      consisting of modules filtered by quotients of $S$, then any non-zero
      morphism from $M$ to $S$ is surjective. \end{lemma}

    \begin{proof} 
    Let $f$ be a nonzero morphism from $M$ to $S$. 
    There is a filtration
      $0< M_1<\dots<M_r=M$ with each $M_i/M_{i-1}$ a quotient of $S$. Let
      $i$ be maximal such that $f|_{M_i}= 0$.  Then $f$ descends to a map
      from $M_{i+1}/M_i$ to $S$. Composing this with the quotient map from
      $S$ to $M_{i+1}/M_i$, we obtain a non-zero endomorphism of $S$. Since
      $S$ is a brick, this endomorphism must be invertible, so the map from $M_{i+1}/M_i$ to
      $S$ must be surjective, and thus the map from $M$ to $S$ must be surjective.\end{proof} 
    The dual statement also holds for torsion-free classes.
    
    We also need the following simple lemma which says that a torsion
    class is determined by the bricks it contains.

    \begin{lemma}\label{bricks enough}  Let $\mathcal T$ be a torsion class. Then every module in
      $\mathcal T$ is filtered by bricks in $\mathcal T$. 
    \end{lemma}

We can thus recover $\cT$ from the set of bricks in $\cT$. 

    \begin{proof} Let $M$ be a minimal-dimensional counter-example.
    Then $M$ is not a brick, so it has some non-invertible non-zero endomorphism.
      Let $I$ be the
      image of this endomorphism. It is both a submodule of $M$ and a quotient
      of $M$.  Since it is a quotient of $M$, $I\in \mathcal T$. Also,
      $M/I\in\mathcal T$. By the hypothesis that $M$ was the minimal
      counterexample, both $I$ and $M/I$ are filtered by bricks in
      $\mathcal T$. But then so is $M$. 
\end{proof}

\begin{theorem}\label{brick thm}
$(\Bricks(A),\t,\ont,\int)$ is a two-acyclic factorization system.
Furthermore, $(\cT,\cF)\mapsto(\cT\cap\Bricks(A),\cF\cap\Bricks(A))$ is an isomorphism from $\tors(A)$ to $\Pairs(\t)$.  
The inverse map takes $(X,Y)$ to $(\cT,\cF)$, where $\mathcal T$ consists of the modules filtered by quotients of bricks from $X$, and $\mathcal F$ consists of the modules filtered by submodules of modules from $Y$.  
\end{theorem}

\begin{proof}
We first prove that $\Fact(\t)=(\ont,\int)$.
Write $(\ont',\int')$ for $\Fact(\t)$.

Suppose $\br{L} \onto \br{M}$.
Specifically, suppose that $M$ has a filtration $0 = M_0 \subset M_1 \subset \cdots \subset M_r = M$ with surjections from $L$ to $M_j/M_{j-1}$.
If $\br{N}$ in $\Bricks(A)$ has $\br{M}\to \br{N}$,
let $\psi : M \to N$ be a nonzero map. 
Let $j$ be minimal such that $\psi$ does not restrict to zero on $M_j$, so $\psi$ descends to a nonzero map from $M_j/M_{j-1}$ to $N$. 
The composition $L \to M_j/M_{j-1} \to N$ is then nonzero, so $\br{L}\to \br{N}$.
We see that $\br{L}\onto' \br{M}$.

Conversely, suppose that $\br{L}\onto' \br{M}$. 
We will show that $M$ is filtered by quotients of $L$.
Let $\mathcal T$ be the torsion class consisting of modules filtered by quotients of $L$. 
Let $tM$ be the torsion part of $M$ with respect to $\mathcal T$.
If $tM=M$, then we are done, since this means that $M\in\mathcal T$, and thus $M$ is filtered by quotients of $L$. 
Otherwise, consider $M/tM$.  
It is in the torsion-free class corresponding to $\mathcal T$, so it admits no morphism from any module in $\mathcal T$, and in particular, it admits no morphism from $L$. 
If $M/tM$ were a brick, we would have found a contradiction, because we would have $\br{M}\to \br{M/tM}$ but $\br{L} \not\to \br{M/tM}$, contradicting our assumption that $\br{L} \onto \br{M}$. 
In general, though, we can take a module $N$ which is minimal-dimensional among modules which are both quotients and submodules of $M/tM$. 
This $N$ must be a brick, because if it has a non-invertible endomorphism, its image would be a smaller quotient and submodule of $M/tM$.  
Now $\Hom(L,N)=0$, but $\Hom(M,N)\ne 0$, contrary to our assumption that $\br{L}\onto' \br{M}$.
We conclude that $\br{L}\onto \br{M}$.

The argument that $\br{L}\into \br{M}$ if and only if $\br{L}\into' \br{M}$ is dual, and we see that $\Fact(\t)=(\ont,\int)$.

We next prove that $\Mult(\ont,\int)=\t$.
Since $\Fact(\t)=(\ont,\int)$, Proposition~\ref{Mult Fact facts}.\ref{mf} says that $\Mult(\ont,\int)\subseteq\t$. 
Conversely, if $\Hom(L,M)\ne 0$, we can define $N$ to be the image of such a map.  
This $N$ need not be a brick, but there is a minimal-dimensional module $N'$ which is both a submodule and a quotient module of $N$.
Then $\br{L}\onto \br{N'} \into \br{M}$.
We have showed that $\Mult(\ont,\int)=\t$.

Suppose that we have two non-isomorphic bricks $S$ and $T$ with
$\br{S}\onto \br{T}$. We will show that
$\Hom(T,S)=0$. Assume otherwise.
 Lemma \ref{AsaiLemma} says that any non-zero map from $T$ to $S$ must be surjective. 
But there is also a non-zero map from $S$ to $T$, and composing these we get a noninvertible endomorphism of $T$, contrary to the hypothesis that $T$ is a
brick.

It follows that $S$ is not in the torsion class consisting of modules filtered
by quotients of $T$, so we do not have $\br{T} \onto \br{S}$. This establishes
the order condition for $\ont$. It also follows that $T$ is not in the
torsion-free class of modules filtered by submodules of $S$, so we do not
have $\br{T}\into \br{S}$. This establishes the brick condition.
The proof of the order condition for $\int$ is dual to the order condition for $\ont$.

Finally, we verify the isomorphism.
Let $(\cT,\cF)$ be a torsion pair for $\Mod A$.  
Define $X=\Bricks(A)\cap \cT$ and $Y=\Bricks(A)\cap \cF$.  
As recalled above, $\cT$ consists of those modules which have no non-zero morphisms to any module in $\cF$. Since by the dual of Lemma \ref{bricks enough} any module in $\cF$ is filtered by submoduless of bricks in~$\cF$, we can also describe $\cT$ as consisting of those modules which have
no non-zero morphisms to any module in $Y$. We therefore have that $X={}^\perp Y$. Dually, $Y=X^\perp$.

Conversely, suppose $(X,Y)\in\Pairs(\t)$.
Define $\cT$ to consist of modules filtered by quotients of modules from $X$, and $\cF$ to consist of modules filtered by submodules of modules from $Y$.  
Clearly, there are no morphisms from any module in $\cT$ to any module in $\cF$.  
We now prove by induction on the dimension of $M$ that if $M$ is in $\mathcal T^\perp$, then $M\in \cF$.  
Let $B$ be minimal dimensional among modules which are both submodules and quotient modules of $M$.  $B$ is necessarily a brick.
Since $B$ is a submodule of $M$, we must have $B\in\cT^\perp$, so $B\in Y\subset \cF$.  
Let $K$ be the kernel of a surjection from $M$ to $B$.  
Now $K\in\mathcal T^\perp$, so by the induction hypothesis, $K\in \cF$.  
So $M$, which is the extension of $B$ and $K$, is also in $\cF$.  
The dual argument shows that $^\perp \cF=\cT$. 
\end{proof}

Combining Theorem~\ref{brick thm} with Theorem~\ref{2afs iff ws kappa} or with Theorem~\ref{FTFSDL1}, we obtain the following corollaries.

\begin{corollary}\label{brick cor}
The poset $\tors(A)$ is a well separated $\kappa$-lattice.
\end{corollary}

\begin{corollary}\label{fin brick cor}
If $\tors(A)$ is finite, then it is a semidistributive lattice.
\end{corollary}

Corollary~\ref{fin brick cor} recovers \cite[Theorem 4.5]{GM}. This is a special case of \cite[Theorem 1.3]{DIRRT}, which says that $\tors(A)$ is completely semidistributive, without the hypothesis that $\tors(A)$ is finite.
However, Corollary~\ref{brick cor} accomplishes something different from \cite[Theorem~1.3]{DIRRT}:  Recall that Examples~\ref{JIrr empty} and~\ref{obnoxious} show that there are no implications between being a well separated $\kappa$-lattice and being a completely semidistributive lattice.

\section*{Acknowledgements} H.T. would like to thank Laurent Demonet and
Osamu Iyama for their hospitality at Nagoya University and helpful
conversations.
D.E.S. would like to thank the attendees of the Maurice Auslander Distinguished Lectures in 2019 for their helpful remarks.
All the authors would like to thank the anonymous referees for their careful reading and valuable suggestions.

\end{document}